\documentclass[aos, preprint]{imsart}
\RequirePackage{amsthm,amsmath,amsfonts,amssymb}
\usepackage{fix-cm}
\DeclareFontShape{T1}{ptm}{m}{scit}{<->ssub*ptm/m/sc}{}
\DeclareFontShape{T1}{ptm}{bx}{scit}{<->ssub*ptm/bx/sc}{}
\usepackage{float}
\usepackage{silence}
\RequirePackage[scr=esstix]{mathalfa}
\RequirePackage[numbers,sort&compress]{natbib}
\RequirePackage[colorlinks,citecolor=blue,urlcolor=blue]{hyperref}
\RequirePackage{graphicx}
\usepackage{subcaption}
\usepackage[capitalize]{cleveref}
\usepackage{crossreftools}

\crefname{proofcntr}{proof}{proofs}
\Crefname{proofcntr}{Proof}{Proofs}

\usepackage[ruled, vlined, linesnumbered]{algorithm2e}

\startlocaldefs
\theoremstyle{plain}
\newtheorem{theorem}{Theorem}[section]
\newtheorem{lem}[theorem]{Lemma}
\newtheorem{prop}[theorem]{Proposition}
\newtheorem{cor}[theorem]{Corollary}
\theoremstyle{remark}
\newtheorem{defn}[theorem]{Definition}
\newtheorem{example}{Example}[section]
\newtheorem{rem}[example]{Remark}
\endlocaldefs

\usepackage{mathtools}

\DeclareMathOperator*{\argmin}{arg\,min}

\begin{document}

\begin{frontmatter}
\title{ Principal curves in metric spaces and the space of probability measures}
\runtitle{Principal curves in metric spaces and W2 space}

\begin{aug}
\author[A]{\fnms{Andrew}~\snm{Warren}\ead[label=e1]{awarren@math.ubc.ca}\orcid{0000-0001-9369-9620}}
\author[A]{\fnms{Anton}~\snm{Afanassiev}\ead[label=e2]{anton.a@math.ubc.ca}\orcid{0000-0003-2387-8783}}
\author[A,B]{\fnms{Forest}~\snm{Kobayashi}\ead[label=e3]{forest.kobayashi@utah.edu}\orcid{0000-0002-0625-8243}}
\author[A]{\fnms{Young-Heon}~\snm{Kim}\ead[label=e4]{yhkim@math.ubc.ca}\orcid{0000-0001-6920-603X}}
\and
\author[A]{\fnms{Geoffrey}~\snm{Schiebinger}\ead[label=e5]{geoff@math.ubc.ca}\orcid{0000-0002-8290-7997}}
\address[A]{Department of Mathematics,
University of British Columbia\printead[presep={,\ }]{e1,e2,e4,e5}}
\address[B]{Department of Mathematics,
The University of Utah\printead[presep={,\ }]{e3}}
\end{aug}

\begin{abstract}

We introduce principal curves in Wasserstein space, and in general compact metric spaces.
Our motivation for the Wasserstein case comes from optimal-transport-based trajectory inference, where a developing population of cells traces out a curve in Wasserstein space.
Our framework enables new experimental procedures for collecting high-density time-courses of developing populations of cells:
time-points can be processed in parallel (making it easier to collect more time-points). However, then the time of collection is unknown, and must be recovered by solving a {\em seriation problem} (or one-dimensional manifold learning problem).

We propose an estimator based on Wasserstein principal curves, and prove it is consistent for recovering a curve of probability measures in Wasserstein space from empirical samples. This consistency theorem is obtained via a series of results regarding principal curves in compact metric spaces. In particular, we establish the validity of certain numerical discretization schemes for principal curves, which is a new result even in the Euclidean setting.

\end{abstract}



\end{frontmatter}
\tableofcontents{}

\section{Introduction}
\label{sec:intro}

Principal components analysis (PCA) is one of the most basic and widely used tools in exploratory data analysis and unsupervised learning. Principal curves~\cite{hastie1989} provide a natural, nonlinear generalization of principal components (Figure \ref{fig:intro}).

Motivated by an application to single cell RNA-sequencing (scRNA-seq), we develop a theory of principal curves in general compact metric spaces.
For this application, we consider principal curves in the space of probability measures over a compact space of cell states (e.g. gene expression space, which can be modeled with the $d$-dimensional simplex). This space of probability measures over cell states is itself a compact metric space when equipped with the Wasserstein metric, given by
optimal transport theory \cite{ambrosio2013user, villani2008optimal}.
Indeed, a population of cells can be thought of as a probability distribution over cell states, and as the population changes over time (e.g. in embryonic development), it traces out a curve in the space of probability distributions~\cite{lavenant2024towards,schiebinger2021opinion}.

Biologists collect empirical distributions along this developmental curve by profiling with single-cell RNA-sequencing, which provides high-dimensional measurements of cell states for thousands to millions of cells~\cite{schiebinger2019optimal,lavenant2024towards,schiebinger2021opinion}.
However, it is currently practically difficult to collect large numbers of time-points along the curve, because each time-point is typically processed manually, in series (e.g. several separate mouse embryos are fertilized, and each is allowed to develop for a specific amount of time before cells are profiled with scRNA-seq).
The benefit of this careful, separate processing of each developmental time-point allows experimentalists to record the time-of-collection for each time-point.
However, it would be possible to process far more embryos if one could {\em infer} the time-of-collection, because this would allow embryos to be processed in parallel, e.g. profile an entire cup of fly embryos, each at a different stage of development -- note that the  cells from same embryo come from the same developmental time-point.
The resulting data would consist of a (unordered) set of empirical distributions, along a developmental curve.
By fitting a principal curve to these empirical distributions, one can order these embryos in developmental time.

More precisely,
let $\rho_{t}$ denote a continuous\footnote{Here continuity is with respect to convergence in distribution.} curve of probability measures on a compact domain, parameterized by time $t\in [0,1].$
Suppose we observe empirical distributions $\hat \rho_{t_1},\ldots,\hat \rho_{t_N}$, constructed from i.i.d. samples from $\rho_{t_n}$, where each $t_n$ is an i.i.d. random time in $[0,1]$, as illustrated in Figure \ref{fig:intro}. However, the temporal labels $t_1,\ldots,t_N$ are unknown, and our goal is to recover them. We can approach this problem by fitting a principal curve to infer $\rho_{t}$, in the space of probability measures, equipped with the \emph{Wasserstein distance }$W_{2}$ which metrizes convergence in distribution.

Inferring the curve $\rho_t$ can also be thought of as a measure-valued \emph{manifold learning }problem: we
want to learn a one-dimensional continuous manifold in the space of
probability measures equipped with the $W_{2}$ distance.

\subsection{Contributions and overview}

We introduce
the principal curve problem in general compact metric spaces, motivated by an application of principal curves in the Wasserstein space of probability measures. Given a compact metric space $(X,d)$ and a probability distribution $\Lambda$ on $X$ (called the {\em data distribution}),
we seek a curve $\gamma^\star : [0,1] \to X$, which satisfies the
following two criteria: the distribution $\Lambda$ is close to
$\gamma^\star$, and the curve $\gamma^\star$ is not ``too long''. We measure the fit of $\gamma$ to $\Lambda$ by the average squared distance of points to $\gamma$, namely
$$\text{Fit}(\Lambda,\gamma) := \int_{X}\inf_{t\in[0,1]}d^{2}(x,\gamma_{t})d\Lambda(x).$$
A principal curve $\gamma^\star$ passing through $\Lambda$ is then defined via
\begin{equation}
\underset{\gamma}{\text{minimize }} \quad
\text{Fit}(\Lambda, \gamma) + \beta\text{Length}(\gamma), \label{eq:ppc-metric-intro}
\end{equation}
where $\beta >0$ determines the level of length penalization. The minimization runs over all
sufficiently regular curves, in a sense we make precise in Section \ref{sec:basics}.

\begin{figure}[H]
  \centering
  \begin{subfigure}[t]{1.0\linewidth}
    \centering
    \includegraphics{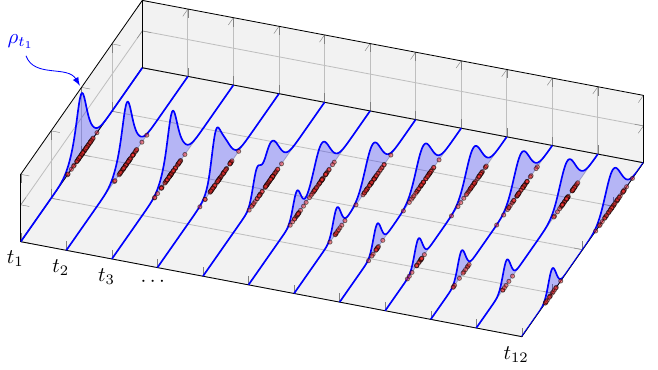}
    \caption{}
  \end{subfigure}
  \begin{subfigure}[t]{1.0\linewidth}
    \centering
    \includegraphics{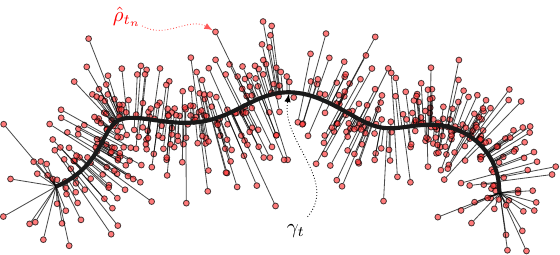}
    \caption{}
  \end{subfigure}
  \caption{(a) An illustration showing ${\color{blue} \rho_t}$ (blue)
    at evenly-spaced time samples $t_1, \ldots, t_N$, with the samples
    comprising $\color{red} \hat \rho_{t_n}$ shown in red.
      (b) An
    illustration of a principal curve $\gamma_t$ (black curve) ``fitting'' the
    empirical measures $\color{red} \hat \rho_{t_n}$, each of which is represented by a single red dot. Note that the curve
    achieves low average projection distance, while the curve itself
    is not ``too long''. The straight lines connecting individual data
    points to their projections along the curve represent length-minimizing geodesics in the Wasserstein space.
    }
    \label{fig:intro}
\end{figure}

One special case of interest is the Wasserstein space of probability measures over a compact space $V$:
$$X = \mathcal{P}(V), \quad \text{and} \quad d(\mu,\nu) = W_2(\mu,\nu).$$
Here, a natural choice for the data distribution $\Lambda$ is one concentrated around a curve of measures $\rho.$
In practice, one may observe data\footnote{See Section \ref{sec:Wasserstein} below for a more verbose description of our data-observation setup.} in the form of empirical distributions $\hat \rho_{t_1},\ldots,\hat \rho_{t_T}$ along the curve $\rho$, as described above in Section \ref{sec:intro}. In this case, the data distribution is the following empirical {measure}:
\begin{equation}
\hat{\Lambda}:=\frac{1}{N}\sum_{n=1}^N\delta_{\hat{\rho}_{t_n}}. \label{eq:double-empirical-measure-intro}
\end{equation}

In this finite-data setting, the principal curve problem~\eqref{eq:ppc-metric-intro} can be discretized by replacing $\Lambda$ with the empirical approximation $\hat{\Lambda}$, and discretizing the Length penalty, as follows:
\begin{equation}
\underset{\gamma_1,\ldots,\gamma_K \in \mathcal{P}(V)}{\text{minimize}} \quad \frac{1}{N}\sum_{n=1}^N\inf_{k\in[1,K]}W_2^{2}(\hat{\rho}_{t_n},\gamma_{k})+\beta\sum_{k=1}^{K-1}{W_2}(\gamma_{k},\gamma_{k+1}).\label{eq:ppc-w2-discrete-intro}
\end{equation}
In Section~\ref{subsec:discrete-to-continuum}, we describe a coupled
Lloyd's algorithm for this discretized objective. Implemented via existing Wasserstein barycenter solvers, we show it can generate reasonable results in practice.

Minimizers of the discrete problem (\ref{eq:ppc-w2-discrete-intro}) can be thought of as estimators for the true curve of measures $\rho$.
One of our main results is that these estimators are consistent.
\begin{theorem}[Consistency of Wasserstein principal curves]
\label{thm:main-intro}
Let $V$ be a compact, convex domain. Suppose we are given data from a ground truth curve $\rho_t:[0,1]\rightarrow \mathcal{P}(V)$ in the form of an empirical measure $\hat{\Lambda}$ as in equation (\ref{eq:double-empirical-measure-intro}), with $M$ samples for each empirical distribution $\hat \rho_{t_1},\ldots,\hat\rho_{t_N}$. Assuming that $\rho_t$ is injective and sufficiently regular\footnote{It suffices to assume $\rho_t$ is Lipschitz as a function from $[0,1]$ to $\mathcal{P}(V)$ equipped with the $W_2$ metric.}, then with probability 1,
\[
(\rho_t)_{t\in[0,1]}=\lim_{\beta\rightarrow 0}
\lim_{N,M,K \to \infty}\underset{\gamma_1,\ldots,\gamma_K}{\argmin}
\left[\frac{1}{N}\sum_{n=1}^N\inf_{k\in[1,K]}W_2^{2}(\hat{\rho}_{t_n},\gamma_{k})+\beta\sum_{k=1}^{K-1}{W_2}(\gamma_{k},\gamma_{k+1})\right].
\]
\end{theorem}

We make precise the sense in which we recover the ground truth in the limit in Section~\ref{sec:seriation}; but the idea is that in the limit, our minimizing $K$-tuple $\{\gamma^*_{k}\}_{k=1}^K$ from (\ref{eq:ppc-w2-discrete-intro}) converges to a curve $\gamma^*_t$ with the \emph{same range} as $\rho_t$. Then if $\rho_t$ is injective (i.e. does not intersect itself), this allows us to infer the ordering of the time labels along $\rho_t$ up to reversing the whole ordering.

In fact, we establish a more general version of Theorem \ref{thm:main-intro} which allows for measurement noise typically encountered in single cell RNA sequencing (scRNA-seq).
In practice, scRNA-seq provides an imperfect snap-shot of each cell's expression profile by sampling from the distribution of messenger RNAs within the cells.
After profiling many cells at various time-points along a developmental curve $\rho$, the resulting data would consist of {\em noisy} empirical distributions $\tilde \rho_{t_1}, \ldots, \tilde \rho_{t_N},$
defined as
\begin{equation}
  \tilde \rho_{t_n} = \frac{1}{M} \sum_{m=1}^{M}
  \delta_{\hat v_m},
  \label{eq:noisy_empirical}
\end{equation}
where $\hat v_m \sim \text{Multinomial}(v_m,R_m)$ is the noisy expression profile when we sequence $R_m$ \lq reads\rq~from a cell with expression profile $v_{m}\sim \rho_{t_n}$ ~\cite{kim2024optimal}.
A version of Theorem \ref{thm:main-intro} also holds with noisy empirical distributions $\tilde \rho_{t_n}$, as long as $R_m \to \infty$ uniformly.

We derive Theorem \ref{thm:main-intro} by way of a series of results for the more general problem (\ref{eq:ppc-metric-intro}) of principal curves on compact metric spaces, which we now summarize.
In Section \ref{sec:basics}, we  prove existence of minimizers for (\ref{eq:ppc-metric-intro})  (Proposition \ref{prop:mins-exist}),  as well as stability of minimizers
with respect to the underlying distribution $\Lambda$ (Proposition \ref{prop:gamma-m-gamma}). We then establish the consistency
of a discretized variational problem generalizing (\ref{eq:ppc-w2-discrete-intro}), meaning minimizers converge to minimizers of the continuum problem (\ref{eq:ppc-metric-intro}) (Theorem \ref{thm:discrete-to-continuum}). This discrete problem enjoys a
\textquotedbl coupled Lloyd's algorithm\textquotedbl{} type numerical
scheme  (Algorithm \ref{alg:coupled-lloyd}); in the case of measure-valued data where $X=\mathcal{P}(V)$ and $d=W_2$, this scheme can be feasibly implemented via existing Wasserstein
barycenter solvers. We also consider a variant numerical scheme which incorporates nonlocal adaptive kernel smoothing, and which offers better performance in our experiments; we likewise prove its consistency with respect to the continuum problem (\ref{eq:ppc-metric-intro}), in Proposition \ref{prop:discrete-nonlocal-to-continuum}.

In Section \ref{sec:Wasserstein}, we consider special features of problem (\ref{eq:ppc-metric-intro}) in the case of where $X=\mathcal{P}(V)$ and $d=W_2$, and so $\Lambda$ is a distribution over probability distributions. In particular, we establish an \emph{iterated Glivenko-Cantelli theorem} for doubly empirical measures over compact metric spaces (Theorem \ref{thm:iterated-glivenko-cantelli}). This theorem establishes the convergence of measures such as the $\hat{\Lambda}$ from (\ref{eq:double-empirical-measure-intro}) to the limiting distribution over distributions $\Lambda$, in a sense which allows us to apply the discrete-to-continuum convergence results from Section \ref{sec:basics}.
We further extend this to a triple-iterated Glivenko-Cantelli theorem to handle the practical setting of finite  reads in single-cell RNA-sequencing~\cite{kim2024optimal, zhang2020determining}.
This  triple-iterated Glivenko-Cantelli result is stated in Proposition~\ref{prop:finite-reads}.

In Section \ref{sec:seriation}, we apply principal curves to the problems of inferring curves and missing
temporal labels from data, mentioned above. {Theorem \ref{thm:seriation-consistency},  our third main theorem, establishes that in the case where the support of $\Lambda$ coincides with the range of an injective curve $\rho_t$ in a compact metric space $X$, as we send the regularization parameter $\beta$ to 0, we recover $\rho_t$ (up to monotone or reverse-monotone time reparametrization).} We show in Proposition \ref{prop:pseudotimes-consistency} that this means we can recover the \emph{ordering} of the $\rho_t$'s (up to total reversal), and so our principal curves machinery can be used as a \emph{seriation} algorithm.

We test the performance of our numerical scheme in Section \ref{sec:experiments} by fitting principal curves to a real mouse time series \cite{mouse_time_series} and comparing the accuracy of the ordering derived from the principal curve to that of existing seriation algorithms. We find that our approach to principal curves is competitive with existing methods, even while simultaneously providing an estimate of the latent one-dimensional structure of the data, something not provided by widely used seriation methods like that of \cite{atkins_spectral_1998}.
We anticipate that this framework could be useful for generating high-density time-courses with single-embryo, single-cell RNA-sequencing~\cite{mouse_time_series, mittnenzweig2021single}.

\subsection{Related work}

Our study of principal curves in the space
of probability measures is a contribution to a growing body of work
which transposes widely used learning algorithms to the setting of
measure-valued data, including: regression \cite{karimi2021regression, chen2023wasserstein, pmlr-v31-poczos13a, ghodrati2022distribution, lavenant2024towards, chizat2022trajectory, schiebinger2019optimal}, spline interpolation \cite{benamou2019second, chen2018measure, chewi2021fast, justiniano2023approximation}, clustering \cite{zhuang2022wasserstein, rao2020wasserstein, verdinelli2019hybrid},  and manifold learning \cite{hamm2023manifold}. In particular, there has been some work on principal component analysis in Wasserstein space \cite{bigot2017geodesic, seguy2015principal, cazelles2018geodesic, karimi2020statistical, wang2013linear}. However, principal curves have not yet been studied in Wasserstein space.

Moreover, some elements of
our results are new even in the setting
of Euclidean space. In particular,  the theoretical guarantees for
numerical schemes for principal curves provided in Section \ref{sec:basics} and Appendix \ref{sec:extensions} appear to be new; similar schemes have been used widely for decades (including in the R princurve package), but to our knowledge have only been
justified heuristically.
Accordingly, conducting our analysis in a general metric space setting allows us to handle the cases of Euclidean- and measure-valued data simultaneously; moreover, it also covers more abstract principal curves which have been considered previously, such as those on compact Riemannian manifolds \cite{hauberg2015principal} and compact subsets of Banach spaces~\cite{smola2001regularized}.

\paragraph*{Comparison to existing work on pricipal curves}

The study of principal curves was inaugurated in the 1980s by Hastie and Steutzle \cite{hastie1984, hastie1989}, who defined a principal curve for a probability distribution $\Lambda$ on $\mathbb{R}^d$ to be a critical point of the functional
\[
\mathrm{PC}(\Lambda)(\gamma):=\int_{\mathbb{R}^d}\inf_{t\in[0,1]}d^2(x,\gamma_t)d\Lambda(x)
\]
where the argument $\gamma_t:[0,1]\rightarrow \mathbb{R}^d$ ranges
over infinitely differentiable injective curves. Hastie and Steutzle
offer an appealing interpretation of critical points of
$\mathrm{PC}(\Lambda)$ as curves which ``locally pass through the
middle'' of the distribution $\Lambda$. Nonetheless the functional $\mathrm{PC}(\Lambda)$ has some undesirable features. Indeed, $\mathrm{PC}(\Lambda)$ achieves a value of zero if one takes $\gamma_t$ to be a space-filling curve on the support of $\Lambda$; while space-filling curves are excluded from the domain of $\mathrm{PC}(\Lambda)$ by definition, it remains the case that approximate minimizers will be approximate space-filling curves, and these do not provide a desired summary of the data distribution in typical applications. One might hope that some non-minimal critical points of $\mathrm{PC}(\Lambda)$ are better behaved, but here too there are obstacles: \cite{duchamp1996extremal} observes that every critical point of $\mathrm{PC}(\Lambda)$ is a saddle point (rather than a local minimum), which prevents the use of cross-validation; see also discussion in \cite{gerber2013principal}. Making a very similar observation but with different terminology, \cite{smola2001regularized} notes that computing critical points of $\mathrm{PC}(\Lambda)$ is an \emph{ill-posed problem}, meaning that these critical points are unstable with respect to small perturbations of the distribution $\Lambda$.

Motivated by these technical problems, numerous alternate formulations
have been proposed, including density ridge estimation
\cite{ozertem2011locally, genovese2012geometry,
  genovese2014nonparametric, chen2015asymptotic}, conditional
expectation optimization \cite{gerber2013principal}, local principal
components \cite{delicado2001another},  spline-based methods
\cite{tibshirani1992principal}, and {\em penalized principal curves}:
a wide range of approaches close in spirit to ours, whereby the
functional $\mathrm{PC}(\Lambda)$ is modified with some kind of
\emph{penalty} on the curve $\gamma$. Various approaches to penalized
principal curves include: a hard constraint on the length of $\gamma$
\cite{kegl2000, delattre2020}, or a soft penalty on the length of
$\gamma$~\cite{smola2001regularized,lu2015regularity}, or some
penalization of other attributes like ``curvature'' \cite{Biau2011Oct, lu2021average, sandilya2002principal,kobayashi2024monge}.

In this article, we employ the ``soft length penalty'' variant of principal curves, because it can be directly translated to the general setting of compact metric spaces, and hence the Wasserstein space of probability measures.
Moreover, the defining variational problem is stable with respect to
perturbations of the data distribution, as we show in Proposition
\ref{prop:gamma-m-gamma}, and  the problem can be naturally
discretized, and attacked numerically, even in metric spaces which
have no differentiable structure and no analog of the Lebesgue measure
(so that there is no meaningful notion of ``probability density'').

\paragraph*{Comparison to related work on seriation}

Seriation, the problem of ordering data given pairwise comparisons,
is related to our principal curve problem.
Early theoretical treatments of the seriation problem are due to Robinson \cite{robinson1951method} and Kendall \cite{kendall1963statistical, kendall1969incidence}, and are closely related to the traveling salesman problem \cite{gertzen2012flinders}. There is another branch of seriation literature focused on inferring permutations from pairwise comparisons between objects~\cite{flammarion2019optimal, shah2016stochastically}. Spectral seriation computes eigenvectors of the Laplacian constructed from a kernel matrix of pairwise comparisons \cite{atkins_spectral_1998,giraud2023localization, janssen2022reconstruction, natik2021consistencyspectralseriation,zendehboodi2023exploring}.
In our principal curve formulation, by contrast, we directly observe the objects (e.g. the empirical distributions), and we explicitly search for a curve embedded in a metric space.

\paragraph*{Comparison to related work on single-cell trajectory inference}

The problem of reconstructing a curve of measures $\rho_t$ from empirical samples $\hat \rho_{t_n}$ has also been studied recently in the context of {\em trajectory inference}~\cite{lavenant2024towards,chizat2022trajectory}.
In trajectory inference, $\rho_t$ represents the curve of marginals of a stochastic process, and the goal is to recover the underlying trajectories (or law on paths). Lavenant et al.~\cite{lavenant2024towards} showed that a convex approach, based on optimal transport, can recover a stochastic differential equation from marginal samples, as long as the vector field for the drift has zero curl ~\cite{lavenant2024towards}.
Assuming that the drift is conservative implies that the curve of marginals uniquely determines the law on paths of the SDE.
The primary difference with the setup here is that now we are not given time-labels $t_n$ for the empirical marginals $\hat \rho_{t_n}$ along the curve (Figure \ref{fig:intro}). Indeed, a corollary of our work is a consistent approach for trajectory inference, without time-labels.

\section{Principal curves in metric spaces}
\label{sec:basics}

We first give definitions of some key objects used in the setup of our inference problem.

Let $(X,d)$ be a compact metric space. By \emph{curve} in $X$, we
mean a continuous function $\gamma_{(\cdot)}:[0,1]\rightarrow X$.
The \emph{metric speed} of a curve $\gamma$ at time $t$ is denoted
$|\dot{\gamma}_{t}|$ and is given by the following limit:
\[
|\dot{\gamma}_{t}|:=\lim_{s\rightarrow t}\frac{d(\gamma_{s},\gamma_{t})}{|s-t|}.
\]
A curve $\gamma$ is said to be \emph{absolutely continuous} (AC)
provided that $|\dot{\gamma}_{t}|$ exists for Lebesgue-almost all
$t\in[0,1]$, and that $|\dot{\gamma}_{t}|$ is an integrable function in time. In this case we define the \emph{length} of $\gamma$
as follows:
\[
\text{Length}_{d}(\gamma):=\int_{0}^{1}|\dot{\gamma}_{t}|dt.
\]
We note that the length of $\gamma$ is independent of time-reparametrization,
and that a sufficient condition for $\text{Length}_{d}(\gamma)$ to
be finite is that $t\mapsto\gamma_{t}$ is Lipschitz, in which case
$\text{Length}_{d}(\gamma)\leq\text{Lip}_{d}(\gamma)$. We write $\text{AC}([0,1];X)$
to denote the set of AC curves into $X$.

A metric space $(X,d)$ is said to be a \emph{length space }if, for
all $x,y\in X$, the distance $d(x,y)$ coincides with the infimum
of the length of all AC curves connecting $x$ to $y$, namely
\[
d(x,y)=\inf_{\gamma\in\text{AC}([0,1];X):\gamma_{0}=x,\gamma_{1}=y}\int_{0}^{1}|\dot{\gamma}_{t}|dt.
\]
Likewise we say that $(X,d)$ is a \emph{geodesic metric space} if
this infimum is attained. In this case a \emph{geodesic curve} $\gamma$
from $x$ to $y$ is one where $d(x,y)=\int_{0}^{1}|\dot{\gamma}_t|dt$,
and $\gamma_{0}=x$ and $\gamma_{1}=y$. A \emph{constant-speed} geodesic is one which also satisfies $d(\gamma_t,\gamma_s)=|t-s|\cdot d(x,y)$ for all $t,s\in[0,1]$.

For further background on curves in metric spaces, we refer the reader to \cite{ambrosio2004topics, burago2001course}. We collect a number of basic results concerning curves in metric spaces in Appendix \ref{appn:cptness}.

We also recall the definition of the \emph{weak* topology} for probability measures on a compact metric space $X$: here, a sequence $(\Lambda_n)_{n\in \mathbb{N}}$ of elements of $\mathcal{P}(X)$ is said to converge weak*ly (a.k.a. converge in distribution) to a probability measure $\Lambda$ provided that, for every continuous $\varphi : X \rightarrow \mathbb{R}$, $$\int_X \varphi d\Lambda_n \underset{n\rightarrow\infty}{\longrightarrow} \int_X \varphi d\Lambda.$$ In this case we write $\Lambda_n \rightharpoonup^* \Lambda$. By Prokorov's theorem, $\mathcal{P}(X)$ is itself a compact metric space when equipped with the weak* topology.

We first consider existence of  minimizers for the penalized version of the principal curves functional
\[
\mathrm{PPC}(\Lambda;\beta)(\gamma):=\int_X d^2(x,\Gamma)d\Lambda(x) + \beta \mathrm{Length}(\gamma)
\]
where $\gamma_{(\cdot)}\in AC([0,1];X)$, $\Lambda\in\mathcal{P}(X)$,
and $\Gamma$ denotes the range of $\gamma$. We shall often drop the
$\beta$ argument and simply write $\mathrm{PPC}(\Lambda)$ for brevity.
We call minimizers of $\mathrm{PPC}(\Lambda;\beta)$ ``principal curves
for $\Lambda$'' with the understanding that there is implicit dependence on the choice of penalty $\beta>0$. In this section as well as in Section \ref{sec:Wasserstein}, we allow $\Lambda$ to be a general probability measure; we then specialize to the case where $\Lambda$ is induced by a ground truth curve $\rho_t$ in Section \ref{sec:seriation}.

\begin{prop}[Existence of principal curves]
\label{prop:mins-exist}
Let $(X,d)$ be a compact metric space. Let $\Lambda\in\mathcal{P}(X)$ and let $\beta>0$.
Then there exist minimizers for $\mathrm{PPC}(\Lambda;\beta)$.
\end{prop}

The proof is deferred to Appendix \ref{subsec:proofs-for-basics}. We note that the argument is close to that of \cite[Lem. 2.2]{lu2016average}, which establishes existence in the case where $X=\mathbb{R}^d$.

A natural next question is whether we can guarantee solutions are
unique. In many cases the answer is no, and we give an elementary
example for $\mathbb R^2$ in a moment in Proposition
\ref{prop:r2-nonuniqueness}. The idea is that when $\Lambda$ has nontrivial
symmetries and an optimizer $\gamma$ of $\mathrm{PPC}(\Lambda)$ is
known to be injective, we may simply apply said symmetries to $\gamma$
to obtain a new optimizer.
While the counterexample to uniqueness we provide here is quite specific, we also note that the data-fitting term $\int_X d^2(x,\Gamma)d\Lambda(x)$ is neither concave nor convex in $\gamma$ even on Euclidean domains \cite[§ 2.2.1]{kobayashi2024monge}, so one should not expect the energy landscape of the functional $\mathrm{PPC}(\Lambda)$ to be simple.

This failure of uniqueness means that, for a given data distribution $\Lambda$ (and length penalty $\beta>0$), we must be willing to accept \emph{each} principal curve for $\Lambda$ as an estimate of $\Lambda$'s implicit one-dimensional structure. Whether this makes sense surely depends on $\Lambda$ itself as well as the domain application. For example, the consistency results of Section \ref{sec:seriation} apply in the limiting case where $\Lambda$ itself has one-dimensional support inside $X$.

\begin{prop}[Non-uniqueness of principal curves]
\label{prop:r2-nonuniqueness}
There exists a compact metric space $(X,d)$ and probability measure
$\Lambda\in\mathcal{P}(X)$ for which minimizers of $\mathrm{PPC}(\Lambda)$
are not unique. In particular we can take $X$ to be the unit ball
in $\mathbb{R}^{2}$ and $\Lambda$ the uniform distribution on an equilateral
triangle.
\end{prop}

The details of the proof are given in Appendix \ref{subsec:proofs-for-basics}.

Next, we investigate questions of \emph{stability}, that is, the dependency of principal curves on
the choice of measure $\Lambda$. As we do not have uniqueness,
  our stability result is for the set of optimizers as a
  whole.

\begin{prop}[Stability of principal curves with respect to data distribution]
\label{prop:gamma-m-gamma}
  Let $(X,d)$ be a compact metric space. Let $(\Lambda_{N})_{N \in
    \mathbb{N}}$ denote a sequence in $\mathcal{P}(X)$, such that
  $\Lambda_{N} \rightharpoonup^{*} \Lambda$. Then minimizers of
  $\mathrm{PPC}(\Lambda_{N})$ converge to minimizers of
  $\mathrm{PPC}(\Lambda)$ in the following sense.

  If
  $\gamma^{*N}\in\text{AC}([0,1],X)$ is a minimizer of
  $\mathrm{PPC}(\Lambda_{N})$ which has been reparametrized to have
  constant speed, then there exists a minimizer $\gamma^{*}$ of
  $\mathrm{PPC}(\Lambda)$ such that up to passage to a subsequence,
  \[
    \gamma_{t}^{*N}\rightarrow\gamma_{t}^{*}\text{ uniformly in }t.
  \]
  Moreover, given any subsequence of minimizers $\gamma^{*N}$ of $\mathrm{PPC}(\Lambda_{N})$
which converges uniformly in $t$ to some $\gamma^{*}\in AC([0,1];X)$,
it holds that $\gamma^{*}$ is a minimizer of $\mathrm{PPC}(\Lambda)$.
\end{prop}

See Appendix \ref{subsec:proofs-for-basics} for the proof.  We then
obtain the following by way of Varadarajan's version \cite{varadarajan1958} of the Glivenko-Cantelli theorem:
\begin{cor}
  Let $\Lambda\in\mathcal{P}(X)$, and let $(\Lambda_{N})_{N \in
    \mathbb{N}}$ be a sequence of empirical measures for $\Lambda$.
  Then with probability 1, it holds that minimizers of
  $\mathrm{PPC}(\Lambda_{N})$ converge to minimizers of
  $\mathrm{PPC}(\Lambda)$ in the sense of the previous proposition.
\end{cor}

\subsection{Discretization and algorithm}
\label{subsec:discrete-to-continuum}

In this subsection, we first formally discretize the principal curves functional $\mathrm{PPC}(\Lambda)$, and describe a Lloyd-type algorithm to minimize this discretization. Then, we provide a rigorous consistency result for our discretization scheme.

We first give the following heuristic computation as motivation. Let $x_{1},x_{2},\ldots,x_{N}$ be $N$ i.i.d. random elements of
$X$ drawn according to $\Lambda$, and let $\gamma\in AC([0,1];X)$.
Fix also timepoints $0=t_{1}<\ldots<t_{k}<\ldots<t_{K}=1$. Now, for
any $1\leq n\leq N$, observe that if $t_{k}(n)$ is the index of
the $\gamma_{t_{k}}$ which is closest to $x_{n}$ amongst all $\{\gamma_{t_{k}}\}_{k=1}^{K}$
(where we pick the lowest index in the event of a tie), then it holds
that $d(x_{n},\gamma_{t_{k}(n)})\approx d(x_{n},\Gamma)$, with equality in the limit where $K \to\infty$. According
to Monte Carlo integration, we expect heuristically that
\[
\int d^{2}(x,\Gamma)d\Lambda(x)\approx\frac{1}{N}\sum_{n=1}^{N}d^{2}(x_{n},\gamma_{t_{k}(n)}).
\]
Now, let
\[
I_{k}:=\{n\in N:x_{n}\text{ projects onto }\gamma_{t_{k}}\}.
\]
(Where again we pick the lowest index in the case of a tie, for the
``projection'' onto the $\gamma_{t_{k}}$'s.) The sets $I_k$ can be
thought of as ``Voronoi cells'' for the points $x_n$ with respect to $\{\gamma_{t_{k}}\}_{k=1}^{K}$. Then we can rewrite
\[
\sum_{n=1}^{N}d^{2}(x_{n},\gamma_{t_{k}(n)})=\sum_{k=1}^{K}\sum_{n\in I_{k}}d^{2}(x_{n},\gamma_{t_{k}}).
\]
At the same time, we can approximate $\text{Length}(\gamma)\approx\sum_{k=1}^{K-1}d(\gamma_{t_{k}},\gamma_{t_{k+1}}).$
Altogether,
\[
\int_{X}d^{2}(x,\Gamma)d\Lambda(x)+\beta\text{Length}(\gamma)\approx\sum_{k=1}^{K}\sum_{n\in I_{k}}\frac{d^{2}(x_{n},\gamma_{t_{k}})}{N}+\beta\sum_{k=1}^{K-1}d(\gamma_{t_{k}},\gamma_{t_{k+1}}).
\]
Note again that the ``Voronoi cells'' $I_{k}$ depend on $\gamma_{t_{k}}$ in this
expression.

This discretization is useful to us for two reasons. Firstly, it turns
out that we can prove that minimizers of the discrete objective
$\sum_{k=1}^{K}\sum_{n\in
  I_{k}}\frac{d^{2}(x_{n},\gamma_{t_{k}})}{N}+\beta\sum_{k=1}^{K-1}d(\gamma_{t_{k}},\gamma_{t_{k+1}})$
converge (in roughly the same sense as in Proposition
\ref{prop:gamma-m-gamma}) to minimizers of $\mathrm{PPC}(\Lambda)$ as
$N,K\rightarrow \infty$. Indeed we prove such a discrete-to-continuum
convergence result later in this section, see Theorem \ref{thm:discrete-to-continuum}.

Second, the discrete objective admits a descent algorithm which is morally similar to Lloyd's algorithm for k-means clustering, except that the means are ``coupled'' to each other via the global length penalty $\beta\sum_{k=1}^{K-1}d(\gamma_{t_{k}},\gamma_{t_{k+1}})$.  While (just as for usual k-means clustering) we are unable to prove global convergence due to the non-convexity of the objective, we nonetheless observe in our experiments that such a descent procedure works well in practice.

To wit, we now give a prose description of a ``coupled Lloyd's algorithm'' for the discrete objective. A restatement of the algorithm in pseudocode is given below as Algorithm \ref{alg:coupled-lloyd}; a single loop of the algorithm is illustrated below in Figure \ref{fig:coupled-lloyds}.

\begin{enumerate}
\item Initialize the ``knots'' $\{\gamma_{k}\}_{k=1}^{K}$. (Line 1 in Algorithm \ref{alg:coupled-lloyd})
\item Given some predetermined threshold $\varepsilon>0$, repeat steps 3-5 below in a loop, until the value of the discrete objective $\sum_{k=1}^{K}\sum_{n\in
  I_{k}}\frac{d^{2}(x_{n},\gamma_{k})}{N}+\beta\sum_{k=1}^{K-1}d(\gamma_{k},\gamma_{k+1})$ drops by less than
 $\varepsilon>0$. (Lines 2 and 6)
\item Permute the indices for the knots $\{\gamma_{k}\}_{k=1}^{K}$ to
minimize the total length $\sum_{k=1}^{K-1}d(\gamma_{k+1},\gamma_{k})$. This amounts to solving the traveling salesman problem with respect to the matrix of pairwise distances between the $\gamma_k$'s. (Line 3)
\item Given $\{\gamma_{k}\}_{k=1}^{K}$, compute the sets $I_{k}:=\{n:\gamma_{k}\text{ is the closest knot to }x_{n}\}$.
Tiebreak lexicographically if necessary. (Line 4)
\item Given the $I_{k}$'s from step 4, compute a new set of $\gamma_{k}$'s
by minimizing
$\sum_{k=1}^{K}\sum_{n\in I_{k}}\frac{d^{2}(x_{n},\gamma_{k})}{N}+\beta\sum_{k=1}^{K-1}d(\gamma_{k+1},\gamma_{k}).$ (line 5)
\end{enumerate}
A version of this algorithm on Euclidean domains has been proposed in \cite{kirov2017multiple}; our version is for general compact metric spaces for which pairwise distances and barycenters can be effectively computed.

\begin{algorithm}[H]
  \SetKwComment{Comment}{/* }{ */}
  \caption{Coupled Lloyd's Algorithm for Principal Curves}\label{alg:coupled-lloyd}
  \SetKwInOut{Input}{Input}
  \Input{$\textrm{data }\{x_n\}_{n=1}^N, \textrm{parameters } \beta>0, \varepsilon>0$}

  \SetKwFunction{Initialize}{initialize\_knots}

  $\{\gamma_k\}_{k=1}^K \gets \Initialize{}$\;

  \Repeat{$\varepsilon\textrm{-convergence}$}{
   $\{\gamma_k\}_{k=1}^K \gets
    \texttt{TSP\_ordering}(\{\gamma_k\}_{k=1}^K)$
    \Comment*[r]{min-length ordering}
   $\{I_k\}_{k=1}^K \gets
    \texttt{compute\_Voronoi\_cells}(\{\gamma_k\}_{k=1}^K)$\;
   $\{\gamma_k\}_{k=1}^K \gets
    \argmin\limits_{\{\gamma'_k\}_{k=1}^K }
    \sum_{k=1}^K \sum_{n\in I_k} \frac{d^2(x_n, \gamma_k')}{N} + \beta
    \sum_{k=1}^{K-1} d(\gamma_{k+1}', \gamma_k')$\;
  }

  \KwResult{$\{\gamma_k\}_{k=1}^K$ \Comment*[r]{The updated output knots}}

\end{algorithm}

Let us make several remarks regarding the details of Algorithm \ref{alg:coupled-lloyd}.

\begin{rem}
The choice of initialization of the knots $\{\gamma_k\}_{k=1}^K$
matters for the performance of Algorithm \ref{alg:coupled-lloyd} in
practice, just like for the usual Lloyd's algorithm for k-means
clustering. This is because the objective is nonconvex for general
data distributions. We suggest a ``k-means++'' style initialization where the $\{\gamma_k\}_{k=1}^K$'s are first taken to be $K$ randomly chosen points among the data $\{x_n\}_{n=1}^N$, as this initialization has performed well for us in practice. With some datasets, a warm start approach can also be useful (see Appendix \ref{appn:extra-experiments}).
\end{rem}

\begin{rem}
Step 2 of Algorithm \ref{alg:coupled-lloyd}, where one solves the traveling salesman problem as a subroutine, is feasible for moderately large $K$ (say $K\sim 1000$) thanks to highly efficient TSP solvers which are now available, such as \cite{applegate2006traveling_concorde}.
\end{rem}

\begin{rem}
The optimization subroutine in Step 4 of Algorithm \ref{alg:coupled-lloyd} is somewhat complicated by the fact that the data-fitting term is quadratic while the length penalty is linear. Nonetheless, optimization of objectives of this form can be done using modern convex optimization methods (e.g split Bregman projection), see discussion in \cite{kirov2017multiple, tibshirani2005sparsity}.
\end{rem}

\begin{rem}
In the case where $K=N$, Algorithm \ref{alg:coupled-lloyd} is a relaxation of the traveling salesman problem. Indeed, when $K=N$ it is possible to make the data-fitting term $\sum_{k=1}^{K}\sum_{n\in
  I_{k}}\frac{d^{2}(x_{n},\gamma_{k})}{N}$ equal zero by placing a $\gamma_k$ directly atop each data point $x_n$. Then, minimizing the length penalty (subject to enforcing that the data-fitting term is identically zero) amounts to choosing an ordering of the $\gamma_k$'s of minimal total length according to the metric $d$. We mention in particular the work \cite{polak2007lazy} which studies the properties this relaxation of the travelling salesman problem in $\mathbb{R}^2$.
\end{rem}

{
  \def\acaption{A local view of the situation in the discretized case.
    Here, the knots $\color{blue} \{\gamma_k\}_{k=1}^K$ are plotted
    with geodesic interpolations between adjacent points, ordered using a TSP solver. Also shown: The data points
    $\color{red} \{x_n\}_{n=1}^N$ and the Voronoi cells (dotted
    black lines).}
  \def\bcaption{An illustration of how the update step works. The
    position of each knot $\color{blue} \gamma_k$ affects the
    overall objective value via: (1) the average distance
    from $\color{blue} \gamma_k$ to the data points ${\color{red}
      x_n} \in I_k$, and (2) the distance from $\color{blue}
    \gamma_{k}$ to the adjacent knots $\color{blue} \gamma_{k-1},
    \gamma_{k+1}$. In that sense, at the update step each
    $\color{blue} \gamma_k$ is ``pulled'' toward the points of $I_k$
    (vectors drawn in gray) and $\{{\color{blue} \gamma_{k-1},
      \gamma_{k+1}}\}$ (vectors draw in blue).}
  \def\ccaption{Moving the knot points according a weighted sum of the
    vectors in \cref{fig:vc-vectors}. In accordance with the
    weightings of the terms in the discretized functional in Step
    5, each vector pointing to a ${\color{red} x_n} \in I_k$ is
    weighted by $\frac{d^{2}({\color{red} x_n},\textcolor{blue}{\gamma_{k}})}{N}$, while the vectors pointing to $\color{blue}
    \gamma_{k-1}, \gamma_{k+1}$ are weighted by $\frac{\beta}{2} d({\color{blue}
      \gamma_k}, {\color{blue} \gamma_{k-1}})$ and $\frac{\beta}{2} d({\color{blue}
      \gamma_k}, {\color{blue} \gamma_{k+1}})$ respectively
    (the factor of two arising because each $d({\color{blue}
      \gamma_k}, {\color{blue} \gamma_{k+1}})$ is split into one
    vector on $\color{blue} \gamma_k$ and one on $\color{blue}
    \gamma_{k+1}$).}
  \def\dcaption{The updated $\color{blue}\{ \gamma_k \}_{k=1}^K$ with
    the associated updated Voronoi cells.}
  \def\vcscale{.85}
  \def\vcwidth{.49\linewidth}
  \def\vcgap{\hspace{0cm}}
  \begin{figure}[H]
    \centering
    \begin{subfigure}[t]{\vcwidth}
      \centering
      \includegraphics[scale=\vcscale]{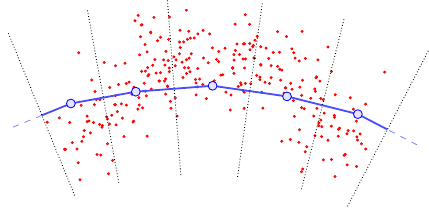}
      \caption{}
    \end{subfigure}
    \vcgap
    \begin{subfigure}[t]{\vcwidth}
      \centering
      \includegraphics[scale=\vcscale]{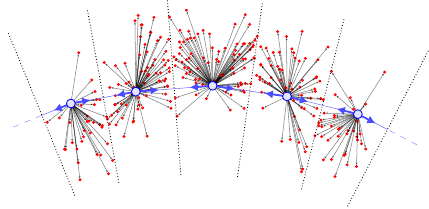}
      \caption{}
      \label{fig:vc-vectors}
    \end{subfigure}

    \begin{subfigure}[t]{\vcwidth}
      \centering
      \includegraphics[scale=\vcscale]{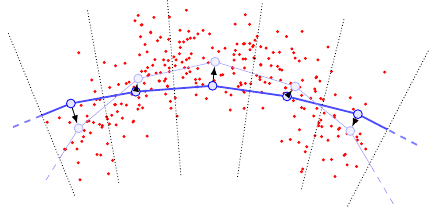}
      \caption{}
    \end{subfigure}
    \vcgap
    \begin{subfigure}[t]{\vcwidth}
      \centering
      \includegraphics[scale=\vcscale]{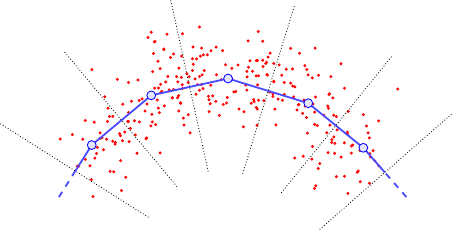}
      \caption{}
    \end{subfigure}
    \caption{A visualization of one loop of the algorithm. (a)
    \acaption\ (b) \bcaption\ (c) \ccaption\ (d) \dcaption}
    \label{fig:coupled-lloyds}
  \end{figure}

In Section \ref{sec:Wasserstein} below we particularize this descent algorithm to the case where $(X,d)$ is a Wasserstein space of probability measures over a compact space; we emphasize that, even in that setting, the algorithm is readily implemented using existing numerics for optimal transport problems.

Next we provide a consistency result for the discretization
we have described above.
  This
requires introducing the following objective functional on $K$-tuples
in $X$:
\[
\text{PPC}^{K}(\Lambda)(\gamma_{1},\ldots,\gamma_{K}):=\int_{X}d^{2}(x,\Gamma^{K})d\Lambda(x)+\beta\sum_{k=1}^{K-1}d(\gamma_{k},\gamma_{k+1})
\]
where $\{\gamma_{k}\}_{k=1}^K$  represents a discrete curve and  $d(x,\Gamma^{K})=\min_{1\leq k\leq K}d(x,\gamma_{k})$.
Notice that
$\int_{X}d^{2}(x,\Gamma^{K})d\Lambda(x)=\sum_{k=1}^{K}\sum_{n\in I_{k}}\frac{d^{2}(x_{n},\gamma_{k})}{N}$ in the case where $\Lambda$ is uniform on $N$ atoms $\{x_n\}_{n=1}^N$.

By solving the minimization problem $\text{PPC}^{K}(\Lambda)$, we
procure a $K$-tuple $\{\gamma_{1},\ldots,\gamma_{K}\}$ of ``knots''
belonging to $X$, which we think of as a discrete estimator for a minimizer of the continuum problem $\text{PPC}(\Lambda)$. However, that this discrete estimator
can also be associated to a \emph{piecewise-geodesic interpolation}
connecting the knots.

\begin{defn}
\label{def:geodesic-interpolation}
Let $\{\gamma_k\}_{k=1}^K\subset X$. A \emph{constant-speed piecewise-geodesic interpolation} of $\{\gamma_k\}_{k=1}^K$ is a curve $\gamma\in\textrm{AC}([0,1];X)$ with $\textrm{Length}(\gamma)=\sum_{k=1}^{K-1}d(\gamma_k,\gamma_{k+1})$, starting at the first knot $\gamma_1$, and  satisfying both
\[
\gamma_{t}=\gamma_{k}\text{ whenever }t=\frac{\sum_{j=1}^{k-1}d(\gamma_{j},\gamma_{j+1})}{\sum_{j=1}^{K-1}d(\gamma_{j},\gamma_{j+1})},\quad(k=2,\ldots,K)
\]
and that $d(\gamma_{t},\gamma_{s})=|t-s|\cdot\sum_{j=1}^{K-1}d(\gamma_{j},\gamma_{j+1})$
for all $t,s\in[0,1]$. In other words, $\gamma_{t}$ is a geodesic
between $\gamma_{k}$ and $\gamma_{k+1}$ when restricted to each
corresponding time interval, and has constant speed over the whole
interval $[0,1]$.

\end{defn}

Thus, the discrete-to-continuum convergence question we address is
actually the following: as we send $K\rightarrow\infty$, does a given constant-speed
piecewise-geodesic interpolation of the minimizing knots
$\{\gamma_{1},\ldots,\gamma_{K}\}$ converge to an $AC$ curve
$\gamma_{t}$ which then minimizes $\text{PPC}(\Lambda)$? We address
this question in the following theorem, which also allows for
simultaneously taking a convergent sequence of data distributions
$(\Lambda_{N})_{N\in\mathbb{N}}$.

\begin{theorem}[Discrete to continuum]
\label{thm:discrete-to-continuum}
Let $X$ be a compact geodesic metric space. Let $K,N\in\mathbb{N}$ and
$\Lambda_{N}\in\mathcal{P}(X)$. Suppose also that
$\Lambda_{N}\rightharpoonup^{*}\Lambda$ where
$\Lambda\in\mathcal{P}(X)$. Let $\{\gamma_{k}^{*K}\}_{k=1}^{K}$ be a
minimizer of $\text{PPC}^{K}(\Lambda_{N})$, and let $\gamma^{*K}_t$ be a
constant-speed piecewise geodesic interpolation of
$\{\gamma_{k}^{*K}\}_{k=1}^{K}$.

Then, up to passage to a subsequence
in $K$, we have that: there exists an AC curve $\gamma^{*}$ such that
$\gamma_{t}^{*K}\rightarrow\gamma_{t}^{*}$ uniformly in
$t$, and $\gamma^{*}$ is a minimizer for $\text{PPC}(\Lambda)$.
Moreover, given any uniformly convergent subsequence of geodesic interpolations
$\gamma^{*K}$ of minimizers $\{\gamma_{k}^{*K}\}_{k=1}^{K}$ with limit
$\gamma^{*}\in AC([0,1];X)$, it holds that $\gamma^{*}$ is a minimizer
of $\text{PPC}(\Lambda)$.
\end{theorem}

We prove this theorem in Appendix \ref{subsec:proofs-for-basics}.

Lastly we mention that in Appendix \ref{sec:extensions}, we discuss some ways that the variational problem $\text{PPC}$ and its discrete approximation $\text{PPC}^K$ can be modified in their details. One variant is an alternate \emph{nonlocal discretization scheme} similar to Nadaraya-Watson kernel regression, which is also consistent in the same sense as Theorem \ref{thm:discrete-to-continuum} (Proposition \ref{prop:discrete-nonlocal-to-continuum}). We use this scheme in our experiments in Section \ref{sec:seriation}, as it seems to offer better performance in practice. We also consider a semi-supervised variant of the $\text{PPC}$ functional, meaning that the locations of some of the points along $\gamma$ are fixed in advance.

\section{Principal curves in the space of probability measures}
\label{sec:Wasserstein}

The preceding sections have been concerned with principal curves
in the rather abstract setting of metric spaces. In this section,
we highlight the case where the metric space $X$ is taken to be $(\mathcal{P}(V),W_{2})$,
the 2-Wasserstein space of probability measures over a compact metric
space $V$.

Given probability measures $\mu,\nu$ on $V$, we write $\Pi(\mu,\nu)$ to denote the space of couplings betwen $\mu$ and $\nu$. For $p\in[1,\infty)$  we define  the \emph{p-Wasserstein metric}
on $\mathcal{P}(V)$ as follows:
\[
 W_{p}(\mu,\nu) := \inf_{\pi\in\Pi(\mu,\nu)}
\left(\int_{V\times V} d^p(x,y)d\pi(x,y)\right)^{1/p}.
\]
The metric $W_p$ metrizes weak* convergence of probability measures on $V$, and $(\mathcal{P}(V),W_{p})$ is a compact metric space \cite{ambrosio2013user}. In this article we only use the $W_2$ or $W_1$ metric specifically.

We can make sense of the length of a curve of probability measures using this metric: in particular,
for any absolutely continuous (AC) curve $\gamma_{(\cdot)} : [0,1]
\rightarrow \mathcal{P}(V)$,
\[
\text{Length}_{W_{2}}(\gamma):=\int_{0}^{1}|\dot{\gamma}_{t}|dt
\]
where $|\dot{\gamma}_{t}|$ is the metric speed of $\gamma$ at time
$t$:
\[
  |\dot{\gamma}_{t}| = \lim_{s\rightarrow t} \frac{W_{2}(\gamma_{s},
    \gamma_{t})}{|s-t|}.
\]
In the case where $(V,d)$ is a compact geodesic metric space, then so too is $(\mathcal{P}(V),W_2)$ \cite[Theorem 2.10]{ambrosio2013user}. In particular, if
$V$ is a compact convex domain in $\mathbb{R}^{d}$ then
$(\mathcal{P}(V),W_{2})$ is a compact geodesic metric space.

Thus far, we have defined principal curves in terms of a probability
measure $\Lambda$ on the metric space $X$. Putting $X=\mathcal{P}(V)$
means that now $\Lambda\in\mathcal{P}(\mathcal{P}(V))$, in other words
$\Lambda$ is a probability measure over probability measures. In
situations where we need to consider convergence of a sequence of
$\Lambda$'s, we still make use of weak* convergence, but let us state
precisely what this means here. To emphasize, in this setting a
sequence $\Lambda_n$ in $\mathcal{P}(\mathcal{P}(V))$ converges in the
weak* sense to $\Lambda\in\mathcal{P}(\mathcal{P}(V))$ if and only if,
for every ``test function'' $\varphi:\mathcal{P}(V)\rightarrow \mathbb{R}$,
$$\int_{\mathcal{P}(V)}\varphi d\Lambda_n \underset{n\rightarrow\infty}{\longrightarrow}\int_{\mathcal{P}(V)}\varphi d\Lambda.$$
Here the test function $\varphi$ is any continuous function from $\mathcal{P}(V)$ to $\mathbb{R}$, where $\mathcal{P}(V)$ is itself equipped with the weak* topology.

We also note that, by applying Prokhorov's theorem twice, $\mathcal{P}(\mathcal{P}(V))$ is compact when equipped with the specific weak* topology we have just described.

\begin{rem}
The class of AC curves for the metric $W_2$ is relevant for our
intended applications because it includes ``naturally arising''
time-dependent probability measures such as the curve of marginals for
certain drift-diffusion processes. To wit, for the SDE $dX_t = -\nabla
V(X_t) + \sqrt{2}dW_t$ (see \cite{schiebinger2019optimal} for how this
SDE arises in models in developmental biology) the distribution
$\rho_t$ of $X_t$ at time $t$ satisfies the Fokker-Planck equation
$\partial_t \rho_t =  \nabla \cdot (\rho_t \nabla V) +  \Delta
\rho_t$. The solution $\rho_t$ to this equation can be identified as
the gradient flow of the relative entropy functional
$\mathcal{H}(\cdot \mid e^{-V}dx)$ with respect to the metric $W_2$
\cite{jordan1998variational}. We do not elaborate on the precise sense
of ``gradient flow with respect to a given metric'' that is meant here, but see \cite{ambrosio2008gradient,ambrosio2013user, santambrogio2017euclidean} for details. However, we note that gradient flows in this sense are automatically $1/2$-H\"{o}lder continuous \cite{santambrogio2017euclidean}, but need not be Lipschitz. In particular, it holds that the curve of marginals for this drift-diffusion SDE belongs to the set of AC curves for the metric $W_2$.
\end{rem}

Given a finite family $\{\mu_k\}_{k=1}^K$ of probability measures on $V$ and nonnegative weights $\{\lambda_k\}_{k=1}^K$ summing to 1, a \emph{2-Wasserstein barycenter} is a minimizer of the variational problem
$$\min_{\nu\in\mathcal{P}(V)}\sum_{k=1}^K \lambda_k W_2^2(\mu_k,\nu).$$
Indeed, this is simply a particular instance of the notion of barycenter on a general metric space. Wasserstein barycenters were initially studied in \cite{agueh2011barycenters}, which established existence of minimizers and provided sufficient conditions for the Wasserstein barycenter to be unique.

One can extend the notion of Wasserstein barycenter by replacing the finite set of weights $\{\lambda_k\}_{k=1}^K$ with a probability measure $\Lambda \in \mathcal{P}(\mathcal{P}(V))$, in which case a Wasserstein barycenter for $\Lambda$ is a minimizer to the variational problem
$$\min_{\nu\in\mathcal{P}(V)}\int_{\mathcal{P}(V)}  W_2^2(\mu,\nu)d\Lambda (\mu).$$
For this continuum notion of Wasserstein barycenter, existence and conditions for uniqueness were  established in \cite{kim2017wasserstein}.

We note that efficient numerics for Wasserstein barycenters are well-established. In the case of a finite family $\{\mu_k\}_{k=1}^K$ of probability measures where each $\mu_k$ is itself composed of finitely many atoms, it is known that any Wasserstein barycenter is again supported  on finitely many atoms \cite{anderes2016discrete}; in this case the computation of Wasserstein barycenters is simply a finitary linear programming problem. We refer the reader to \cite{cuturi2014fast} for a now-classical approach based on entropic regularization, and \cite{chizat2023doubly} for an up-to-date comparison of existing methods. More generally, the textbook \cite{peyre2019computational} offers a thorough explanation of computational aspects of optimal transport writ large.

In this setting, the PPC objective reads as follows:
\[
\text{PPC}(\Lambda)(\gamma)=\int_{\mathcal{P}(V)}W_{2}^{2}(\mu,\Gamma)d\Lambda(\mu)+\beta\text{Length}(\gamma).
\]
Here $\Lambda\in\mathcal{P}(\mathcal{P}(V))$, $\gamma\in AC([0,1];\mathcal{P}(V))$, and just as before $\Gamma$ is the range of $\gamma$ in $\mathcal{P}(V)$; and
where $\mathcal{P}(V)$ is equipped with the $W_{2}$ metric, so that
$\text{Length}(\gamma)$ is  defined with respect to $W_{2}$.
Likewise, the $\text{PPC}^{K}$ objective (which discretizes the curves
$\gamma$) reads as follows, where $\gamma_{1},\ldots,\gamma_{K}\in\mathcal{P}(V)$:
\[
\text{PPC}^{K}(\Lambda)(\gamma_{1},\ldots,\gamma_{K})=\int_{\mathcal{P}(X)}W_{2}^{2}(\mu,\Gamma^{K})d\Lambda(\mu)+\beta\sum_{k=1}^{K-1}W_{2}(\gamma_{k},\gamma_{k+1})
\]
where $\Gamma^{K}=\{\gamma_{1},\ldots,\gamma_{K}\}$ and hence $W_{2}^{2}(\mu,\Gamma^{K})=\min_{1\leq k\leq K}W_{2}^{2}(\mu,\gamma_{k})$.
In the particular case where $\Lambda_{N}=\frac{1}{N}\sum_{n=1}^{N}\delta_{\mu_{n}}$
(that is, $\Lambda$ is supported on finitely many atoms in $\mathcal{P}(V)$
and those atoms have equal weight) we have
\begin{align*}
\text{PPC}^{K}(\Lambda_{N})(\gamma_{1},\ldots,\gamma_{K}) & =\frac{1}{N}\sum_{n=1}^{N}W_{2}^{2}(\mu_{n},\Gamma^{K})+\beta\sum_{k=1}^{K-1}W_{2}(\gamma_{k},\gamma_{k+1})\\
 & =\frac{1}{N}\sum_{k=1}^{K}\sum_{n\in I_{k}}W_{2}^{2}(\mu_{n},\gamma_{k})+\beta\sum_{k=1}^{K-1}W_{2}(\gamma_{k},\gamma_{k+1})
\end{align*}
where as before $\{I_{k}\}$ is a Voronoi partition of $\{\mu_{n}\}_{n=1}^{N}$
with respect to $\{\gamma_{k}\}_{k=1}^{K}$, in other words the $I_{k}$'s
are disjoint and every $\mu_{n}\in I_{k}$ satisfies $\gamma_{k}\in\argmin_{1\leq k\leq K}W_{2}^{2}(\mu_{n},\gamma_{k})$.

Lastly, the nonlocal discrete objective $\text{PPC}_{w}^{K}(\Lambda_{N})$ introduced in Appendix \ref{sec:extensions} also makes sense in this setting, and is defined in an identical manner to $\text{PPC}^K(\Lambda_{N})$.

For the sake of motivation, we give two examples where a distribution $\Lambda$ on the space of distributions $\mathcal{P}(V)$ arises naturally, and for which the task of estimating principal curves makes sense.

\begin{example}

  Suppose we observe a
    continuously-varying distribution $\rho_{t}$ which is supported on $V$.
    Consider the map
  \[
    \rho_{(\cdot)}:[0,1]\rightarrow\mathcal{P}(V);\qquad t\mapsto\rho_{t}.
  \]
  Then let $\Lambda$ denote the pushforward of the Lebesgue measure on
  $[0,1]$ via $\rho_{(\cdot)}$. This measure $\Lambda$ encodes for us
  the ``1d volume measure'' of the curve $\rho_t$; phrased
  differently, sampling from $\Lambda$ corresponds to first drawing
  i.i.d. time points $t_{1},t_{2},\ldots,t_{N}$ and then for each
  $t_{n}$ observing the distribution $\rho_{t_{n}}$. Note however that
  $\Lambda$ does not encode time labels; so, the temporal
  ordering of the distributions $\rho_{t_n}$ is unknown even with
  perfect knowledge of $\Lambda$.
\end{example}

\begin{example}
Continuing the previous example, suppose at each time $t_{n}$ we
make \emph{noisy} observations of $\rho_{t_{n}}$: namely, consider a random
variable $Y_{t_{n}}=X_{t_{n}}+E$ taking values in a domain $U$ which is compactly contained within $\mathbb{R}^d$, where $X_{t_{n}}\sim\rho_{t_{n}}$
and $E$ is a compactly supported noise variable with distribution $\mathcal{L}(E)$ (which
is independent at each time). Then at time $t_{n}$ our observations of $Y_{t_n}$
are distributed according to $\rho_{t_{n}}*\mathcal{L}(E)$. Let $V$ be a compact domain containing the support of $\rho_{t_{n}}*\mathcal{L}(E)$ for all $t$. Then  $\Lambda \in \mathcal{P}(\mathcal{P}(V))$ corresponds to the pushforward of the Lebesgue measure
on $[0,1]$ via $\rho_{(\cdot)}*\mathcal{L}(E)$. Crucially,
in this case $\Lambda$ is still concentrated on a curve in $\mathcal{P}(\mathcal{P}(V))$,
and so this situation (with additive measurement noise) presents no
technical distinction from the previous example.

By contrast, suppose that $Y_{t_{n}}=X_{t_{n}}+E_{\omega}$, where
$E_{\omega}$ is a random variable that is itself independently randomly
chosen, according to index $\omega\in\Omega$ which is drawn according
to some probability distribution $\mathbb{P}$. (For instance, consider
the case where the variance of the measurement noise is in some sense
changing randomly between observations.) Then, $\Lambda$ is the pushforward
of $\text{Leb}_{[0,1]}\times\mathbb{P}$ under the map $(t,\omega)\mapsto\rho_{t}*\mathcal{L}(E_{\omega})$.
In this case $\Lambda$ is not concentrated on a curve in $\mathcal{P}(\mathcal{P}(V))$,
but rather is ``concentrated near a curve'' provided that $\mathcal{L}(E_{\omega})$
does not vary too much (in the $W_{2}$ sense) with $\omega$.
\end{example}

In this setting, our results for (compact, geodesic) metric spaces particularize
as follows: minimizers for $\text{PPC}(\Lambda)$ exist for any $\Lambda \in\mathcal{P}(\mathcal{P}(V))$;
minimizers are stable with respect to weak{*} perturbations of $\Lambda$;
and up to subsequence, minimizers for $\text{PPC}^{K}(\Lambda_{N})$
converge to minimizers of $\text{PPC}(\Lambda)$ as $K\rightarrow\infty$ and
$\Lambda_{N}\rightharpoonup^{*}\Lambda$ (and similarly for the alternate objective $\text{PPC}_{w}^{K}(\Lambda_{N})$ from Appendix \ref{sec:extensions}). In particular we have that
$\Lambda_{N}\rightharpoonup^{*}\Lambda$ with probability 1 when $(\Lambda_{N})_{N\in\mathbb{N}}$
is a sequence of empirical measures for $\Lambda$; meaning that $\Lambda_{N}=\frac{1}{N}\sum_{n=1}^{N}\delta_{\mu_{n}}$
and the $\mu_{n}$'s are \emph{random probability measures} which
are i.i.d. with distribution $\Lambda$.

Statistically, the interpretation of principal curves in the space
$(\mathcal{P}(V),W_{2})$ is as follows. We have a family of probability
measures $\{\mu\}$ with distribution $\Lambda\in\mathcal{P}(\mathcal{P}(V))$, which we believe to be concentrated near an AC
curve $\rho_{t}$ in the space $(\mathcal{P}(V),W_{2})$.
Can we infer $\rho_{t}$ from partial observations of the distribution
$\Lambda$? However, in this space we must be careful about what it
means to ``make observations'' of $\Lambda$. One option is that
we draw i.i.d. $\mu_{n}\sim\Lambda$ and get to ``observe'' each
probability measure $\mu_{n}\in\mathcal{P}(V)$ exactly. However,
this observation scheme is rather idealized for the applications we
wish to consider. Rather, we wish to consider the following observation
scheme:
\begin{align*}
\Lambda & \in\mathcal{P}(\mathcal{P}(V))\\
\mu_{1},\ldots,\mu_{N} & \sim\Lambda\\
X_{n}^{1},\ldots,X_{n}^{M_n} & \sim\mu_{n}.
\end{align*}
This scheme has the following interpretation: we get $N$ ``batches''
of data points in $V$, and each batch of data consists of $M_n$ i.i.d. draws from a randomly
chosen probability measure $\mu_{n}$, which in turn is randomly distributed
according to $\Lambda$. For simplicity we generally take $M_n=M$, i.e. every batch has the same number of draws.  We then ``store'' the data as follows: for
each batch of data, we have the empirical measure $\hat{\mu}_{n}=\frac{1}{M}\sum_{m=1}^M\delta_{X_{n}^{m}}$,
and then, corresponding to the collection of batches of data, we form
the measure
\[
\hat{\Lambda}_{N,M}:=\frac{1}{N}\sum_{n=1}^{N}\delta_{\hat{\mu}_{n}}.
\]
Finally, the functional $\text{PPC}^{K}(\hat{\Lambda}_{N,M})$ now
encodes the discretized principal curve problem for the data which
has been observed according to this scheme, in other words we minimize the objective
\[
\text{PPC}^{K}(\hat{\Lambda}_{N,M})(\gamma_{1},\ldots,\gamma_{K})=\frac{1}{N}\sum_{k=1}^{K}\sum_{n\in I_{k}}W_{2}^{2}(\hat{\mu}_{n},\gamma_{k})+\beta\sum_{k=1}^{K-1}W_{2}(\gamma_{k},\gamma_{k+1}).
\]

The nonlocal discrete functional $\text{PPC}_{w}^{K}$ from Appendix \ref{sec:extensions} is defined similarly.
We emphasize that the optimization problem of minimizing $\text{PPC}^{K}(\hat{\Lambda}_{N,M})$ is amenable to off-the-shelf numerical methods for optimal transport; indeed, minimizing $\text{PPC}^{K}(\hat{\Lambda}_{N,M})$ can be viewed as solving $K$ many Wasserstein barycenter problems which are coupled together via the global length penalty $\beta\sum_{k=1}^{K-1}W_{2}(\gamma_{k},\gamma_{k+1})$. In Appendix \ref{appn:extra-Wasserstein}, we offer further details regarding the implementation of Algorithm \ref{alg:coupled-lloyd} in this case, given the choice of a Wasserstein barycenter solver. In Section \ref{sec:experiments} below, we offer concrete examples of numerical experiments with measure-valued data of the form $\hat\Lambda_{N,M}$.

In the remainder of this section, we address the following question:
do minimizers of $\text{PPC}^{K}(\hat{\Lambda}_{N,M})$ converge to
minimizers of $\text{PPC}(\Lambda)$ as $K\rightarrow\infty$, $N\rightarrow\infty$,
and $M\rightarrow\infty$, where $N$ is the number of batches and
$M$ is the number of datapoints per batch? To reiterate, in the case where $\Lambda=\rho_{(\cdot)\#}\textrm{Leb}_{[0,1]}$, the interpretation of the three parameters is: $N$ controls our temporal resolution of the ground truth curve $\rho$, $M$ controls the measurement precision with which we observe each time marginal $\rho_t$, and $K$ controls the resolution of our estimator, which is a minimizer of $\text{PPC}^{K}(\hat{\Lambda}_{N,M})$.

By Theorem \ref{thm:discrete-to-continuum} it suffices to show that
the ``doubly empirical measure'' $\hat{\Lambda}_{N,M}$ converges
to $\Lambda$ in the weak{*} sense with probability 1. Surprisingly,
we were unable to find any investigation of this convergence question
in the literature, with one exception: as our work was nearing completion, we learned of the recent article \cite{catalano2024hierarchical} which considers closely related questions for an application to nonparametric Bayesian inference. (Their results are more quantitative than ours but require stronger assumptions on the base space, and so neither our nor their results imply the other.) Accordingly, we provide an argument that $\hat{\Lambda}_{N,M}\rightharpoonup^{*}\Lambda$
under mild conditions.

In order to state this convergence result, we need to consider the
1-Wasserstein metric on $\mathcal{P}(\mathcal{P}(V))$. Recall that
the 1-Wasserstein metric makes sense atop any complete separable metric
space. Thus, let $d_{\mathcal{P}(V)}$ be a complete separable metric
on $\mathcal{P}(V)$; given this choice of metric on $\mathcal{P}(V)$,
for any $\Lambda,\Xi\in\mathcal{P}(\mathcal{P}(V))$ we can define
\[
\mathbb{W}_{1}(\Lambda,\Xi):=\inf_{\pi\in\Pi(\Lambda,\Xi)}\int_{\mathcal{P}(V)\times\mathcal{P}(V)}d_{\mathcal{P}(V)}(\mu,\mu^{\prime})d\pi(\mu,\mu^{\prime}).
\]
We use the notation $\mathbb{W}_{1}$ just for disambiguation with
$W_{1}$ as defined on $\mathcal{P}(V)$ (rather than $\mathcal{P}(\mathcal{P}(V))$).

We defer the proof of the following theorem to Appendix \ref{subsec:proofs-for-Wasserstein}.

\begin{theorem}[Iterated Glivenko-Cantelli theorem]
\label{thm:iterated-glivenko-cantelli}
Let $V$ be a compact metric space, let $\Lambda\in\mathcal{P}(\mathcal{P}(V))$,
and let $(\Lambda_{N})_{N\in\mathbb{N}}$ be a sequence of empirical
measures for $\Lambda$. Let $M$ depend on $N$ and assume that $M\rightarrow\infty$
as $N\rightarrow\infty$. Then, regarding the doubly empirical measure
$\hat{\Lambda}_{N,M}$ for $\Lambda$, we have the following convergence
results:
\begin{enumerate}
\item There exists a metric $d_{\mathcal{P}(V)}$ which metrizes weak{*}
convergence on $\mathcal{P}(V)$, for which $\mathbb{W}_{1}(\hat{\Lambda}_{N,M},\Lambda)$
converges to $0$ in probability. In particular, there exists a subsequence
of $N$ (and thus of $M$) along which $\hat{\Lambda}_{N,M}\rightharpoonup^{*}\Lambda$
almost surely.
\item Assume that $M\geq C(\log N)^{q}$ for some $C>0$ and $q>1$. Then
$\hat{\Lambda}_{M,N}\rightharpoonup^{*}\Lambda$ almost surely.
\end{enumerate}
\end{theorem}

\begin{rem}
The proof of Theorem \ref{thm:iterated-glivenko-cantelli} also works in some cases where the number of samples $M_n$ in the empirical measure $\hat{\mu}_n$ varies with $n$. For example, if we take $M:=\min_{1\leq n \leq N}M_n$ and assume that $M\rightarrow\infty$ as $N\rightarrow \infty$, then Theorem \ref{thm:iterated-glivenko-cantelli} still holds as stated.
\end{rem}

\begin{rem}
Part (1) of this theorem is enough for our purposes, since our stability results all require the passage to a subsequence regardless. We include part (2) for completeness, since it shows that the need to pass to a subsequence can be replaced with a mild assumption on the dependency of $M$ on $N$. Additionally, we note that, even though part (1) implies that every subsequence of $N$ contains a further subsequence along which $\hat{\Lambda}_{M,N}\rightharpoonup^{*}\Lambda$ almost surely, this does \emph{not} imply that $\hat{\Lambda}_{M,N}\rightharpoonup^{*}\Lambda$ almost surely, since almost sure convergence is non-metrizable.
\end{rem}

\begin{rem}
Part (2) leaves open the possibility that, along some subsequences, $\hat{\Lambda}_{M,N}$ fails to converge to $\Lambda$ if $M$ grows sub-logarithmically with $N$; we are unable to resolve this question. However, there is at least one mathematical reason to have the intuition that taking $M$ to be very small in terms of $N$ may cause problems. Namely, one the one hand, if we first send $M\rightarrow\infty$ with $N$ initially fixed, then one simply recovers the (single-level) i.i.d. empirical measure $\Lambda_N$ for $\Lambda$. On the other hand, if we fix $M$ and first send $N$ to infinity, this corresponds to drawing $M$ i.i.d. samples from \emph{every} measure in the support of $\Lambda$: the resulting object is an \emph{i.i.d. random field} over the space $\mathcal{P}(V)$. This object is known to be measure-theoretically degenerate in all but exceptional cases, see \cite[Proposition 1.1]{sun1998almost}.

\end{rem}

It is also possible to establish analogs of Theorem \ref{thm:iterated-glivenko-cantelli} (and therefore analogs of Theorem~\ref{thm:main-intro} via Theorem \ref{thm:discrete-to-continuum}) in cases where the empirical data is not observed exactly, but rather we only get \emph{partial measurements} of each sample in each empirical distribution. For example, in the context of scRNA-seq, the support points of the empirical distributions are also noisy (see Eq~\eqref{eq:noisy_empirical}).
From the noisy empirical distributions,
we then form the \emph{three level empirical measure} $\hat{\Lambda}_{N,M,R}$
where $R$ is the number of reads per ``fully observed'' empirical measure $\hat\mu_n$.

Similarly to the previous theorem,
it is possible to prove  convergence under the assumption
that $M$ grows fast enough enough relative to $M$ and also that
$R$ grows fast enough relative to $M$, using the concentration inequalities
established in \cite{kim2024optimal}. In order to make use of these concentration inequalities, we also take as given Assumption 2.3 from the article  \cite{kim2024optimal}; this technical assumption, which we do not precisely restate here, amounts to requiring that the distribution of reads over the support points in a given empirical measure is approximately uniform.
\begin{prop}
\label{prop:finite-reads}
Let $V$ be the unit simplex in $\mathbb{R}^{d}$, let $\Lambda\in\mathcal{P}(\mathcal{P}(V))$,
and let $(\Lambda_{N})_{N\in\mathbb{N}}$ be a sequence of empirical
measures for $\Lambda$. Let $M$ depend on $N$ and assume that $M\rightarrow\infty$
as $N\rightarrow\infty$. Additionally, let $R$ be the number of
reads per empirical measure $\hat{\mu}_n$, and assume $M/R\rightarrow 0$ as $M\rightarrow\infty$, and also that the distribution of reads satisfies \cite[Assumption 2.3]{kim2024optimal}.
Then, regarding the three level empirical measure $\hat{\Lambda}_{N,M,R}$
for $\Lambda$, we have the following convergence results:
\begin{enumerate}

\item Setting $d_{\mathcal{P}(V)}=W_{1}$,
then it holds that $\mathbb{W}_{1}(\hat{\Lambda}_{N,M,R},\Lambda)$
converges to $0$ in probability. In particular, there exists a subsequence
of $N$ (and thus of $M$ and $R$) along which $\hat{\Lambda}_{N,M,R}\rightharpoonup^{*}\Lambda$
almost surely.

\item Furthermore, if $M\geq C(\log N)^{q}$ for some $C>0$ and $q>1$,
then it also holds that $\hat{\Lambda}_{N,M,R} \rightharpoonup^* \Lambda$ almost surely.
\end{enumerate}
\end{prop}

A proof can be found in Section \ref{subsec:proofs-for-Wasserstein}.

\section{Application to the seriation problem}\label{sec:seriation}

In this section, we return to the problem outlined in the introduction (Figure~\ref{fig:intro}).
Namely, given a data distribution which comes from observing a ground truth curve with unknown time labels, we show how to infer both the ground truth curve and the ordering of the points along said curve, and so obtain a principal curves based seriation algorithm.

In Section \ref{subsec:consistency} we describe how principal curves provide consistent estimators for the ground truth curve as well as the ordering of points along said curve. We first state our consistency results for the general case of metric-space-valued data and then explain how our results cover the case of data taking values in the Wasserstein space of probability measures. In Section \ref{sec:experiments}, we then provide  numerical experiments where principal curves for measure-valued data are used for seriation.

\subsection{Consistency theory}
\label{subsec:consistency}

While in previous sections we have considered arbitrary data distributions $\Lambda$, here we focus on the case where $\Lambda$ is the ``1d volume measure''
associated to some ``ground truth'' curve $\rho\in\text{AC}([0,1];X)$,
meaning that $\Lambda=(\rho_{(\cdot)})_{\#}\text{Leb}_{[0,1]}$.
As a step towards our main consistency results below, we first state the following proposition, which asserts that principal curves for $\Lambda$ are concentrated around, and no longer than, $\rho$ itself. See Appendix \ref{subsec:proofs-for-seriation} for the proof.

\begin{prop}\label{prop:lengthbound-1d}
(i) Let $X$ be a compact metric space. Let $\rho\in AC([0,1];X)$
and $\Lambda=\rho_{(\cdot)\#}\text{Leb}_{[0,1]}$.
Let $\gamma^{*\beta}\in AC([0,1];X)$ be a minimizer for $\text{PPC}(\Lambda;\beta)$.
Then,
\[
\text{Length}(\gamma^{*\beta})\leq\text{Length}(\rho).
\]

(ii) (Concentration estimate) Let $\alpha>0$.
Then,
we have that $\Gamma^{*\beta}$ (the graph of $\gamma^{*\beta}$) is concentrated around $\rho$
in the following sense:
\[
\text{Leb}_{[0,1]}\left\{ t:d^{2}(\rho_{t},\Gamma^{*\beta})\geq\alpha\right\} \leq\frac{\beta\text{Length}(\rho)}{\alpha}.
\]
\end{prop}

Intuitively, one might expect that we can send $\beta\rightarrow 0$ in the statement of Proposition~\ref{prop:lengthbound-1d}, and show that in the limit, the principal curves $\gamma^{*\beta}$ converge to the ground truth curve $\rho$. The following theorem tells us that this is indeed the case.

\begin{theorem}[consistency up to time-reparametrization]
\label{thm:seriation-consistency}
Let $X$ be a compact metric
space. Let $\rho\in AC([0,1];X)$ and $\Lambda=\rho_{(\cdot)\#}\mathrm{Leb}_{[0,1]}\in\mathcal{P}(X)$.
Suppose that $t\mapsto\rho_{t}$ is injective. Let $\beta\rightarrow0$,  and
$\gamma^{*\beta}\in AC([0,1];X)$ be a minimizer for $\text{PPC}(\Lambda;\beta)$.
Then, up to passage to a subsequence of $\beta$, there exists
a $\gamma^*\in AC([0,1];X)$ such that
\[
\gamma_{t}^{*\beta}\rightarrow\gamma^*_{t}\text{ uniformly in }t
\]
and $\gamma^*$ is a reparametrization of $\rho$ which is either order-preserving or order-reversing. Moreover any  limit of a  convergent
subsequence of $\gamma^{*\beta}$'s  satisfies this property.

\end{theorem}

The proof can be found in Appendix \ref{subsec:proofs-for-seriation}. (We remark that the proof is not simply a direct consequence of Proposition \ref{prop:lengthbound-1d}, as the limiting objective $\textrm{PPC}(\Lambda;0)$ is poorly behaved.)

The preceding result establishes that the ground truth $\rho$ and
minimizer $\gamma^*$ have the same graphs, and that $\gamma^*$ is either an order-preserving, or order-reversing,
reparametrization of $\rho$. If we are purely interested in seriation, it remains only to check
whether $\gamma^*$ has the correct, versus backwards, ordering compared
to $\rho$. Our variational framework cannot do this directly.
We note, however, that it suffices to ask an oracle whether, according
to the ground truth, $\rho_0$ comes before $\rho_1$, or
vice versa. In our domain application of scRNA data analysis, we expect that the true ordering of $\rho_0$ versus $\rho_1$ is typically available as expert knowledge.

By combining Theorem \ref{thm:seriation-consistency} with our earlier discrete-to-continuum result from Theorem \ref{thm:discrete-to-continuum}, we obtain the following corollary.

\begin{cor}
Let $X$ be a compact geodesic metric
space. Let $\rho\in AC([0,1];X)$ and $\Lambda=\rho_{(\cdot)\#}\mathrm{Leb}_{[0,1]}$.
Suppose that $t\mapsto\rho_{t}$ is injective. Let $(\Lambda_N)_{N\in\mathbb{N}}$ be a sequence in $\mathcal{P}(X)$ converging weak*ly to $\Lambda$. Let $\{\gamma_{k}^{*\beta,K}\}_{k=1}^{K}$ be a
minimizer of $\text{PPC}^{K}(\Lambda_{N};\beta)$, and let $\gamma^{*\beta,K}_t$ be a
constant-speed piecewise geodesic interpolation of
$\{\gamma_{k}^{*\beta,K}\}_{k=1}^{K}$.

Then: up to passage to a subsequence twice,
as we send $N,K\rightarrow\infty$ followed by $\beta\rightarrow 0$, there exists a $\gamma^{*}\in \textrm{AC}([0,1];X)$ such that
$\gamma_{t}^{*\beta,K}\rightarrow\gamma_{t}^{*}$ uniformly in
$t$, and $\gamma^*$ is a reparametrization of $\rho$ which is either order-preserving or order-reversing.
\end{cor}

Alternatively, one can replace minimizers of $\mathrm{PPC}^K(\Lambda_N;\beta)$ with minimizers of the nonlocal discrete objective $\mathrm{PPC}^K_w(\Lambda_N;\beta)$ in the statement of the corollary above, and invoke Proposition \ref{prop:discrete-nonlocal-to-continuum} instead of Theorem \ref{thm:discrete-to-continuum}. Likewise the two passages to a subsequence can be dropped if we consider the convergence of $\gamma_{t}^{*\beta,K}$ to $\gamma^*_t$ in a slightly weaker topology, such as the Hausdorff topology on the ranges of the curves $\gamma_{t}^{*\beta,K}$.

By applying this corollary in the case where: $X=\mathcal{P}(V)$, $V$ is a compact convex domain in $\mathbb{R}^d$, $d=W_2$, and $(\Lambda_N)_{N\in\mathbb{N}}$ is a sequence of doubly empirical measures for $\Lambda$, then (thanks to Theorem \ref{thm:iterated-glivenko-cantelli}) we recover a precise statement of Theorem \ref{thm:main-intro} which was advertised in the introduction. Likewise, employing Proposition \ref{prop:finite-reads} in place of Theorem \ref{thm:iterated-glivenko-cantelli} above gives us the version of Theorem \ref{thm:main-intro} where the empirical data is measured imprecisely with finite reads.

\subsubsection*{Ordering estimation}
The results above concern consistency at the level of estimation of the curve $\rho$; we now explain how to use principal curves to consistently assign an \emph{ordering} to points along $\rho$.

Suppose that $\Lambda_N$ is discrete and so supported on atoms $\{x_n\}_{n=1}^N$ belonging to $X$. Let $\gamma^N$ be a principal curve for $\Lambda_N$. Using $\gamma^N$, we can assign \emph{projection pseudotimes}
 to the $x_n$'s, as follows: define
\begin{equation}
\label{eq:pseudotime_objective}
\hat{\tau}(x_n)\in\arg\min_{t\in [0,1]}d^2(x_n,\gamma_t^N).
\end{equation}
In other words, we assign each data point $x_n$ a time label $\hat{\tau}(x_n)$ by projecting $x_n$ onto $\gamma^N$ and using the time argument of this projection. Note that in general this projection is not unique, in which case we simply pick one of the projections arbitrarily. In practice, $\hat{\tau}(x_n)$ will be approximated, e.g. by instead projecting onto a discrete approximation of a $\gamma^N$ which is a minimizer for  $\mathrm{PPC}^K(\Lambda_N)$. In Appendix \ref{appn:linearized-ot-projection-scheme}, we detail a finer numerical scheme for approximating $\hat{\tau}(x_n)$ in the Wasserstein case, based on linearized optimal transport.

Now, consider the limiting case where $\Lambda_N \rightharpoonup^* \Lambda$ and $\Lambda$ is supported on a ground truth curve $\rho_t$. In this case, Theorem \ref{thm:seriation-consistency} tells us that by then sending $\beta\rightarrow 0$, we extract a principal curve with the same range as $\rho_t$. In this limiting case, the projection distance from each point in the support of $\Lambda$ to the limiting principal curve $\gamma$ is, of course, \emph{zero}. In other words, in the limit our pseudotime assignment simply reads off the time labels from $\gamma$, which is a monotone or reverse-monotone time-reparametrization of the ground truth $\rho$.

We also have the following consistency result for the ordering obtained from projection pseudotimes, which allows for $\beta$ to be small but nonzero.

\begin{prop}
\label{prop:pseudotimes-consistency}
Under the same assumptions as Theorem \ref{thm:seriation-consistency}, fix $T$ distinct points $\{\rho_i\}_{i=1}^T$ along $\rho_t$ whose true time labels $\{t_i\}_{i=1}^T$ are unknown. Let $\gamma^{*\beta}\in\arg\min\textrm{PPC}(\Lambda;\beta)$, and let $\hat{\tau}(\rho_i)\in\arg\min_{t\in[0,1]}d^2(\rho_i,\gamma^{*\beta}_t)$ be a projection pseudotime for $\rho_i$. Lastly, let $\beta_j \rightarrow 0$ be a subsequence of $\beta$ converging to zero, along which $\gamma^{*\beta}_t$ converges uniformly in $t$.

Then, for all sufficiently small $\beta_j>0$, it holds  that either:
\begin{enumerate}
\item For all  $1\leq i,i'\leq T$, $t_i < t_{i'} \iff \hat{\tau} (\rho_i) < \hat{\tau} (\rho_i')$, or
\item For all  $1\leq i,i'\leq T$, $t_i < t_{i'} \iff \hat{\tau} (\rho_i) > \hat{\tau} (\rho_{i'})$.
\end{enumerate}
In other words, the ordering given by $\hat{\tau}$ is correct up to total reversal.
\end{prop}

This result is derived directly from Theorem
\ref{thm:seriation-consistency}; a
proof is provided in Appendix
\ref{subsec:proofs-for-seriation}.

\subsubsection*{Time label estimation}
Suppose we wish to use the principal curve $\gamma^{*}$ to infer the correct time
labels along $\rho$, rather than merely the ordering. Assume we have already queried an oracle for the
labels of the endpoints, so that we can put $\gamma_{0}^*=\rho_{0}$
and $\gamma_{1}^*=\rho_{1}$. By convention, we take $\gamma^*$ to have
constant speed, but $\rho$ typically has unknown speed. Nonetheless, for a particular
sampling scheme for $\rho$,
it is possible to infer the correct time-reparametrization
of $\gamma^*$
(i.e. so that $\gamma_{t}^*=\rho_{t}$ for all $t$) as
follows.
Suppose that our observations of $\rho$ are i.i.d. from $\Lambda$, and hence come in the following form:
first we draw $T$ i.i.d. random times uniformly on $[0,1]$, then
for each random time $\tau$ we observe $\rho_{\tau}$, but the time
label $\tau$ is unknown. In the limit where $T\rightarrow \infty$ then $\beta\rightarrow 0$, each $\rho_{\tau}$ coincides
with some $\gamma_{t}^*$. Define the following estimated ordering on
$\rho$'s as follows:
\[
\rho_{\tau}\preceq\rho_{\tau^{\prime}}\iff\rho_{\tau}=\gamma_{t},\rho_{\tau^{\prime}}=\gamma_{t^{\prime}},\text{ and }t\leq t^{\prime}.
\]
In other words we first endow the $\rho$'s with the ordering inherited
from $\gamma^*$. Label the $\rho_\tau$'s according to this ordering;
as the $\tau$'s are
drawn uniformly, it is consistent to estimate the true times of the
$\rho_\tau$'s simply by assigning them time labels which are
evenly spaced in $[0,1]$, in order.

\subsection{Experiments}
\label{sec:experiments}
In this subsection, we test our approach to principal curve estimation on data drawn from a single-cell RNA-sequencing time series of mouse embryonic development \cite{mouse_time_series}. Specifically we consider the data observation scheme outlined in the introduction: given a true curve of probability measures $\rho$, we sample $N$ i.i.d. times $t_i$, and at each time we observe an empirical measure $\hat\rho_{t_i}$ on $M$ datapoints drawn from the true measure $\rho_{t_i}$ at time $t_i$. This scheme mimics the scRNA-seq data collection process capturing a sample of cells present at a limited number of times. We then estimate the curve $\rho$ via the doubly empirical measure $\hat\Lambda_{N,M}=\frac{1}{N}\sum_{i=1}^N\delta_{\hat\rho_{t_i}}$ which encodes the observed data.

In practice, we compute an approximate principal curve $\hat\gamma$
for $\hat\Lambda_{N,M}$ by running a version of Algorithm
\ref{alg:coupled-lloyd} modified with a non-local kernel and fixed
endpoints. We denote this modified objective as
$\text{PPC}_{w}^{K}(\hat\Lambda_{N,M};\beta)$ and describe it as
Algorithm \ref{alg:coupled-lloyd-nonlocal-fixed} in Appendix
\ref{sec:extensions}. $\hat\gamma$ is taken to be a piecewise-geodesic
interpolation of the output. We note that the interpolating knots
$\hat{\gamma}_k$ are represented discretely as a sum of diracs. How
``fixed endpoints'' are incorporated into the objective is detailed in Appendix \ref{sec:extensions}, but this means, in particular, that we take as known which batch among the $\hat\rho_{t_i}$'s is the \emph{start} and which is the \emph{end}. This allows us to identify whether the forward or backward ordering along the curve is the correct one, and represents a form of expert knowledge in the biological domain application context.

To evaluate the performance of our approach, we view it purely as a seriation method, and compare to two existing seriation methods that can be used in $W_2$ space: the \emph{Traveling Salesman Problem} (TSP) approach to seriation \cite{LAPORTE1978259_tsp_seriation}, and \emph{spectral seriation} \cite{atkins_spectral_1998, spectral_seriation_fogel_alexandre, natik2021consistencyspectralseriation}, which we outline in Appendix \ref{appn:extra-experiments}. These two methods take as input the matrix $[W_2(\hat\rho_{t_i}, \hat\rho_{t_j})]_{i,j}$ of pairwise $W_2$ distances between empirical measures $\hat\rho_{t_i}$.

To quantify performance of each method, we use a loss on the space of permutations of time labels. Specifically, we consider the error metric
$$\mathcal{E}(\hat{\tau}) = \frac{2}{T(T-1)} \sum_{i<j}\mathbf{1}(\hat{\tau}(\hat\rho_{t_i})>\hat{\tau}(\hat{\rho}_{t_j}))$$
which computes the percentage of pairs of non-identical time-indices whose order is reversed by the given pseudotime $\hat{\tau}$ for the set of $\rho_{t_i}$'s. This metric is a normalized version of \emph{Kendall's Tau distance} on the space of permutations; see \cite{diaconis1977spearman} for a discussion of this metric and its relation to other metrics on permutations.

\subsubsection*{Seriation on a Mouse Time Series}
In this section, we test seriation on a real scRNA-seq mouse dataset
with $43$ timepoints, comprising observations of mouse embryonic
development \cite{mouse_time_series}. In selecting a dataset, our
primary goal was to test our method in its most direct application:
the scaling of developmental time series collection. When scaling
scRNA-seq time-series for high numbers of timepoints, a critical
question arises: ``how many cells per timepoint do we need to recover
meaningful data''.
In this section, we explore the performance trade-off that occurs as the number of cells per timepoint changes. Particularly, we will compare the principal curves approach to TSP and spectral seriation, showing principal curves performs strongly in seriating scRNA-seq time series.

First, we discuss our reasoning for selecting the dataset from \cite{mouse_time_series} as our real-world exemplar. Generally, the most rapid changes in gene expression take place during embryonic development, making this the focus of many studies. Out of these, studies that sequence a single whole embryo per timepoint (single-embryo) are generally seen as a gold standard. Otherwise, due to asynchronous development speeds, having multiple embryos per timepoint can result in it being an aggregate of multiple developmental times. Additionally, due to the speed of developmental changes, having tightly spaced timepoints is critical. As such, the mouse dataset from  \cite{mouse_time_series} is a single-embryo timecourse with tightly-spaced timepoints from late gastrulation to birth. Embryos are staged by morphology, giving high confidence in the provided ordering to be ground truth. Finally, as mouse is a common model organism it serves as a useful example for groups wishing to apply seriation in their own collection pipeline.

To test seriation, we created a series of subsampled datasets in the high-timepoint, moderate cell counts context in which we expect principal curves to be applied. As the original dataset has approximately 12.4 million cells, testing all methods would be prohibitive on the full dataset. We created a smaller dataset of only $100{,}000$ cells on which we ran diffusion mapping for dimensionality reduction. To create our test datasets, we further downsampled the number of cells while the number of timepoints was fixed. This represents the number of cells per timepoint changing, as might be seen as a time series experiment scales up to higher numbers of timepoints. Specifically, we selected a range of total cells in $[600, 10{,}000]$ and produced $20$ independent downsamplings of the data at each time. To select parameters, we used a heuristic parameter selection for the principal curves method (Appendix \ref{appn:param_selection}) and a ground-truth parameter sweep for spectral seriation. Further details on the pre-processing, parameter selection steps, and competing methods are contained in Appendix \ref{appn:extra-experiments}. We then ran all the methods on the downsampled datasets, the results of which can be seen in Figure \ref{fig:mouse_comparison_sweep}. Principal curves performed strongly at all levels of downsampling. Our method outperformed the competing methods in terms of normalized Kendall Tau error at every downsampling level.

The strong performance of principal curves on this mouse dataset shows that it is a viable option for seriating large scRNA-seq time series. We provide additional results on simulated datasets in Appendix  \ref{appn:simulated_experiments}. There, we operate in a fixed cell budget mode and look at performance as the number of timepoints is scaled up to $N = 1000$. We first look at a dataset with very simple geometry where principal curves has low error, but spectral seriation is moderately better performing. After, we present a slightly more complex dataset simulating rapid cell developmental changes. In this setting, principal curves outperforms competing methods, demonstrating the data mode where the principal curves becomes performant.

\begin{figure}[H]
    \centering
    \includegraphics[width=\linewidth]{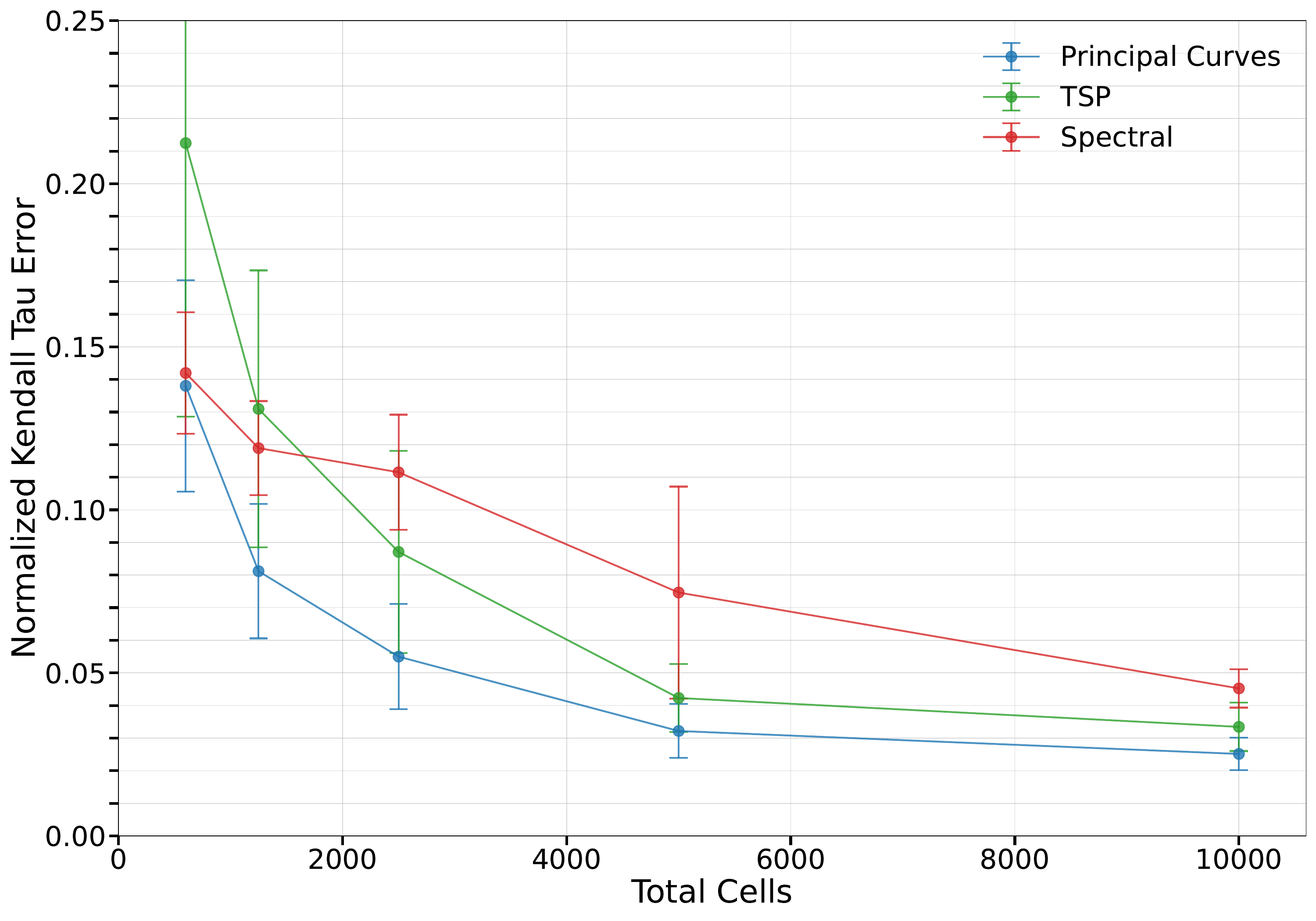}
    \caption{A performance comparison of three seriation methods on the mouse dataset. The total number of cells varies while the number of timepoints remains constant, showing performance at different levels of cells per timepoint. Values are given as the mean of normalized Kendall Tau error over $20$ trials involving different subsamplings of the data. Error bars indicate standard deviation over these trials. See Appendix \ref{appn:extra-experiments} for details on the pre-processing.}
    \label{fig:mouse_comparison_sweep}
\end{figure}

\section{Discussion}

We have introduced principal curves in general metric spaces, which we envision will enable high-frequency data collection for single-cell time-courses.
We propose an estimator and prove it is consistent for recovering a curve of probability measures in Wasserstein space from empirical samples.
This can be interpreted as a one-dimensional manifold learning problem, and is related to seriation~\cite{shah2016stochastically,flammarion2019optimal} and trajectory inference~\cite{lavenant2024towards}.

Ordering scRNA-seq datasets along a principal developmental curve can be thought of as an {\em unsupervised} version of the trajectory inference (TI) problem
\cite{schiebinger2021opinion,lavenant2024towards,chizat2022trajectory}.
In TI, one observes empirical marginals $\hat\rho_{t_i}$ from independent snapshots of a developmental stochastic process $\rho$, and one wishes to recover the {\em law on paths} for $\rho$.

In future work, a combined statistical analysis of our principal curve problem and {trajectory inference} could shed light on optimal experimental design. We anticipate there would be a statistical trade-off between the sub-problems of principal curves (where more samples per time-point are ideal), and trajectory inference (where more time-points are ideal). We intend to investigate these considerations as part of future work involving real biological data.

A statistical analaysis of the principal curves problem would require
 a quantitative version of Theorem \ref{thm:main-intro}. This would necessarily involve a careful finite-sample analysis in terms of $N,M,K,$ and $\beta$, but may also require making the \emph{assumptions} of the theorem more quantitative. In particular, Theorem \ref{thm:main-intro} assumes that the ground truth curve $\rho$ is injective; it is plausible that a quantitative consistency result would depend on some \emph{quantitative injectivity} notion for $\rho$. Multiple distinct notions of quantitative injectivity have been developed in the geometric measure theory literature, such as ``reach'' \cite{aamari2019reach, federer1969geometric} and ``Ahlfors-David regularity'' \cite{david1993analysis}; and so one possible research question is to determine which quantitative injectivity notion is most relevant for principal curves.}

Another worthwhile direction for future work would be further analysis
of Algorithm \ref{alg:coupled-lloyd}. When $\beta$ is set equal to
zero, this algorithm reduces to the usual Lloyd's algorithm for
k-means clustering; however the convergence behavior of Lloyd's
algorithm itself is known to be subtle \cite{du2006convergence}. In
some highly structured cases, it is nonetheless known that Lloyd's
algorithm converges to a global minimizer, such as the case of the
$N\rightarrow \infty$ limit for an underlying data distribution which
is univariate and log-concave \cite{trushkin1982sufficient}.
Accordingly, one might ask whether, for certain ``special'' ground truth curves $\rho$, Algorithm \ref{alg:coupled-lloyd} converges to the global minimizer of the $\textrm{PPC}^K$ objective, and thus to $\rho$ itself in the limit.

Finally, we conclude with a summary of several minor variants of the principal curve problem, including {\em loops, multiple curves,} and {\em spline interpolation}.

\subsection{Variants of principal curves}
In this section we consider several variants on the basic inference
problem discussed in the bulk of the article.

\emph{Loops}. In the case where $\Lambda$ is believed to be concentrated
near a \emph{loop, }i.e. a closed curve, it makes sense to pose the
principal curve problem over the space of closed curves, that is,
we minimize over the set
\[
\{\gamma\in AC([0,1];X):\gamma_{0}=\gamma_{1}\}.
\]
While this situation is atypical for the case where $\Lambda$ comes
from the marginals of a drift-diffusion SDE, say, there are circumstances in cellular
dynamics/single-cell transcriptomics where the cells follow an approximately
periodic dynamics (in particular ``cell cycle'' \cite{liu2017reconstructing}),
and so unsupervised inference of a closed curve is required. Our variational
framework can easily accommodate this case: similarly to the case of
fixed endpoints, the set $\{\gamma\in AC([0,1];X):\gamma_{0}=\gamma_{1}\}$
is closed inside $AC([0,1];X)$. At the continuum level, one obtains
the optimization problem
\[
\inf_{\gamma\in AC([0,1];X)}\left\{ \int_{X}d^{2}(x,\Gamma)d\Lambda(x)+\beta\text{Length}(\gamma):\gamma_{0}=\gamma_{1}\right\} ;
\]
at the discrete level, where one has knots $\{\gamma_{k}\}_{k=1}^{K}$
instead of a continuous curve $\gamma$, we simply replace the ``discrete
length'' $\sum_{k=1}^{K-1}d(\gamma_{k},\gamma_{k+1})$ with $d(\gamma_{K},\gamma_{1})+\sum_{k=1}^{K-1}d(\gamma_{k},\gamma_{k+1})$,
and at each Lloyd-iteration we label the knots by computing a Hamiltonian
cycle through $\{\gamma_{k}\}_{k=1}^{K}$ rather than solving the
travelling salesman problem.

\emph{Multiple curves}. In the Euclidean setting, \cite{kirov2017multiple}
proposed a modification of the principal curves problem where one
optimizes over finite collections of curves rather than a single curve,
with a penalty on the number of curves. In other words, one considers
the optimization problem
\[
\inf_{K\in\mathbb{N},\{\gamma^{k}\}_{k=1}^{K}\in(AC([0,1];X))^{K}}\left\{ \left(\int_{X}d^{2}\left(x,\bigcup_{k=1}^{K}\Gamma^{k}\right)d\Lambda(x)+\beta\sum_{k=1}^{K}\text{Length}(\gamma^{k})\right)+\beta^{\prime}K\right\} .
\]
Here $\bigcup_{k=1}^{K}\Gamma^{k}$ denotes the union of the graphs
of the curves $\gamma^{k}$. We remark that the penalty on $\{\gamma^{k}\}$
(total length plus number of components) is distinct from penalizing
the total variation of the collection $\{\gamma^{k}\}$, and this
turns out to ensure that minimizers have a finite number of components
(which is not obvious with TV penalization).

Briefly, \cite{kirov2017multiple} observes that proving the existence
of minimizers for the ``multiple penalized principal curves'' problem
is routine: indeed, along any minimizing sequence, the number of components
is necessarily bounded, so one can pass to the limit by extracting
a limiting curve for each component separately. Additionally, \cite{kirov2017multiple}
propose numerics for their multiple curves problem which are similar
to ours (albeit without proof of consistency of the scheme). Here
we simply point out that the minimization problem and numerics described
in \cite{kirov2017multiple} also make sense in our more general setting
of a compact geodesic metric space, in particular their proof of existence
of minimizers can be translated to our setting without modification.

\emph{Spline interpolation}.  The algorithm we have proposed for producing
minimizers to the discrete variational problem $\text{PPC}^{K}$ gives
us a $K$-tuple of points $\{\gamma_{k}\}_{k=1}^{K}$ as its output.
If we prefer the output to be an AC curve instead, one natural choice
is to take a piecewise-geodesic interpolation of $\{\gamma_{k}\}_{k=1}^{K}$.

In the case where $X$ is a Euclidean domain, we also have access
to splines as an alternate way to interpolate between the points $\{\gamma_{k}\}_{k=1}^{K}$.
Cubic splines, in particular, have the interpretation \cite{de1963best} as curves of \emph{least total squared acceleration} fitting the prescribed
points: that is, one solves the optimization problem
\[
\inf_{\gamma:\in C^{2}([0,1];X)}\left\{ \int|\ddot{\gamma}_{t}|^{2}dt:\gamma_{t_{k}}=\bar{\gamma}_{t_{k}}\right\}
\]
for prescribed points $\{\bar{\gamma}_{t_{k}}\}_{k=1}^{K}$. This
optimization problem can be posed more generally on a space where
we can make sense of the magnitude of acceleration $|\ddot{\gamma}_{t}|$
of a curve $\gamma$, for instance one can take $X$ to be a Riemannian
manifold. General geodesic metric spaces do not have enough differential
structure for $|\ddot{\gamma}_{t}|$ to be well-defined, however.

Nonetheless, in the case of the space of probability measures equipped
with the 2-Wasserstein metric, which is of particular interest to
us, several versions of cubic splines have been proposed in the literature,
together with effective numerics \cite{banerjee2025efficient, benamou2019second, chen2018measure, chewi2021fast, justiniano2023approximation}. In particular we direct readers to
\cite{justiniano2023approximation} which provides an up to date review of related
literature.

\section*{Acknowledgements}
AW, AA and GS were supported by the Burroughs Wellcome Fund and a CIHR
Project Grant. AW, AA, GS, and YHK were partially supported by the
Exploration Grant NFRFE-2019-00944 from the New Frontiers in Research
Fund (NFRF). FK was supported by the UBC Four Year Doctoral Fellowship
(4YF). In addition, FK is supported by NSF RTG grant DMS-2136198. AW,
FK, and YHK also conducted portions of this research as guests of the
Stochastic Analysis and Application Research Center (SAARC) at the
Korea Advanced Institute of Science and Technology (KAIST), and
gratefully acknowledge SAARC's hospitality. AW thanks Dejan Slepčev
for bringing the articles \cite{lu2016average, kirov2017multiple} to
our attention, and thanks Danica Sutherland for helpful discussions on
reproducing kernel Hilbert spaces, especially for directing us to
\cite{wynne2022kernel}. AW additionally thanks Filippo Santambrogio
for helpful discussions regarding his article
\cite{iacobelli2019weighted}, in particular for explaining to us that
the results of that article would not be directly applicable for our
work. {YHK is partially supported by the Natural Sciences and
  Engineering Research Council of Canada (NSERC), with Discovery Grant
  RGPIN-2019-03926. This research is also partially supported by the
  Pacific Institute for the Mathematical Sciences (PIMS), through the
  PIMS Research Network (PRN) program (PIMS-PRN01), the Kantorovich
  Initiative (KI).
}

\section*{Code Availability} The code required to reproduce the results in this article can be found at: \hyperlink{https://github.com/antonafana/wasserstein\_principal\_curves}{https://github.com/antonafana/wasserstein\_principal\_curves}.

\bibliographystyle{plain}
\bibliography{references}

@article{varadarajan1958,
 ISSN = {00364452},
 URL = {http://www.jstor.org/stable/25048365},
 author = {V. S. Varadarajan},
 journal = {Sankhyā: The Indian Journal of Statistics (1933-1960)},
 number = {1/2},
 pages = {23--26},
 publisher = {Springer},
 title = {On the Convergence of Sample Probability Distributions},
 urldate = {2024-09-03},
 volume = {19},
 year = {1958}
}

@article{Biau2011Oct,
	author = {Biau, G{\ifmmode\acute{e}\else\'{e}\fi}rard and Fischer, Aur{\ifmmode\acute{e}\else\'{e}\fi}lie},
	title = {{Parameter Selection for Principal Curves}},
	journal = {IEEE Trans. Inf. Theory},
	volume = {58},
	number = {3},
	pages = {1924--1939},
	year = {2011},
	month = oct,
	publisher = {IEEE},
	doi = {10.1109/TIT.2011.2173157}
}

@book{ambrosio2004topics,
  title={Topics on Analysis in Metric Spaces},
  author={Ambrosio, L. and Tilli, P.},
  isbn={9780198529385},
  lccn={2004298920},
  series={Oxford lecture series in mathematics and its applications},
  url={https://books.google.ca/books?id=BLtvtjEfDDoC},
  year={2004},
  publisher={Oxford University Press}
}

@book{ambrosio2008gradient,
  title={Gradient flows: in metric spaces and in the space of probability measures},
  author={Ambrosio, Luigi and Gigli, Nicola and Savar{\'e}, Giuseppe},
  year={2008},
  publisher={Springer Science \& Business Media}
}

@incollection{ambrosio2013user,
  title={A user's guide to optimal transport},
  author={Ambrosio, Luigi and Gigli, Nicola},
  booktitle={Modelling and optimisation of flows on networks},
  pages={1--155},
  year={2013},
  publisher={Springer}
}

@article{benamou2019second,
  title={Second-order models for optimal transport and cubic splines on the Wasserstein space},
  author={Benamou, Jean-David and Gallou{\"e}t, Thomas O and Vialard, Fran{\c{c}}ois-Xavier},
  journal={Foundations of Computational Mathematics},
  volume={19},
  pages={1113--1143},
  year={2019},
  publisher={Springer}
}

@article{boissard2014mean,
  title={On the mean speed of convergence of empirical and occupation measures in {W}asserstein distance},
  author={Boissard, Emmanuel and Le Gouic, Thibaut},
  journal={{Annales de l'Institut Henri Poincar{\'e} (B) Probabilit{\'e}s et Statistiques}},
  volume={50},
  number={2},
  pages={539--563},
  year={2014}
}

@book{braides2002gamma,
  title={{$\Gamma$}-convergence for beginners},
  author={Braides, Andrea},
  series={of Oxford Lecture Series in Mathematics and its Applications},
  year={2002},
  volume={22},
  publisher={Oxford University Press}
}

@article{chen2018measure,
  title={Measure-valued spline curves: An optimal transport viewpoint},
  author={Chen, Yongxin and Conforti, Giovanni and Georgiou, Tryphon T},
  journal={SIAM Journal on Mathematical Analysis},
  volume={50},
  number={6},
  pages={5947--5968},
  year={2018},
  publisher={SIAM}
}

@inproceedings{chewi2021fast,
  title={Fast and smooth interpolation on wasserstein space},
  author={Chewi, Sinho and Clancy, Julien and Le Gouic, Thibaut and Rigollet, Philippe and Stepaniants, George and Stromme, Austin},
  booktitle={International Conference on Artificial Intelligence and Statistics},
  pages={3061--3069},
  year={2021},
  organization={PMLR}
}

@article{de1963best,
  title={Best approximation properties of spline functions of odd degree},
  author={De Boor, Carl},
  journal={Journal of Mathematics and Mechanics},
  pages={747--749},
  year={1963}
}

@book{degiorgi2013selected,
  title={Selected Papers},
  author={De Giorgi, Ennio},
  isbn={9783642403798},
  series={Springer Collected Works in Mathematics},
  url={https://books.google.fr/books?id=QsOrngEACAAJ},
  year={2013},
  publisher={Springer Berlin Heidelberg}
}

@article{hauberg2015principal,
  title={Principal curves on Riemannian manifolds},
  author={Hauberg, S{\o}ren},
  journal={IEEE transactions on pattern analysis and machine intelligence},
  volume={38},
  number={9},
  pages={1915--1921},
  year={2015},
  publisher={IEEE}
}

@article{jordan1998variational,
  title={The variational formulation of the {F}okker--{P}lanck equation},
  author={Jordan, Richard and Kinderlehrer, David and Otto, Felix},
  journal={SIAM Journal on Mathematical Analysis},
  volume={29},
  number={1},
  pages={1--17},
  year={1998},
  publisher={SIAM}
}

@article{justiniano2023approximation,
  title={Approximation of Splines in Wasserstein Spaces},
  author={Justiniano, Jorge and Rumpf, Martin and Erbar, Matthias},
journal={ESAIM: Control, Optimisation and Calculus of Variations},
  volume={30},
  pages={64},
  year={2024},
  publisher={EDP Sciences}
}

@book{kechris2012classical,
  title={Classical descriptive set theory},
  author={Kechris, Alexander},
  volume={156},
  year={2012},
  publisher={Springer Science \& Business Media}
}

@article{kirov2017multiple,
  title={Multiple penalized principal curves: Analysis and computation},
  author={Kirov, Slav and Slep{\v{c}}ev, Dejan},
  journal={Journal of Mathematical Imaging and Vision},
  volume={59},
  pages={234--256},
  year={2017},
  publisher={Springer}
}

@book{kolmogorov1999elements,
  title={Elements of the Theory of Functions and Functional Analysis},
  author={Kolmogorov, Andrey and Fomin, Sergey},
  isbn={9780486406831},
  url={https://books.google.ca/books?id=OyWeDwfQmeQC},
  year={1999},
  publisher={Dover}
}

@article{lavenant2024towards,
author = {Hugo Lavenant and Stephen Zhang and Young-Heon Kim and Geoffrey Schiebinger},
title = {{Toward a mathematical theory of trajectory inference}},
volume = {34},
journal = {The Annals of Applied Probability},
number = {1A},
publisher = {Institute of Mathematical Statistics},
pages = {428 -- 500},
year = {2024},
doi = {10.1214/23-AAP1969},
URL = {https://doi.org/10.1214/23-AAP1969}
}

@article{lu2016average,
  title={Average-distance problem for parameterized curves},
  author={Lu, Xin Yang and Slep{\v{c}}ev, Dejan},
  journal={ESAIM: Control, Optimisation and Calculus of Variations},
  volume={22},
  number={2},
  pages={404--416},
  year={2016},
  publisher={EDP Sciences}
}

@article{peyre2019computational,
  title={Computational Optimal Transport: With Applications to Data Science},
  author={Peyr{\'e}, Gabriel and Cuturi, Marco},
  journal={Foundations and Trends in Machine Learning},
  volume={11},
  number={5-6},
  pages={355--607},
  year={2019},
  publisher={Now Publishers, Inc.}
}

@book{santambrogio2015optimal,
  title={Optimal transport for applied mathematicians},
  author={Santambrogio, Filippo},
  year={2015},
  publisher={Springer}
}

@article{schiebinger2019optimal,
  title={Optimal-transport analysis of single-cell gene expression identifies developmental trajectories in reprogramming},
  author={Schiebinger, Geoffrey and others},
  journal={Cell},
  volume={176},
  number={4},
  pages={928--943},
  year={2019},
  publisher={Elsevier}
}

@book{villani2008optimal,
  title={Optimal transport: old and new},
  author={Villani, C{\'e}dric},
  volume={338},
  year={2008},
  publisher={Springer Science \& Business Media}
}

@article{weed2019sharp,
  title={Sharp asymptotic and finite-sample rates of convergence of empirical measures in {W}asserstein distance},
  author={Weed, Jonathan and Bach, Francis},
  journal={Bernoulli},
  volume={25},
  number={4A},
  pages={2620--2648},
  year={2019},
  publisher={Bernoulli Society for Mathematical Statistics and Probability}
}

@article{diaconis1977spearman,
 ISSN = {00359246},
 URL = {http://www.jstor.org/stable/2984804},
 author = {Persi Diaconis and R. L. Graham},
 journal = {Journal of the Royal Statistical Society. Series B (Methodological)},
 number = {2},
 pages = {262--268},
 publisher = {[Royal Statistical Society, Wiley]},
 title = {Spearman's Footrule as a Measure of Disarray},
 urldate = {2024-05-07},
 volume = {39},
 year = {1977}
}

@phdthesis{hastie1984,
  author  = {Hastie, Trevor},
  title   = {Principal Curves and Surfaces},
  school  = {Stanford University},
  year    = {1984}
}

@article{hastie1989,
	title = {Principal {Curves}},
	volume = {84},
	issn = {0162-1459},
	url = {https://www.jstor.org/stable/2289936},
	doi = {10.2307/2289936},
	number = {406},
	urldate = {2023-11-04},
	journal = {Journal of the American Statistical Association},
	author = {Hastie, Trevor and Stuetzle, Werner},
	year = {1989},
	note = {Publisher: [American Statistical Association, Taylor \& Francis, Ltd.]},
	pages = {502--516},
}

@article{kegl2000,
	title = {Learning and design of principal curves},
	volume = {22},
	issn = {01628828},
	url = {http://ieeexplore.ieee.org/document/841759/},
	doi = {10.1109/34.841759},
	number = {3},
	urldate = {2023-11-09},
	journal = {IEEE Transactions on Pattern Analysis and Machine Intelligence},
    author={K{\'e}gl, Bal{\'a}zs and Krzyzak, Adam and Linder, Tam{\'a}s and Zeger, Kenneth},
	year = {2000},
	pages = {281--297}
}

@article{delattre2020,
  title = {On Principal Curves with a Length Constraint},
  author = {Delattre, Sylvain and Fischer, Aur{\'e}lie},
  year = {2020},
  month = aug,
  journal ={{Annales de l'Institut Henri Poincar{\'e} (B) Probabilit{\'e}s et Statistiques}},
  volume = {56},
  number = {3},
  pages = {2108--2140},
  publisher = {{Institut Henri Poincar\'e}},
  issn = {0246-0203},
  doi = {10.1214/19-AIHP1030},
  urldate = {2023-10-05}
}

@inproceedings{seguy2015principal,
  title={Principal geodesic analysis for probability measures under the optimal transport metric},
  author={Seguy, Vivien and Cuturi, Marco},
  booktitle={Proceedings of the 28th International Conference on Neural Information Processing Systems-Volume 2},
  pages={3312--3320},
  year={2015}
}

@article{bigot2017geodesic,
  TITLE = {{Geodesic PCA in the Wasserstein space by convex PCA}},
  AUTHOR = {Bigot, J{\'e}r{\'e}mie and Gouet, Raul and Klein, Thierry and Lopez, Alfredo},
  URL = {https://hal.science/hal-01978864},
  JOURNAL = {{Annales de l'Institut Henri Poincar{\'e} (B) Probabilit{\'e}s et Statistiques}},
  PUBLISHER = {{Institut Henri Poincar{\'e} (IHP)}},
  VOLUME = {53},
  NUMBER = {1},
  PAGES = {1-26},
  YEAR = {2017},
  DOI = {10.1214/15-aihp706},
  HAL_ID = {hal-01978864},
  HAL_VERSION = {v1},
}

@article{hamm2023manifold,
  title={Manifold learning in {W}asserstein space},
  author={Hamm, Keaton and Moosm{\"u}ller, Caroline and Schmitzer, Bernhard and Thorpe, Matthew},
  journal={SIAM Journal on Mathematical Analysis},
  volume={57},
  number={3},
  pages={2983--3029},
  year={2025},
  publisher={SIAM}
}

@article{spectral_seriation_fogel_alexandre,
  author  = {Fajwel Fogel and Alexandre d'Aspremont and Milan Vojnovic},
  title   = {Spectral Ranking using Seriation},
  journal = {Journal of Machine Learning Research},
  year    = {2016},
  volume  = {17},
  number  = {88},
  pages   = {1--45},
  url     = {http://jmlr.org/papers/v17/16-035.html}
}

@article{kobayashi2024monge,
  title={Monge-Kantorovich Fitting With Sobolev Budgets},
  author={Kobayashi, Forest and Hayase, Jonathan and Kim, Young-Heon},
  journal={arXiv preprint arXiv:2409.16541},
  year={2024}
}

@article{bartlett2002rademacher,
  title={Rademacher and Gaussian complexities: Risk bounds and structural results},
  author={Bartlett, Peter L and Mendelson, Shahar},
  journal={Journal of Machine Learning Research},
  volume={3},
  number={Nov},
  pages={463--482},
  year={2002}
}

@book{wainwright2019high,
  title={High-dimensional statistics: A non-asymptotic viewpoint},
  author={Wainwright, Martin J},
  volume={48},
  year={2019},
  publisher={Cambridge university press}
}

@book{applegate2006traveling_concorde,
  title={The traveling salesman problem: a computational study},
  author={Applegate, David L},
  volume={17},
  year={2006},
  publisher={Princeton university press}
}

@article{LAPORTE1978259_tsp_seriation,
title = {The seriation problem and the travelling salesman problem},
journal = {Journal of Computational and Applied Mathematics},
volume = {4},
number = {4},
pages = {259-268},
year = {1978},
issn = {0377-0427},
doi = {https://doi.org/10.1016/0771-050X(78)90024-4},
url = {https://www.sciencedirect.com/science/article/pii/0771050X78900244},
author = {Gilbert Laporte}
}

@article{kim2024optimal,
  title={Optimal sequencing depth for single-cell RNA-sequencing in Wasserstein space},
  author={Kim, Jakwang and Kubal, Sharvaj and Schiebinger, Geoffrey},
  journal={arXiv preprint arXiv:2409.14326},
  year={2024}
}

@article{santambrogio2017euclidean,
  title={$\{$Euclidean, metric, and Wasserstein$\}$ gradient flows: an overview},
  author={Santambrogio, Filippo},
  journal={Bulletin of Mathematical Sciences},
  volume={7},
  pages={87--154},
  year={2017},
  publisher={Springer}
}

@article{wynne2022kernel,
  author  = {George Wynne and Andrew B. Duncan},
  title   = {A Kernel Two-Sample Test for Functional Data},
  journal = {Journal of Machine Learning Research},
  year    = {2022},
  volume  = {23},
  number  = {73},
  pages   = {1--51},
  url     = {http://jmlr.org/papers/v23/20-1180.html}
}

@article{kim2017wasserstein,
  title={Wasserstein barycenters over Riemannian manifolds},
  author={Kim, Young-Heon and Pass, Brendan},
  journal={Advances in Mathematics},
  volume={307},
  pages={640--683},
  year={2017},
  publisher={Elsevier}
}

@inproceedings{cuturi2014fast,
  title={Fast computation of Wasserstein barycenters},
  author={Cuturi, Marco and Doucet, Arnaud},
  booktitle={International conference on machine learning},
  pages={685--693},
  year={2014},
  organization={PMLR}
}

@article{chizat2023doubly,
  title={Doubly regularized entropic Wasserstein barycenters},
  author={Chizat, L{\'e}na{\"\i}c},
  journal={Foundations of Computational Mathematics},
  pages={1--38},
  year={2025},
  publisher={Springer}
}

@article{agueh2011barycenters,
  title={Barycenters in the Wasserstein space},
  author={Agueh, Martial and Carlier, Guillaume},
  journal={SIAM Journal on Mathematical Analysis},
  volume={43},
  number={2},
  pages={904--924},
  year={2011},
  publisher={SIAM}
}

@article{burago2001course,
  title={A course in metric geometry},
  author={Burago, Dmitri and Burago, Yuri and Ivanov, Sergei},
  journal={American Mathematical Society},
  year={2001}
}

@InProceedings{catalano2024hierarchical,
  title = 	 {Hierarchical Integral Probability Metrics: A distance on random probability measures with low sample complexity},
  author =       {Catalano, Marta and Lavenant, Hugo},
  booktitle = 	 {Proceedings of the 41st International Conference on Machine Learning},
  pages = 	 {5841--5861},
  year = 	 {2024},
  volume = 	 {235},
  series = 	 {Proceedings of Machine Learning Research},
  month = 	 {21--27 Jul},
  publisher =    {PMLR},
  url = 	 {https://proceedings.mlr.press/v235/catalano24a.html}
}

@article{smola2001regularized,
  title={Regularized principal manifolds},
  author={Smola, Alexander J and Mika, Sebastian and Sch{\"o}lkopf, Bernhard and Williamson, Robert C},
  journal={Journal of machine learning research},
  volume={1},
  pages={179--209},
  year={2001},
  publisher={JMLR. org}
}

@article{slepvcev2014counterexample,
  title={Counterexample to regularity in average-distance problem},
  author={Slep{\v{c}}ev, Dejan},
  journal={Annales de l'IHP Analyse non lin{\'e}aire},
  volume={31},
  number={1},
  pages={169--184},
  year={2014}
}

@article{natik2021consistencyspectralseriation,
  title={Consistency of spectral seriation},
  author={Natik, Amine and Smith, Aaron},
  journal={ESAIM: Probability and Statistics},
  volume={30},
  pages={120--161},
  year={2026},
  publisher={EDP Sciences}
}

@article{duchamp1996extremal,
author = {Tom Duchamp and Werner Stuetzle},
title = {{Extremal properties of principal curves in the plane}},
volume = {24},
journal = {The Annals of Statistics},
number = {4},
publisher = {Institute of Mathematical Statistics},
pages = {1511 -- 1520},
year = {1996},
doi = {10.1214/aos/1032298280},
URL = {https://doi.org/10.1214/aos/1032298280}
}

@article{atkins_spectral_1998,
	title = {A {Spectral} {Algorithm} for {Seriation} and the {Consecutive} {Ones} {Problem}},
	volume = {28},
	url = {https://doi.org/10.1137/S0097539795285771},
	doi = {10.1137/S0097539795285771},
	number = {1},
	journal = {SIAM Journal on Computing},
	author = {Atkins, Jonathan E. and Boman, Erik G. and Hendrickson, Bruce},
	year = {1998},
	note = {\_eprint: https://doi.org/10.1137/S0097539795285771},
	pages = {297--310},
}

@article{gerber2013principal,
  author  = {Samuel Gerber and Ross Whitaker},
  title   = {Regularization-Free Principal Curve Estimation},
  journal = {Journal of Machine Learning Research},
  year    = {2013},
  volume  = {14},
  number  = {39},
  pages   = {1285--1302},
  url     = {http://jmlr.org/papers/v14/gerber13a.html}
}

@article{tibshirani1992principal,
  title={Principal curves revisited},
  author={Tibshirani, Robert},
  journal={Statistics and computing},
  volume={2},
  pages={183--190},
  year={1992},
  publisher={Springer}
}

@article{polak2007lazy,
     author = {Polak, Paz and Wolansky, Gershon},
     title = {The lazy travelling salesman problem in $\mathbb {R}^2$},
     journal = {ESAIM: Control, Optimisation and Calculus of Variations},
     pages = {538--552},
     publisher = {EDP-Sciences},
     volume = {13},
     number = {3},
     year = {2007},
     doi = {10.1051/cocv:2007025},
     mrnumber = {2329175},
     language = {en},
     url = {http://www.numdam.org/articles/10.1051/cocv:2007025/}
}

@article{bientinesi_parallel_2005,
	title = {A {Parallel} {Eigensolver} for {Dense} {Symmetric} {Matrices} {Based} on {Multiple} {Relatively} {Robust} {Representations}},
	volume = {27},
	url = {https://doi.org/10.1137/030601107},
	doi = {10.1137/030601107},
	number = {1},
	journal = {SIAM Journal on Scientific Computing},
	author = {Bientinesi, Paolo and Dhillon, Inderjit S. and van de Geijn, Robert A.},
	year = {2005},
	note = {\_eprint: https://doi.org/10.1137/030601107},
	pages = {43--66},
}

@article{marek_elpa_2014,
	title = {The {ELPA} library: scalable parallel eigenvalue solutions for electronic structure theory and computational science},
	volume = {26},
	url = {https://dx.doi.org/10.1088/0953-8984/26/21/213201},
	doi = {10.1088/0953-8984/26/21/213201},
	number = {21},
	journal = {Journal of Physics: Condensed Matter},
	author = {Marek, A. and Blum, V. and Johanni, R. and Havu, V. and Lang, B. and Auckenthaler, T. and Heinecke, A. and Bungartz, H.-J. and Lederer, H.},
	month = may,
	year = {2014},
	note = {Publisher: IOP Publishing},
	pages = {213201},
}

@article{sandilya2002principal,
  title={Principal curves with bounded turn},
  author={Sandilya, Sathyakama and Kulkarni, Sanjeev R},
  journal={IEEE Transactions on Information Theory},
  volume={48},
  number={10},
  pages={2789--2793},
  year={2002},
  publisher={IEEE}
}

@article{lu2021average,
  title={Average-distance problem with curvature penalization for data parameterization: regularity of minimizers},
  author={Lu, Xin Yang and Slep{\v{c}}ev, Dejan},
  journal={ESAIM: Control, Optimisation and Calculus of Variations},
  volume={27},
  pages={8},
  year={2021},
  publisher={EDP Sciences}
}

@article{lu2015regularity,
  title={Regularity of densities in relaxed and penalized average distance problem},
  author={Lu, Xin Yang},
  journal={Networks and Heterogeneous Media},
  volume={10},
  number={4},
  pages={837--855},
  year={2015},
  publisher={Networks and Heterogeneous Media}
}

@article{ozertem2011locally,
  title={Locally defined principal curves and surfaces},
  author={Ozertem, Umut and Erdogmus, Deniz},
  journal={The Journal of Machine Learning Research},
  volume={12},
  pages={1249--1286},
  year={2011},
  publisher={JMLR. org}
}

@article{genovese2012geometry,
  title={The geometry of nonparametric filament estimation},
  author={Genovese, Christopher R and Perone-Pacifico, Marco and Verdinelli, Isabella and Wasserman, Larry},
  journal={Journal of the American Statistical Association},
  volume={107},
  number={498},
  pages={788--799},
  year={2012},
  publisher={Taylor \& Francis}
}

@article{genovese2014nonparametric,
  title={Nonparametric ridge estimation},
  author={Genovese, Christopher R and Perone-Pacifico, Marco and Verdinelli, Isabella and Wasserman, Larry},
  journal={The Annals of Statistics},
  pages={1511--1545},
  year={2014},
  publisher={JSTOR}
}

@article{chen2015asymptotic,
  title={Asymptotic theory for density ridges},
  author={Chen, Yen-Chi and Genovese, Christopher R. and Wasserman, Larry},
  journal={Annals of Statistics},
  volume={43},
  number={5},
  pages={1896--1928},
  year={2015},
  publisher={Institute of Mathematical Statistics}
}

@article{karimi2021regression,
  title={Regression analysis of distributional data through multi-marginal optimal transport},
  author={Karimi, Amirhossein and Georgiou, Tryphon T},
  journal={arXiv preprint arXiv:2106.15031},
  year={2021}
}

@article{chen2023wasserstein,
  title={Wasserstein regression},
  author={Chen, Yaqing and Lin, Zhenhua and M{\"u}ller, Hans-Georg},
  journal={Journal of the American Statistical Association},
  volume={118},
  number={542},
  pages={869--882},
  year={2023},
  publisher={Taylor \& Francis}
}

@InProceedings{pmlr-v31-poczos13a,
  title = 	 {Distribution-Free Distribution Regression},
  author = 	 {Poczos, Barnabas and Singh, Aarti and Rinaldo, Alessandro and Wasserman, Larry},
  booktitle = 	 {Proceedings of the Sixteenth International Conference on Artificial Intelligence and Statistics},
  pages = 	 {507--515},
  year = 	 {2013},
  editor = 	 {Carvalho, Carlos M. and Ravikumar, Pradeep},
  volume = 	 {31},
  series = 	 {Proceedings of Machine Learning Research},
  address = 	 {Scottsdale, Arizona, USA},
  month = 	 {29 Apr--01 May},
  publisher =    {PMLR},
  url = 	 {https://proceedings.mlr.press/v31/poczos13a.html}
}

@article{ghodrati2022distribution,
  title={Distribution-on-distribution regression via optimal transport maps},
  author={Ghodrati, Laya and Panaretos, Victor M},
  journal={Biometrika},
  volume={109},
  number={4},
  pages={957--974},
  year={2022},
  publisher={Oxford University Press}
}

@inproceedings{zhuang2022wasserstein,
  title={Wasserstein K-means for clustering probability distributions},
  author={Zhuang, Yubo and Chen, Xiaohui and Yang, Yun},
  booktitle={Proceedings of the 36th International Conference on Neural Information Processing Systems},
  pages={11382--11395},
  year={2022}
}

@inproceedings{rao2020wasserstein,
  title={Wasserstein K-Means for Clustering Tomographic Projections},
  author={Rao, Rohan and Moscovich, Amit and Singer, Amit},
  booktitle={Advances in neural information processing systems},
  year={2020}
}

@article{verdinelli2019hybrid,
  title={Hybrid Wasserstein distance and fast distribution clustering},
  author={Verdinelli, Isabella and Wasserman, Larry},
  journal={Electronic Journal of Statistics},
  volume={13},
  pages={5088--5119},
  year={2019}
}

@article{cazelles2018geodesic,
  title={Geodesic PCA versus log-PCA of histograms in the Wasserstein space},
  author={Cazelles, Elsa and Seguy, Vivien and Bigot, J{\'e}r{\'e}mie and Cuturi, Marco and Papadakis, Nicolas},
  journal={SIAM Journal on Scientific Computing},
  volume={40},
  number={2},
  pages={B429--B456},
  year={2018},
  publisher={SIAM}
}

@article{delicado2001another,
  title={Another look at principal curves and surfaces},
  author={Delicado, Pedro},
  journal={Journal of Multivariate Analysis},
  volume={77},
  number={1},
  pages={84--116},
  year={2001},
  publisher={Elsevier}
}

@book{chung1997spectral,
  title={Spectral graph theory},
  author={Chung, Fan RK},
  series={Regional Conference Series in Mathematics},
  volume={92},
  year={1997},
  publisher={American Mathematical Society}
}

@article{iacobelli2019weighted,
  title={Weighted ultrafast diffusion equations: from well-posedness to long-time behaviour},
  author={Iacobelli, Mikaela and Patacchini, Francesco S and Santambrogio, Filippo},
  journal={Archive for Rational Mechanics and Analysis},
  volume={232},
  pages={1165--1206},
  year={2019},
  publisher={Springer}
}

@article{karimi2020statistical,
  title={Statistical learning in Wasserstein space},
  author={Karimi, Amirhossein and Ripani, Luigia and Georgiou, Tryphon T},
  journal={IEEE Control Systems Letters},
  volume={5},
  number={3},
  pages={899--904},
  year={2020},
  publisher={IEEE}
}

@article{chizat2022trajectory,
  title={Trajectory inference via mean-field langevin in path space},
  author={Chizat, L{\'e}na{\"\i}c and Zhang, Stephen and Heitz, Matthieu and Schiebinger, Geoffrey},
  journal={Advances in Neural Information Processing Systems},
  volume={35},
  pages={16731--16742},
  year={2022}
}

@article{wang2013linear,
  title={A linear optimal transportation framework for quantifying and visualizing variations in sets of images},
  author={Wang, Wei and Slep{\v{c}}ev, Dejan and Basu, Saurav and Ozolek, John A and Rohde, Gustavo K},
  journal={International journal of computer vision},
  volume={101},
  pages={254--269},
  year={2013},
  publisher={Springer}
}

@article{robinson1951method,
  title={A method for chronologically ordering archaeological deposits},
  author={Robinson, William S},
  journal={American antiquity},
  volume={16},
  number={4},
  pages={293--301},
  year={1951},
  publisher={Cambridge University Press}
}

@article{kendall1963statistical,
  title={A statistical approach to Flinders Petrie's sequence-dating},
  author={Kendall, David G},
  journal={Bulletin of the International Statistical Institute},
  volume={40},
  number={2},
  pages={657--681},
  year={1963},
  publisher={INT STATISTICAL INSTITUTE 428 PRINSES BEATRIXLAEN, VOORBURG, NETHERLANDS}
}

@article{kendall1969incidence,
  title={Incidence matrices, interval graphs and seriation in archeology},
  author={Kendall, David},
  journal={Pacific Journal of mathematics},
  volume={28},
  number={3},
  pages={565--570},
  year={1969},
  publisher={Mathematical Sciences Publishers}
}

@article{gertzen2012flinders,
  title={Flinders Petrie, the travelling salesman problem, and the beginning of mathematical modeling in archaeology},
  author={Gertzen, Thomas L and Gr{\"o}tschel, Martin},
  journal={Documenta Mathematica},
  volume={2012},
  pages={199--210},
  year={2012}
}

@article{giraud2023localization,
  title={Localization in 1D non-parametric latent space models from pairwise affinities},
  author={Giraud, Christophe and Issartel, Yann and Verzelen, Nicolas},
  journal={Electronic Journal of Statistics},
  volume={17},
  number={1},
  pages={1587--1662},
  year={2023},
  publisher={The Institute of Mathematical Statistics and the Bernoulli Society}
}

@article{janssen2022reconstruction,
  title={Reconstruction of line-embeddings of graphons},
  author={Janssen, Jeannette and Smith, Aaron},
  journal={Electronic Journal of Statistics},
  volume={16},
  number={1},
  pages={331--407},
  year={2022},
  publisher={The Institute of Mathematical Statistics and the Bernoulli Society}
}

@mastersthesis{zendehboodi2023exploring,
  title={Exploring the use of spectral seriation to uncover dynamics in embryonic development: a geometric and probabilistic approach},
  author={Zendehboodi, Roomina},
  year={2023},
  school={University of British Columbia}
}

@article{flammarion2019optimal,
  title={Optimal rates of statistical seriation},
  author={Flammarion, Nicolas and Mao, Cheng and Rigollet, Philippe},
  journal={Bernoulli},
  volume={25},
  number={1},
  pages={623--653},
  year={2019}
}

@article{shah2016stochastically,
  title={Stochastically Transitive Models for Pairwise Comparisons: Statistical and Computational Issues},
  author={Shah, Nihar B and Balakrishnan, Sivaraman and Guntuboyina, Adityanand and Wainwright, Martin J},
  journal={IEEE Transactions on Information Theory},
  volume={63},
  number={2},
  pages={934--959},
  year={2016},
  publisher={IEEE}
}

@article{saelens2019comparison,
  title={A comparison of single-cell trajectory inference methods},
  author={Saelens, Wouter and Cannoodt, Robrecht and Todorov, Helena and Saeys, Yvan},
  journal={Nature biotechnology},
  volume={37},
  number={5},
  pages={547--554},
  year={2019},
  publisher={Nature Publishing Group US New York}
}

@article{ng2001spectral,
  title={On spectral clustering: Analysis and an algorithm},
  author={Ng, Andrew and Jordan, Michael and Weiss, Yair},
  journal={Advances in neural information processing systems},
  volume={14},
  year={2001}
}

@article{schiebinger2015geometry,
  title={The geometry of kernelized spectral clustering},
  author={Schiebinger, Geoffrey and Wainwright, Martin J and Yu, Bin},
  journal={The Annals of Statistics},
  volume={43},
  number={2},
  pages={819--846},
  year={2015}
}

@article{mittnenzweig2021single,
  title={A single-embryo, single-cell time-resolved model for mouse gastrulation},
  author={Mittnenzweig, Markus and Mayshar, Yoav and Cheng, Saifeng and Ben-Yair, Raz and Hadas, Ron and Rais, Yoach and Chomsky, Elad and Reines, Netta and Uzonyi, Anna and Lumerman, Lior and others},
  journal={Cell},
  volume={184},
  number={11},
  pages={2825--2842},
  year={2021},
  publisher={Elsevier}
}

@article{zhang2020determining,
  title={Determining sequencing depth in a single-cell RNA-seq experiment},
  author={Zhang, Martin Jinye and Ntranos, Vasilis and Tse, David},
  journal={Nature communications},
  volume={11},
  number={1},
  pages={774},
  year={2020},
  publisher={Nature Publishing Group UK London}
}

@article{schiebinger2021opinion,
  title={Reconstructing developmental landscapes and trajectories from single-cell data},
  author={Schiebinger, Geoffrey},
  journal={Current Opinion in Systems Biology},
  volume={27},
  pages={100351},
  year={2021},
  publisher={Elsevier}
}

@article{tibshirani2005sparsity,
  title={Sparsity and smoothness via the fused lasso},
  author={Tibshirani, Robert and Saunders, Michael and Rosset, Saharon and Zhu, Ji and Knight, Keith},
  journal={Journal of the Royal Statistical Society Series B: Statistical Methodology},
  volume={67},
  number={1},
  pages={91--108},
  year={2005},
  publisher={Oxford University Press}
}

@InProceedings{buttazzo2002adp,
author="Buttazzo, Giusppe
and Oudet, Edouard
and Stepanov, Eugene",
editor="dal Maso, Gianni
and Tomarelli, Franco",
title="Optimal Transportation Problems with Free Dirichlet Regions",
booktitle="Variational Methods for Discontinuous Structures",
year="2002",
publisher="Birkh{\"a}user",
address="Basel",
pages="41--65",
series="Progress in Nonlinear Differential Equations and Their Applications",
volume="51",
isbn="978-3-0348-8193-7"
}

@article{liu2017reconstructing,
  title={Reconstructing cell cycle pseudo time-series via single-cell transcriptome data},
  author={Liu, Zehua and Lou, Huazhe and Xie, Kaikun and Wang, Hao and Chen, Ning and Aparicio, Oscar M and Zhang, Michael Q and Jiang, Rui and Chen, Ting},
  journal={Nature communications},
  volume={8},
  number={1},
  pages={22},
  year={2017},
  publisher={Nature Publishing Group UK London}
}

@inproceedings{kneedle,
  title     = {Finding a "Kneedle" in a Haystack: Detecting Knee Points in System Behavior},
  author    = {Satopaa, Ville A. and Albrecht, Jeannie and Irwin, David E. and Raghavan, Barath},
  booktitle = {Proceedings of the 31st International Conference on Distributed Computing Systems Workshops (ICDCSW)},
  year      = {2011},
  pages     = {166--171},
  doi       = {10.1109/ICDCSW.2011.20},
  url       = {https://raghavan.usc.edu/papers/kneedle-simplex11.pdf}
}

@article{mouse_time_series,
	title = {A single-cell time-lapse of mouse prenatal development from gastrula to birth},
	volume = {626},
	issn = {0028-0836, 1476-4687},
	url = {https://www.nature.com/articles/s41586-024-07069-w},
	doi = {10.1038/s41586-024-07069-w},
	language = {en},
	number = {8001},
	urldate = {2026-04-28},
	journal = {Nature},
	author = {Qiu, Chengxiang and Martin, Beth K. and Welsh, Ian C. and Daza, Riza M. and Le, Truc-Mai and Huang, Xingfan and Nichols, Eva K. and Taylor, Megan L. and Fulton, Olivia and O’Day, Diana R. and Gomes, Anne Roshella and Ilcisin, Saskia and Srivatsan, Sanjay and Deng, Xinxian and Disteche, Christine M. and Noble, William Stafford and Hamazaki, Nobuhiko and Moens, Cecilia B. and Kimelman, David and Cao, Junyue and Schier, Alexander F. and Spielmann, Malte and Murray, Stephen A. and Trapnell, Cole and Shendure, Jay},
	month = feb,
	year = {2024},
	pages = {1084--1093},
}

@article{scanpy,
	title = {{SCANPY}: large-scale single-cell gene expression data analysis},
	volume = {19},
	issn = {1474-760X},
	shorttitle = {{SCANPY}},
	url = {https://genomebiology.biomedcentral.com/articles/10.1186/s13059-017-1382-0},
	doi = {10.1186/s13059-017-1382-0},
	language = {en},
	number = {1},
	urldate = {2026-05-03},
	journal = {Genome Biology},
	author = {Wolf, F. Alexander and Angerer, Philipp and Theis, Fabian J.},
	month = dec,
	year = {2018},
	pages = {15},
}

@article{anderes2016discrete,
  title={Discrete Wasserstein barycenters: Optimal transport for discrete data},
  author={Anderes, Ethan and Borgwardt, Steffen and Miller, Jacob},
  journal={Mathematical Methods of Operations Research},
  volume={84},
  number={2},
  pages={389--409},
  year={2016},
  publisher={Springer}
}

@article{pooladian2021entropic,
  title={Entropic estimation of optimal transport maps},
  author={Pooladian, Aram-Alexandre and Niles-Weed, Jonathan},
  journal={arXiv preprint arXiv:2109.12004},
  year={2021}
}

@article{rigollet2025sample,
  title={On the sample complexity of entropic optimal transport},
  author={Rigollet, Philippe and Stromme, Austin J},
  journal={The Annals of Statistics},
  volume={53},
  number={1},
  pages={61--90},
  year={2025},
  publisher={Institute of Mathematical Statistics}
}

@article{divol2025tight,
  title={Tight Stability Bounds for Entropic Brenier Maps},
  author={Divol, Vincent and Niles-Weed, Jonathan and Pooladian, Aram-Alexandre},
  journal={International Mathematics Research Notices},
  volume={2025},
  number={7},
  year={2025},
  publisher={Oxford UP}
}

@article{moosmuller2023linear,
  title={Linear optimal transport embedding: provable Wasserstein classification for certain rigid transformations and perturbations},
  author={Moosm{\"u}ller, Caroline and Cloninger, Alexander},
  journal={Information and Inference: A Journal of the IMA},
  volume={12},
  number={1},
  pages={363--389},
  year={2023},
  publisher={Oxford University Press}
}

@inproceedings{bai2023linear,
  title={Linear optimal partial transport embedding},
  author={Bai, Yikun and Medri, Ivan Vladimir and Martin, Rocio Diaz and Khan, Rana Muhammad Shahroz and Kolouri, Soheil},
  booktitle={Proceedings of the 40th International Conference on Machine Learning},
  pages={1492--1520},
  year={2023}
}

@article{sun1998almost,
  title={The almost equivalence of pairwise and mutual independence and the duality with exchangeability},
  author={Sun, Yeneng},
  journal={Probability Theory and Related Fields},
  volume={112},
  number={3},
  pages={425--456},
  year={1998},
  publisher={Springer}
}

@article{aamari2019reach,
author = {Eddie Aamari and Jisu Kim and Fr{\'e}d{\'e}ric Chazal and Bertrand Michel and Alessandro Rinaldo and Larry Wasserman},
title = {{Estimating the reach of a manifold}},
volume = {13},
journal = {Electronic Journal of Statistics},
number = {1},
publisher = {Institute of Mathematical Statistics and Bernoulli Society},
pages = {1359 -- 1399},
year = {2019},
doi = {10.1214/19-EJS1551},
URL = {https://doi.org/10.1214/19-EJS1551}
}

@book{federer1969geometric,
  title={Geometric measure theory},
  author={Federer, Herbert},
  year={1969},
  publisher={Springer}
}

@book{david1993analysis,
  title={Analysis of and on uniformly rectifiable sets},
  author={David, Guy and Semmes, Stephen},
  volume={38},
  year={1993},
  publisher={American Mathematical Society}
}

@article{du2006convergence,
  title={Convergence of the Lloyd algorithm for computing centroidal Voronoi tessellations},
  author={Du, Qiang and Emelianenko, Maria and Ju, Lili},
  journal={SIAM journal on numerical analysis},
  volume={44},
  number={1},
  pages={102--119},
  year={2006},
  publisher={SIAM}
}

@article{trushkin1982sufficient,
  title={Sufficient conditions for uniqueness of a locally optimal quantizer for a class of convex error weighting functions},
  author={Trushkin, Alexander V.},
  journal={IEEE Transactions on Information Theory},
  volume={28},
  number={2},
  pages={187--198},
  year={1982},
  publisher={IEEE}
}

@inproceedings{
banerjee2025efficient,
title={Efficient Trajectory Inference in Wasserstein Space Using Consecutive Averaging},
author={Amartya Banerjee and Harlin Lee and Nir Sharon and Caroline Moosm{\"u}ller},
booktitle={The 28th International Conference on Artificial Intelligence and Statistics},
pages={2260-2268},
year={2025},
url={https://openreview.net/forum?id=ucyKTM7lO5}
}

\begin{appendix}

\section{Properties of curves in metric spaces}\label{appn:cptness}
In this section, we collect some basic results about curves into metric spaces. Many of them are straightforward, but for results where we could not find a statement and proof elsewhere in the literature, we include the proofs in order to make the article more self-contained.

The following is a key lemma for us, and is a restatement of \cite[Lemma
1.1.4(b)]{ambrosio2008gradient}.
\begin{lem}[Existence of constant speed reparametrization]
\label{lem:const-speed-reparam}
 Let $X$ be a complete metric space. Let $\gamma\in AC([0,1];X)$.
Then there exists an increasing, absolutely continuous $\varphi:[0,1]\rightarrow[0,1]$
and a $\hat{\gamma}\in AC([0,1];X)$ such that $\gamma=\hat{\gamma}\circ\varphi$
and $|\dot{\hat{\gamma}}_{t}|=\text{Length}(\gamma)$ for almost all
$t\in[0,1]$.
\end{lem}

Another basic fact is the following version of the Arzel\`a-Ascoli theorem, which is
a consequence
of \cite[Vol. 1, Ch. II, §18, Thm. 7]{kolmogorov1999elements}.
\begin{lem}[Arzel\`a-Ascoli]
  \label{thm:arzela-ascoli}
  Let $X$ and $Y$ be compact metric spaces.
  Let $\{f_{n}\}_{n\in\mathbb{N}}$ be a sequence of functions $f_n: X\to Y$,  \emph{uniformly equicontinuous}
  with \emph{modulus of uniform
    continuity }$\zeta:\mathbb{R}\rightarrow\mathbb{R}$, meaning that:
  \[
    \forall\varepsilon>0, \forall n,\quad
    d_{X}(x,x^{\prime})<\zeta(\varepsilon)\implies
    d_{Y}(f_{n}(x),f_{n}(x^{\prime}))<\varepsilon.
  \]
  Then,
   there exists a subsequence $\{f_{n_{k}}\}$ converging uniformly to a continuous function $f:X\rightarrow Y$ with the same   modulus of uniform
  continuity $\zeta$.
\end{lem}

As an easy consequence of these two lemmas,  we deduce the following compactness result for curves in metric spaces:

\begin{lem}[Compactness and lower semicontinuity of length functional]
\label{lem:length-is-lsc}  (i) Let $X$ be a compact metric space.
Let $n\in\mathbb{N}$, and for each $n$ let $\gamma^{n}\in AC([0,1];X)$
be constant speed, in the sense that $|\dot{\gamma}_{t}^{n}|=\text{Length}(\gamma^{n})$
for almost all $t\in[0,1]$. Suppose that $\sup_{n\in\mathbb{N}}\text{Length}(\gamma^{n})<\infty$.
Then there exists some $\gamma^{*}\in AC([0,1];X)$ (but not necessarily
constant speed) such that, up to a subsequence $n_{k}$, it holds
that
\[
\gamma_{t}^{n_{k}}\rightarrow\gamma_{t}^{*}\text{ uniformly for all }t\in[0,1].
\]

(ii) Let $\gamma^{n}$ be a sequence in $AC([0,1];X)$ converging
uniformly in $t$ to some $\gamma^* \in AC([0,1];X)$. Then,
$\liminf_{n_{k}}\text{Length}(\gamma^{n_{k}})\geq\text{Length}(\gamma^{*})$.
\end{lem}

\begin{rem}
  This lemma can be rephrased as follows: if one considers the
  quotient space $AC([0,1];X)/\sim$ where $\sim$ denotes equality up
  to time reparametrization, then the functional
  $\text{Length}(\cdot)$ is lower semicontinuous and has compact
  sublevel sets on this quotient space.
\end{rem}
\begin{proof}
(i)
Since
$|\dot{\gamma_{t}}^{n}|=\text{Length}(\gamma^{n})$
for almost all $t$, we have
\[
|t-s|<\delta\implies d_{X}(\gamma_{t}^{n},\gamma_{s}^{n})<\delta\cdot\text{Length}(\gamma^{n})
\le \delta  \cdot \sup_{n\in\mathbb{N}}\text{Length}(\gamma^{n}).
\]
Therefore, the fact that $\sup_{n\in\mathbb{N}}\text{Length}(\gamma^{n})<\infty$ implies uniform equicontinuity of $\{\gamma^{n}\}$ with the Lipschitz constant $\sup_{n\in\mathbb{N}}\text{Length}(\gamma^{n})$. The existence of $\gamma^*$ now follows immediately from the Arzela-Ascoli theorem  \cref{thm:arzela-ascoli}.

In particular, we have that $t\mapsto\gamma_{t}^{*}$ is Lipschitz with
Lipschitz constant at most
$\sup_{n\in\mathbb{N}}\text{Length}(\gamma^{n})$.  Therefore
$\gamma^{*} \in AC([0,1];X)$ with metric speed and length at most
$\sup_{n\in\mathbb{N}}\text{Length}(\gamma^{n})$. Note that the metric
derivative $|\dot{\gamma}_{t}^{*}|:=\lim_{s\rightarrow
  t}\frac{d_{X}(\gamma_{t}^{*},\gamma_{s}^{*})}{|t-s|}$ exists for
a.e. $t \in [0,1]$ by a version of Rademacher's theorem for functions taking values in a compact metric space (see \cite[Theorem 4.1.6]{ambrosio2004topics}).

(ii)  Let $\gamma^{*}\in AC([0,1];X)$, and fix $\delta>0$. There
exists a $K\in\mathbb{N}$ (depending on $\delta$) and a sequence
of times $0=t_{0}<t_{1}<\ldots<t_{K-1}<t_{K}=1$ such that
\[
\sum_{k=0}^{K-1}d_{X}(\gamma_{t_{k}}^{*},\gamma_{t_{k+1}}^{*})>\text{Length}(\gamma^{*})-\delta.
\]
Next, consider a sequence $\gamma^{n}\rightarrow\gamma^{*}$ converging uniformly
in $t$. Pick $N\in\mathbb{N}$ such that  for all $n \geq N$ we
  have  $d_{X}(\gamma_{t}^{n},\gamma_{t}^{*})<\delta/K$; it follows that for all $n\geq N$,
\begin{align*}
  \text{Length}(\gamma^{n}) &   \geq \sum_{k=0}^{K-1}
  d_{X}(\gamma_{t_{k}}^{n},
  \gamma_{t_{k+1}}^{n}) \\
  & \geq
  \sum_{k=0}^{K-1} \left[ d_{X}(\gamma_{t_{k}}^{*},\gamma_{t_{k+1}}^{*}) -d_{X}(\gamma_{t_{k}}^{*},\gamma_{t_{k}}^{n}) - d_{X}(\gamma_{t_{k+1}}^{*},\gamma_{t_{k+1}}^{n})\right]\\
  & >
  \text{Length}(\gamma^{*})-3\delta.
\end{align*}
Sending $\delta \rightarrow 0$ we conclude that
$\liminf_{n\rightarrow\infty} \text{Length}(\gamma^{n}) \geq
\text{Length}(\gamma^{*})$ as desired.
\end{proof}

Next we establish a compactness result for certain piecewise-geodesic
approximations of AC curves. In the context of the next lemma, we say
that, given $\gamma\in AC([0,1];X)$, a curve $\gamma^{K}\in
AC([0,1];X)$ is a \emph{K}th piecewise-geodesic approximation of
$\gamma$ provided that: for $t=0, \frac{1}{K}, \frac{2}{K},...,1$,
$\gamma_{t}^{K}=\gamma_{t}$; and, for each interval $[\frac{k}{K},
\frac{k+1}{K}]$,  the restriction $\gamma^{K}\llcorner[ \frac{k}{K}, \frac{k+1}{K}]$
is
a constant speed geodesic connecting $\gamma_{k/K}$ and
$\gamma_{(k+1)/K}$. (Note that we do not assume uniqueness of
geodesics; we only claim that, in any geodesic metric space, such
piecewise-geodesic approximations exist.) We state the lemma below for
constant-speed curves since this is the result we need for the bulk of
the article, but
the same result holds for Lipschitz curves
more generally.
\begin{lem}
\label{lem:piecewise-geodesic}
Let $X$ be a geodesic metric space. Let $\gamma\in AC([0,1];X)$
be constant-speed. For each $K\in\mathbb{N}$, let $\gamma^{K}$ be
a $K$th piecewise-geodesic approximation of $\gamma$. Then, as $K\rightarrow\infty$,
it holds that $\gamma^{K}$ converges uniformly to $\gamma$.
\end{lem}

\begin{proof}
This follows easily from the triangle inequality, since
for each $t\in[ \frac{k}{K}, \frac{k+1}{K})$, $d(\gamma_{t},\gamma_{k/K}) \leq
      \frac{1}{K}\cdot\text{Length}(\gamma);$
$ d(\gamma_{t}^{K},\gamma_{k/K}^{K}) \le \frac{1}{K}\cdot \text{Length}(\gamma) $; and       $\gamma_{k/K}=\gamma_{k/K}^{K}$.
\end{proof}

The next lemma relates the Length functional and the $1$-dimensional Hausdorff measure  $\mathcal{H}^{1}$.
\begin{lem}
\label{lem:hausdorff equals length}
Let $X$ be a metric space and let $\gamma\in \text{AC}[0,1];X)$, with range $\Gamma$. Then $\text{Length}(\gamma)\geq\mathcal{H}^{1}(\Gamma)$,
with equality if $\gamma$ is injective.
\end{lem}
\begin{proof}
This is established in the case where $\gamma$ is Lipschitz by combining
\cite[Theorem 4.1.6]{ambrosio2004topics} and \cite[Theorem 4.4.2]{ambrosio2004topics}.
The same holds more generally by replacing $\gamma$ with its constant
speed (hence 1-Lipschitz) reparametrization, and observing that length,
1d Hausdorff measure, and injectivity are all preserved under this
reparametrization.
\end{proof}

 We state the following two basic analysis results without proof.

\begin{lem} \label{prop:inv-cont}
  Let $(X,\mathscr{T}_X)$ be compact and $(Y, \mathscr{T}_Y)$ be
  Hausdorff. Let $f : X \to Y$ be continuous and injective. Then
  $f^{-1} : f(X) \to X$ is continuous.
\end{lem}
\begin{lem} \label{prop:inj-mon}
  Suppose $\varphi : [0,1] \to [0,1]$ is a continuous, injective
  function. Then $\varphi$ is strictly monotone.
\end{lem}

Next we show that injective reparametrization does not change the length of an AC curve. This is basically routine, but again we were unable to find an appropriate reference.
\begin{lem} \label{lem:reparam}
  Suppose $f \in \mathrm{AC}([0,1]; X)$ and $\varphi : [0,1] \to
  [0,1]$ is a continuous, injective function. Then $f \circ \varphi
  \in \mathrm{AC}([0,1]; X)$. Moreover, if $\varphi$ is a bijection,
  then $\mathrm{Length}(f) = \mathrm{Length}(f \circ \varphi)$.
\end{lem}

\begin{proof}
  Let $\varepsilon > 0$. Since $f$ is AC, there exists
  $\delta > 0$ such that for all finite sequences of pairwise disjoint
  intervals $[a_1, b_1], \ldots, [a_k, b_k] \subseteq [0,1]$ such that
  \[
    \sum_j |b_j - a_j| < \delta
  \]
  one has
  \[
    \sum_j d(f(b_j), f(a_j)) < \varepsilon.
  \]
  Define $\varepsilon' = \delta/k$. Note that $\varphi$ is a
  continuous function on a compact set and hence uniformly continuous.
  Therefore there exists a $\delta' > 0$ such that for all $a,b \in
  [0,1]$ we have
  \[
    |b-a| < \delta' \implies |\varphi(b) - \varphi(a)| < \varepsilon'.
  \]
  Now, select an arbitrary sequence of pairwise disjoint intervals
  $[a_1, b_1], \ldots, [a_k, b_k] \subseteq [0,1]$ such that
  \[
    \sum_j |b_j - a_j| < \delta'.
  \]
  Since $\varphi$ is a continuous, injective function, by
  \cref{prop:inj-mon} it is strictly monotone. Without the loss of
  generality we suppose $\varphi$ is strictly increasing. Then the
  intervals $[\varphi(a_1), \varphi(b_1)], \ldots, [\varphi(a_k),
  \varphi(b_k)]$ are nonempty, and (by injectivity of $\varphi$)
  remain pairwise disjoint. Then by uniform continuity of $\varphi$,
  \[
    \sum_j |\varphi(b_j) - \varphi(a_j)| < k\varepsilon' = \delta,
  \]
  whence
  \[
    \sum_j d(f(\varphi(b_j)), f(\varphi(a_j))) < \varepsilon,
  \]
  so $f \circ \varphi$ is absolutely continuous as desired.

  Now, further suppose $\varphi$ is a bijection. Let $I(N)$ denote the
  set of all partitions of $[0,1]$ according to $0 = t_0 < t_1 <
  \cdots < t_{N-1} < t_N = 1$. Note that by bijectivity and
  strictly-increasing-ness of $\varphi$ we have
  \[
    \varphi(I(N)) = I(N) = \varphi^{-1}(I(N)),
  \]
  thus
  \[
    \sup_{N} \sum_{j=1}^N d(f(\varphi(t_j)), f(\varphi(t_{j-1})))
    = \sup_N \sum_{j=1}^N d(f(t_j), f(t_{j-1})),
  \]
  in particular $\mathrm{Length}(f) = \mathrm{Length}(f \circ
  \varphi)$, as desired.
\end{proof}
Related to the previous lemma, injective reparametrization gives the minimum arc-length. The argument is again routine, but we provide an explicit argument to make the article self-contained.
\begin{lem} \label{lem:arclen-most-efficient}
  Let $f \in \mathrm{AC}([0,1]; X)$ be injective and denote
  \[
    \Phi(f) = \{g \in \mathrm{AC}([0,1]; X) \mid g \text{ is
      constant speed and } g([0,1]) = f([0,1])\}.
  \]
  Let $\hat{f}$ denote the constant speed reparametrization of $f$.
  Then
  $\mathrm{Length}(g) = \mathrm{Length}(f)$ iff $g$ is injective,
  whence (up to a possible time reversal) we have $g = \hat{f}$.
 In particular,
 $\hat{f}$ (and its time reversal) are the only elements of $\Phi(f)$
  such that
  \[
    \mathrm{Length}(\hat{f}) = \inf_{g \in \Phi(f)} \mathrm{Length}(g).
  \]
\end{lem}

\begin{proof}
  Let $g \in \Phi(f)$ be arbitrarily chosen.
  We proceed
  by casework.
  \begin{enumerate}
    \item Suppose $g$ is injective. We want to show that $g = \hat{f}$ up to time-reversal. To
      that end, by \cref{prop:inv-cont}, $g^{-1}$ is well-defined and
      continuous. So $h : [0,1] \to [0,1]$ given by
      \[
      h(t) = g^{-1}(f(t))
      \]
      is a continuous bijection. By \cref{lem:reparam}, we thus see
      \[
      \mathrm{Length}(g) = \mathrm{Length}(g \circ h) =
      \mathrm{Length}(f).
      \]
      It remains to show $g = \hat{f}$ (up to a possible time-reversal). To
      that end, by \cref{prop:inj-mon}, $h$ is strictly monotone.
      First suppose $h$ is increasing. Then
      $f,g$ visit points of $f([0,1])$ in the same order. By the
      definition of the constant-speed parametrization we immediately get
      $\hat{f} = g$, as desired. Alternatively, $h$ is decreasing, in which case $\hat{f}$ coincides with the time-reversal of $g$.
    \item Suppose that $g$ is not injective and that $\{g(0), g(1)\}
      \neq \{f(0), f(1)\}$. We want to show this implies
      $\mathrm{Length}(g) > \mathrm{Length}(f)$. To that end: First
      note that reversing the parametrization of $f$  as
      necessary we may assume
      \[
      \inf g^{-1}(\{f(0)\}) < \inf g^{-1}(\{f(1)\}).
      \]
      (Note that we cannot always achieve this by reversing the
      parametrization of $g$; for example, if $g(0) = f(1) = g(1)$).
      Define $s_0 = \inf g^{-1}(\{f(0)\})$ and let
      \[
      s_1 = \inf([s_0, 1] \cap g^{-1}(\{f(1)\})).
      \]
      Suppose $[s_0, s_1] \neq [0,1]$; without the loss of generality
      we may assume $s_0 \neq 0$. Then $g(0) \neq g(s_0)$ and so
      \[
      0 < d(g(0), g(s_0)) \leq \mathrm{Length}(g\llcorner [0, s_0]).
      \]
      Now, on the other hand, observe $f^{-1} \circ g : [0,1] \to
      [0,1]$ is well-defined and continuous. Since $(f^{-1} \circ
      g)(s_0) = 0$ and $(f^{-1}\circ g)(s_1) = 1$ and the continuous
      image of a connected set is connected, we get $(f^{-1} \circ
      g)([s_0, s_1]) = [0,1]$, whence we have $g([s_0, s_1]) =
      f([0,1])$. It follows that
      \[
      \mathrm{Length}(f) \leq \mathrm{Length}(g\llcorner [0, s_0]).
      \]
      Together with the earlier bound we thus
      obtain $$\mathrm{Length}(g) \geq \mathrm{Length}(g\llcorner [0, s_0]) + \mathrm{Length}(g\llcorner [0, s_0]) >
      \mathrm{Length}(f),$$ as desired.
    \item Suppose $g$ is not injective and that $\{g(0), g(1)\} =
      \{f(0), f(1)\}$. We want to show $\mathrm{Length}(g) >
      \mathrm{Length}(f)$. To that end fix $t \in [0,1]$ such that
      $g^{-1}(\{g(t)\})$ contains multiple values. Let
      \[
      t_0 = \inf g^{-1}(\{g(t)\})\quad\textrm{and}\quad t_1 = \sup g^{-1}(\{g(t)\}).
      \]
      Note $t_0 \neq t_1$. Observe that since $g$ has constant-speed
      parametrization we have $g([t_0, t_1])$ is not a singleton. Hence
      we have
      \[
      0 < \mathrm{diam}(g([t_0, t_1])) \leq
      \mathrm{Length}(g\llcorner [t_0, t_1]).
      \]
      Define
      \[
      \tilde g(t) =
      \begin{cases}
        g(t) & t \not \in [t_0, t_1] \\
        g(t_0) & t \in [t_0, t_1].
      \end{cases}
      \]
      Note that $\tilde g$ is continuous,
      hence $f^{-1} \circ \tilde g$ is continuous, and the same analysis
      as in the previous case shows $\mathrm{Length}(f) \leq
      \mathrm{Length}(\tilde g)$. Finally we have
      \[
      \mathrm{Length}(g) =
      \mathrm{Length}(g\llcorner[t_0, t_1]) + \mathrm{Length}(\tilde g) >
      \mathrm{Length}(f),
      \]
      as desired.
  \end{enumerate}
  Since these cases are exhaustive, we see $\mathrm{Length}(g) =
  \mathrm{Length}(f)$ iff $g$ is injective, which occurs iff $g = f$
  (up to time reversal).
\end{proof}

\section{Background on reproducing kernel Hilbert spaces}
\label{sec:rkhs}

For technical reasons, our arguments for Section \ref{sec:Wasserstein} rely on the existence of a specific
reproducing kernel Hilbert space (RKHS) defined atop a given compact
metric space. We recall the following definition.
\begin{defn}
Given a Polish space $V$ and an RKHS $\mathcal{H}$ of functions
from $V$ to $\mathbb{R}$, the \emph{maximum mean discrepancy} (MMD)
distance on $\mathcal{P}(V)$ is defined by
\[
\text{MMD}_{\mathcal{H}}(\mu,\nu):=\sup_{\Vert f\Vert_{\mathcal{H}}\leq1}\left|\int_{V}fd\mu-\int_{V}fd\nu\right|.
\]
\end{defn}

The following is a folklore result which can be deduced by combining
certain published results on RKHSes with a classical universality theorem from descriptive set theory.
\begin{theorem}
  Let $V$ be a compact metric space. Then there exists an RKHS
  $\mathcal{H}$ of functions from $V$ to $\mathbb{R}$ for which the
  topology on $\mathcal{P}(V)$ induced by $\text{MMD}_{\mathcal{H}}$
  is the same as the weak{*} topology, and for which the reproducing
  kernel $k$ is bounded.
\end{theorem}
\begin{proof}
We argue by abstract construction of a specific reproducing kernel. Let $Y$ be a separable Hilbert space and $T:V\rightarrow Y$ be a continuous injection. Then Theorems 9 and 21 from \cite{wynne2022kernel} establish that the MMD distance associated to the RKHS $\mathcal{H}$ induced by the kernel $$k(x,x^\prime)=e^{-\frac{1}{2}\Vert T(x) - T(x^\prime)\Vert^2_Y}$$
induces the weak* topology on $\mathcal{P}(V)$. (Note that this kernel is automatically bounded since $T(V)$ is compact inside $Y$.)

Lastly, such a map $T$ and separable Hilbert space $Y$ always exist whenever $V$ is compact. Indeed, taking $Y=\ell^2$, this follows from the classical fact that every compact metric space can be embedded homeomorphically as a compact subset of  $\ell^2$, see Section 4.C in \cite{kechris2012classical}.
\end{proof}
Given a compact metric space $V$ and RKHS $\mathcal{H}$ atop $V$,
we equip the space $\mathcal{P}(V)$ with the metric $\text{MMD}_{\mathcal{H}}$.
In particular we choose $\mathcal{H}$ so that $\text{MMD}_{\mathcal{H}}$
puts the same topology on $\mathcal{P}(V)$ as the weak{*} topology,
so in particular $(\mathcal{P}(V),\text{MMD}_{\mathcal{H}})$ is itself
a compact metric space.

Now given two measures $\Lambda,\Xi\in\mathcal{P}(\mathcal{P}(V))$,
we consider the 1-Wasserstein distance on $\mathcal{P}(\mathcal{P}(V))$
defined with respect to $\text{MMD}_{\mathcal{H}}$ on $\mathcal{P}(V)$:
\[
\mathbb{W}_{1}(\Lambda,\Xi):=\inf_{\pi\in\Pi(\Lambda,\Xi)}\int_{\mathcal{P}(V)\times\mathcal{P}(V)}\text{MMD}_{\mathcal{H}}(\mu,\mu^{\prime})d\pi(\mu,\mu^{\prime}).
\]
It holds that $\mathbb{W}_{1}$ metrizes the weak{*} topology on $\mathcal{P}(\mathcal{P}(V))$.

At the same time, it is shown in the proof of \cite[Theorem 4.10]{wainwright2019high}  that $\mathbb{E}[\text{MMD}(\mu,\hat{\mu}_{n})]$
is bounded by twice the Rademacher complexity $Rad_{n}(F)$ of the
function class $F=\{\Vert f\Vert_{\mathcal{H}}\leq1\}$; and in the
case where the reproducing kernel $k$ is bounded, \cite[Section 4.3]{bartlett2002rademacher} establishes that $Rad_{n}(F)\leq\sqrt{\frac{\sup_{x}k(x,x)}{n}}$.
Consequently $\mathbb{E}[\text{MMD}(\mu,\hat{\mu}_{n})]\leq C_{\mathcal{H}}n^{-1/2}$
for some uniform constant $C_{\mathcal{H}}>0$. Likewise, \cite[Theorem 4.10]{wainwright2019high} establishes that
\[
\mathbb{P}[\text{MMD}(\mu,\hat{\mu}_{n})\geq t+\mathbb{E}[\text{MMD}(\mu,\hat{\mu}_{n})]]\leq\exp\left(-\frac{nt^{2}}{D_{\mathcal{H}}}\right)
\]
where $D_{\mathcal{H}}$ is a constant which depends only on the reproducing
kernel $k(\cdot,\cdot)$ for $\mathcal{H}$ and is finite whenever
$\sup_{x,y\in V}k(x,y)<\infty$.
\section{Deferred Proofs}
\subsection{\texorpdfstring{Proofs for Section \ref{sec:basics}}{Proofs for Section 2}}
\label{subsec:proofs-for-basics}
\begin{proof}[Proof of Proposition \ref{prop:mins-exist}]
The proof follows essentially from the compactness and semicontinuity properties of the family of absolutely continuous curves as given in
Appendix~\ref{appn:cptness}. To give the details,  let
\[
M:=\inf_{\gamma\in AC([0,1];X)}\left\{ \int_{X}d^{2}(x,\Gamma)d\Lambda(x)+\beta\text{Length}(\gamma)\right\} .
\]

  Let $\gamma^{n}$ be a sequence of AC curves in $X$, such that
\[
\lim_{n\rightarrow\infty}\left\{ \int_{X}d^{2}(x,\Gamma^{n})d\Lambda(x)+\beta\text{Length}(\gamma^{n})\right\} =M.
\]
Without loss of generality, we can take $\sup_{n}
  \text{Length}(\gamma^{n}) \leq \frac{2M}{\beta}$. Note also
that
\[
\int_{X}d^{2}(x,\Gamma)d\Lambda(x)+\beta\text{Length}(\gamma)
\]
is invariant under time-reparametrization. In
particular, by Lemma \ref{lem:const-speed-reparam} we can
take each $\gamma^{n}$ to have unit-speed
parametrization.
  We can thus apply Lemma \ref{lem:length-is-lsc}
to deduce that there is some subsequence $n_{k}$ and some limiting AC
curve $\gamma^{*}:[0,1]\rightarrow X$ satisfying
$\liminf_{k\rightarrow\infty} \text{Length}(\gamma^{n_{k}}) \geq
\text{Length}(\gamma^{*})$, and $\gamma_{t}^{n_{k}} \rightarrow
  \gamma_{t}^{*}$ uniformly in $t$. The latter
  implies that pointwise in $x$, $\lim_{k\rightarrow\infty}
  d(x,\Gamma^{n_{k}}) = d(x,\Gamma^{*})$. Consequently
\begin{multline*}
M=\liminf_{k\rightarrow\infty}\int_{X}d^{2}(x,\Gamma^{n_{k}})d\Lambda(x)+\beta\text{Length}(\gamma^{n_{k}})\geq\int_{X}d^{2}(x,\Gamma^{*})d\Lambda(x)+\beta\text{Length}(\gamma^{*}),
\end{multline*}
and so $\gamma^{*}$ is a minimizer, as desired.
\end{proof}

\begin{proof}[Proof of Proposition \ref{prop:r2-nonuniqueness}] Let $X
  \subseteq \mathbb R^2$ such that $X^\circ \neq \varnothing$ and let
  $\mu \in \mathcal P(X)$. For some fixed $\beta > 0$ let $\gamma$
  minimize $\mathrm{PPC}(\Lambda)$. Then by \cite[Thm.\
  1.3]{lu2016average}, $\gamma$ is injective; in particular,
  $\gamma(0) \neq \gamma(1)$. Suppose there exist distinct isometries
  $\varphi_1$, $\varphi_2$, $\varphi_3$ such that $(\varphi_1)_{\#
    \mu} = \mu$; taking an affine extension, we may treat them as
  isometries of $\mathbb{R}^2$. There are only two such isometries
  that fix $\{\gamma(0), \gamma(1)\}$: (1) identity, and (2)
  reflection across $\{d(x, \gamma(0)) = d(x,\gamma(1))\}$. So some
  $\varphi_i$ maps $\{\gamma(0), \gamma(1)\}$ to a different set.
  Since $\gamma$ is a homeomorphism onto its image
  (\cref{prop:inv-cont}), $\gamma(0), \gamma(1)$ are the unique
  noncut\footnote{Given a set $S$, we say $p \in S$ is a
    \emph{noncut} point of $S$ if $S \setminus \{p\}$ is connected.}
  points of $\mathrm{image}(\gamma)$, and since $\varphi_i$ is also a
  homeomorphism, $\varphi_i$ not fixing $\{\gamma(0), \gamma(1)\}$
  implies $\mathrm{image}(\varphi \circ \gamma) \neq
  \mathrm{image}(\gamma)$.

  Finally, since $\mu$ is invariant under $\varphi_i$, and since the
  rest of both the data-fitting term and penalty term in
  $\mathrm{PPC}(\Lambda)$ depend only on the metric and $\varphi_i$ is
  an isometry, we see $\varphi_i \circ \gamma \neq \gamma$ are
  distinct minimizers of $\mathrm{PPC}(\Lambda)$ not equivalent by
  reparametrization.
\end{proof}
{
  \def\triscale{.8}
  \def\trispace{-3em}
  \begin{figure}[H]
    \centering
    \includegraphics[scale=\triscale]{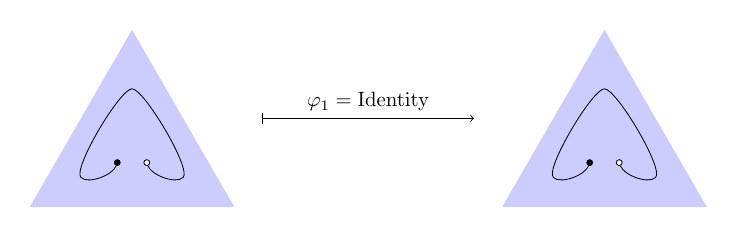}
  \end{figure}
  \vspace{\trispace}
  \begin{figure}[H]
    \includegraphics[scale=\triscale]{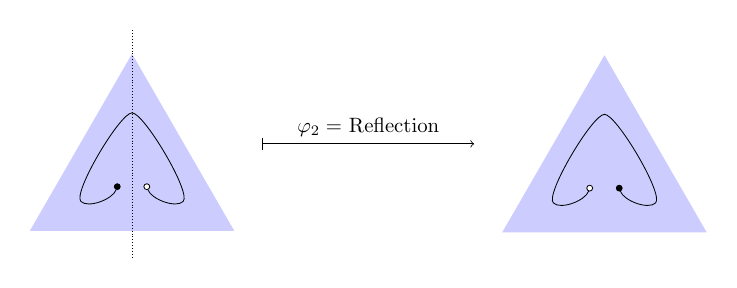}
  \end{figure}
  \vspace{\trispace}
  \begin{figure}[H]
  \includegraphics[scale=\triscale]{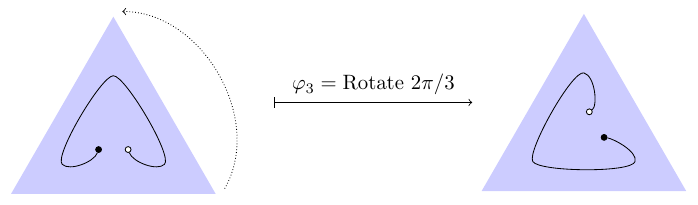}
  \caption{ Illustration for the proof of Proposition \ref{prop:r2-nonuniqueness}. Take $\varphi_1, \varphi_2, \varphi_3$ to be the three distinct isometries indicated, applied to a
    curve with $\mu$ uniform on a triangle. In this case,
    $\varphi_3$ yields a distinct image.}
\end{figure}
}

We use the following lemma as an ingredient in the proof of Proposition \ref{prop:gamma-m-gamma}.

\begin{lem}
  \label{lem:continuity-unpenalized-objective}Let $X$ be a compact
  metric space. Let $(\gamma^{n})_{n\in\mathbb{N}}$ denote a sequence
  of measurable functions from $[0,1]$ to $X$, and let
  $(\Lambda_{m})_{m\in\mathbb{N}}$ denote a sequence in
  $\mathcal{P}(X)$. Suppose that for some $\gamma:[0,1]\rightarrow X$
  and $\Lambda\in\mathcal{P}(X)$ we have
  $\gamma_{t}^{n}\rightarrow\gamma_{t}$ uniformly in $t$, and
  $\Lambda_{m}\rightharpoonup^{*}\Lambda$, respectively. Then,
  \[
    \lim_{m,n\rightarrow\infty} \int_{X} d^{2}(x,\Gamma^{n})
    d\Lambda_{m}(x) = \int_{X} d^{2}(x,\Gamma) d\Lambda(x).
  \]
\end{lem}

\begin{proof}
We employ the Moore-Osgood theorem for double limits. Accordingly,
it suffices to show that:
\begin{enumerate}
  \item $\int_{X}d^{2}(x,\Gamma^{n})d\Lambda_{m}(x)$ converges
    uniformly in $m$ to $\int_{X}d^{2}(x,\Gamma)d\Lambda_{m}(x)$ as
    $n\rightarrow\infty$; and,
  \item $\int_{X}d^{2}(x,\Gamma^{n})d\Lambda_{m}(x)$ converges
    pointwise in $n$ to $\int_{X}d^{2}(x,\Gamma^{n})d\Lambda(x)$ as
    $m\rightarrow\infty$.
\end{enumerate}
For (1), we use the fact that $\gamma_{t}^{n}\rightarrow\gamma_{t}$
uniformly in $t$. Indeed let $N$ be sufficiently large that for
  all $n \geq N$, $d(\gamma_{t}^{n},\gamma_{t})<\varepsilon$
uniformly in $t$. Then, for fixed $x\in X$,
we claim that
\[
   \Big| \inf_{s \in [0,1]} d(x,\gamma^{n}_s)-\inf_{
  t \in [0,1]}d(x,\gamma_{t})  \Big| < 2\varepsilon.
\]
Indeed, without the loss of generality suppose that
$\inf_{ s \in [0,1]} d(x,\gamma^{n}_{ s}) \geq
\inf_{ t \in [0,1]} d(x,\gamma_{t})$, and let
 $ t_0 \in [0,1]$ such that
$d(x, \gamma_{t_{0}}) < \inf_{ t \in [0,1]} d(x,\gamma_{t}) +
\varepsilon$. Then by the triangle inequality,
$d(x,\gamma_{t_{0}}^{n}) < \inf_{ t \in [0,1]} d(x,\gamma_{t}) +
2\varepsilon$; hence
$\inf_{ s \in [0,1]}d(x,\gamma^{n}_{ s}) < \inf_{ t \in
  [0,1]} d(x,\gamma_{t}) + 2\varepsilon$. Applying the same reasoning
when
$\inf_{ s \in [0,1]} d(x,\gamma^{n}_{ s}) \leq \inf_{ t \in
  [0,1]} d(x,\gamma_{t})$ establishes the claim.

Now, supposing again that
$\inf_{ s \in [0,1]} d(x,\gamma^{n}_{ s}) \geq \inf_{ t \in
  [0,1]} d(x,\gamma_{t})$, we have that
\[
  \left(\inf_{ s \in [0,1]} d(x,\gamma^{n}_s)\right)^{2} <
  \left(\inf_{ t \in [0,1]} d(x, \gamma_t) +
  2\varepsilon\right)^{2}.
\]
 Note that
  we may rewrite the inequality above as
\begin{align*}
  \inf_{ s \in [0,1]} d^{2}(x,\gamma^{n}_{ s})
  & < \inf_{ t \in [0,1]} (d^2(x, \gamma_t) + 4 \varepsilon
    d(x, \gamma_t) + \varepsilon^2), \\
  & \leq\inf_{ t \in [0,1] } d^{2}(x,\gamma_{t}) + 4\varepsilon
    \cdot \text{Diam}(X)+\varepsilon^{2}.
\end{align*}
By a symmetric argument, we  thus
have that
\begin{align*}
  |d^{2}(x,\Gamma^{n})-d^{2}(x,\Gamma)|
  &= { \Big |} \inf_{ s \in [0,1]} d^{2}(x,\gamma^{n}_{s}) -
    \inf_{ t \in [0,1]}d^{2}(x,\gamma_{t}) { \Big |}\\
 & <4\varepsilon\cdot\text{Diam}(X)+\varepsilon^{2}.
\end{align*}
This estimate is uniform in $\mu$. Therefore we compute that
\begin{align*}
  \left| \int_{X}d^{2}(x,\Gamma^{n}) d\Lambda_{m}(x) - \int_{X}
  d^{2}(x,\Gamma) d\Lambda_{m}(x)\right|
  & \leq \int_{X} |d^{2}(x,\Gamma^{n})-d^{2}(x,\Gamma)|
  d\Lambda_{m}(x)\\
  & <4\varepsilon\cdot\text{Diam}(X)+\varepsilon^{2}
\end{align*}
which is uniform in $m$ as desired.

For (2), first note that  for any nonempty $A
  \subseteq X$, the Hausdorff distance $d(x, A)$ is continuous in $x$;
  unrelatedly, the domain $X$ is compact, hence bounded. Therefore, for any fixed $n$ the function
$x\mapsto d^{2}(x,\Gamma^{n})$ is bounded and continuous on $(X,d)$.
Hence,
  directly from the fact that $\Lambda_m \rightharpoonup^* \Lambda$,
  we deduce ${ \lim_{m \to \infty}}
\int_{X} d^{2}(x,\Gamma^{n}) d\Lambda_{m}(x) =  \int_{X}
d^{2}(x,\Gamma^{n}) d\Lambda(x)$.
\end{proof}

\begin{rem}
  Let us give an indication of the proof strategy for Proposition \ref{prop:gamma-m-gamma}, which may be
  beneficial for readers familiar with the notion of
  $\Gamma$-convergence. (See \cite{braides2002gamma, degiorgi2013selected} for background on $\Gamma$-convergence, which is an abstract notion of convergence of minimization problems; familiarity with this notion is beneficial but not necessary to follow our arguments.)  As we remark immediately following Lemma
  \ref{lem:length-is-lsc}, Lemma \ref{lem:length-is-lsc} can be
  interpreted as showing that the functional $\text{Length}(\cdot)$ is
  lower semicontinuous on the space $AC([0,1];X)/\sim$, where $\sim$
  denotes equality up to time reparametrization. Similarly, here our
  strategy amounts to arguing that if $\Lambda_{N} \rightharpoonup^{*}
  \Lambda$, then the sequence of functionals $\text{PPC}(\Lambda_{N})$
  $\Gamma$-converge to $\text{PPC}(\Lambda)$, at least if the
  functionals $\text{PPC}(\Lambda_{N})$ and $\text{PPC}(\Lambda)$ are
  understood as taking arguments in the
  space $AC([0,1];X)/\sim$. We then combine this $\Gamma$-convergence
  with a compactness argument applied to the sequence of minimizers.

  We note, moreover, that the final part of the statement of the
  proposition can also be read as asserting that there exists a
  subsequence of minimizers $\gamma^{*N}$ which converges to some
  $\gamma^{*}$ in a suitable quotient topology on $AC([0,1];X)/\sim$.
\end{rem}

\begin{rem}
 We mention that
  Proposition \ref{prop:gamma-m-gamma} is broadly analogous to a result
  obtained by two of the authors in \cite[Cor.
 5.2]{kobayashi2024monge}, albeit in a different setting; see also
 \cite[Lem. 3]{slepvcev2014counterexample} for a similar stability
 result on Euclidean space for the closely related ``average distance
 variational problem'' from \cite{buttazzo2002adp}.
\end{rem}

\begin{proof}[Proof of Proposition~\ref{prop:gamma-m-gamma}]
  \emph{Step 1.} First, we observe the following. Let
  $(\gamma^{N})_{N\in\mathbb{N}}$ be any sequence in $AC([0,1];X)$
  converging uniformly in $t$ to some $\gamma\in AC([0,1];X)$. Then, by
  combining Lemmas \ref{lem:continuity-unpenalized-objective} and
  \ref{lem:length-is-lsc}(ii), we have that
  \[
    \liminf_{N\rightarrow\infty} \left( \int_{X}d^{2}(x,\Gamma^{N})
      d\Lambda_{N}(x) + \beta \text{Length}(\gamma^{N}) \right) \geq
    \int_{X} d^{2}(x,\Gamma) d\Lambda(x) + \beta \text{Length}(\gamma).
  \]

  \emph{Step 2}. Let $\gamma\in AC([0,1];X)$ be arbitrary. We claim
  that there exists a sequence $\gamma^{N}$ in $AC([0,1];X)$
  converging uniformly in $t$ to $\gamma\in AC([0,1];X)$, such that
  \[
    \lim_{N\rightarrow\infty} \left( \int_{X}d^{2}(x,\Gamma^{N})
      d\Lambda_{N}(x) + \beta \text{Length}(\gamma^{N}) \right) \leq
    \int_{X} d^{2}(x,\Gamma) d\Lambda(x) + \beta \text{Length}(\gamma).
  \]
  Indeed we can just take $\gamma^N = \gamma$. Then Lemma
  \ref{lem:continuity-unpenalized-objective} tells us that
  \[
    \int_{X}d^{2}(x,\Gamma^{N})d\Lambda_{N}(x) \rightarrow
    \int_{X}d^{2}(x,\Gamma)d\Lambda(x),
  \]
  while $\text{Length}(\gamma^{N})=\text{Length}(\gamma).$

  \emph{Step 3}. Next, note that for arbitrary $x,y\in X$, we have
  $d^{2}(x,y)\leq\text{diam}(X)^{2}$, hence for arbitrary $\gamma^N$,
  \[
    \int_{X}d^{2}(x,\Gamma^{N})d\Lambda_{N}(x)\leq\text{diam}(X)^{2}
  \]
  also. Furthermore, if we consider a $\gamma^{N}$ which is constant
  in time, then we have that $\gamma^{N}$ has length zero, hence
  \begin{align*}
    \min_{\gamma \in AC([0,1];X)}\text{PPC}(\Lambda_{N})
    & \leq\int_{X} d^{2}(x,\Gamma^{N}) d\Lambda_{N}(x) +
      \beta\text{Length}(\gamma^{N})\\
    & =\int_{X} d^{2}(x,\Gamma^{N}) d\Lambda_{N}(x)\\
    & \leq\text{diam}(X)^{2}.
  \end{align*}
  Consequently, let $\gamma^{*N}$ denote a minimizer for
  $\text{PPC}(\Lambda_{N})$; it follows that
  $\text{Length}(\gamma^{*N}) \leq \text{diam}(X)^{2} / \beta$ also. We can take $\gamma^{*N}$ to be
  constant speed without loss of generality. And so, arguing as in the
  proof of Proposition \ref{prop:mins-exist}, by Lemma
  \ref{lem:length-is-lsc}(i), we deduce that: passing to a subsequence $N_{k}$,
  there exists some AC curve $\gamma^{*}$ such that
  $\gamma_{t}^{*N_{k}}\rightarrow\gamma_{t}^{*}$ uniformly.

  More generally, let $\gamma_{t}^{*N_{k}}$ be any subsequence of minimizers
of $\text{PPC}(\Lambda_{N})$ which converges uniformly in $t$, with
limit $\gamma^{*}\in AC([0,1];X)$. Now, fix  any $\gamma\in
  AC([0,1],X)$. Step 1 implies that:
  \begin{align*}
    \int_X d^{2}(x,\Gamma^{*}) d\Lambda(x) +
    \text{Length}(\gamma^{*})
    &\leq \liminf_{N_{k} \rightarrow \infty} \left(
      \int_{X} d^{2}(x,\Gamma^{*N_{k}})
      d\Lambda_{N_{k}}(x) +
      \text{Length}(\gamma^{*N_{k}})\right) {,}
      \shortintertext{ and since each $\gamma^{*N_k}$ minimizes
      $\mathrm{PPC}(\Lambda_{N_k})$, letting $\gamma^{N_k} = \gamma$ as
      in Step 2 yields}
    &\leq \lim_{N_{k}\rightarrow\infty} \left(\int_{X}
      d^{2}(x,\Gamma^{N_{k}}) d\Lambda_{N_{k}}(x) +
      \text{Length}(\gamma^{N_{k}})\right)\\
    & = \int_{X}
      d^{2}(x,\Gamma) d\Lambda(x) + \text{Length}(\gamma).
  \end{align*}
   As $\gamma$ was chosen
  arbitrarily, this shows that $\gamma^{*}$ is a minimizer of
  $\text{PPC}(\Lambda)$ as desired.
\end{proof}

\begin{rem}
We structure the proof of Theorem \ref{thm:discrete-to-continuum} below as if we are proving the $\Gamma$-convergence
of $\text{PPC}^{K}(\Lambda_{N})$ to $\text{PPC}(\Lambda)$, and then
using compactness to extract a convergent subsequence of minimizers
which therefore converges to a minimizer of $\text{PPC}(\Lambda)$.
However, we alert the reader who is a specialist in $\Gamma$-convergence
that in order for the proof below to ``really'' be thought of as
establishing $\Gamma$-convergence, we must formally replace the space
on which $\text{PPC}^{K}(\Lambda_{N})$ is defined: specifically we
must replace the space of $K$-tuples with the space of $K$-piecewise
geodesics, and then think of the integral $\int_{X}d^{2}(x,\Gamma_{K})d\Lambda(x)$
as a functional over $K$-piecewise geodesics depending only on the locations of the knots. Subject to this replacement,
$\text{PPC}^{K}(\Lambda_{N})$ and $\text{PPC}(\Lambda)$ now are
defined atop a common space (namely the space $AC([0,1];X)/\sim$
mentioned previously), and thus the question of whether $\text{PPC}^{K}(\Lambda_{N})$
$\Gamma$-converges to $\text{PPC}(\Lambda)$ is well-defined. (However,
we emphasize that the validity of our proof does not hinge on these
abstract analytic considerations.)
\end{rem}

\begin{proof}[Proof of Theorem \ref{thm:discrete-to-continuum}]
  \emph{Step 1.} We consider a sequence of sets of points
  $(\{\gamma_{k}^{K}\}_{k=1}^{K})_{K\in\mathbb{N}}$. For each $K$, let
  $\gamma^{K}$ denote a constant speed piecewise geodesic connecting
  the points $\gamma_{k}^K$ to $\gamma_{k+1}^K$, for all $1\leq k\leq
  K-1$; such a $\gamma^{K}$ exists because $X$ is a geodesic metric
  space. Note that by construction, we have that
\[
\text{Length}(\gamma^{K})=\sum_{k=1}^{K-1}d(\gamma_{k}^{K},\gamma_{k+1}^{K}).
\]
At the same time, let $\Gamma_{cont}^{K}$ denote the graph of $\gamma^{K}$.
Note that $\Gamma^{K}\subset\Gamma_{cont}^{K}$, and so for any $x\in X$,
we have $d^{2}(x,\Gamma_{cont}^{K})\leq d^{2}(x,\Gamma^{K})$. Therefore,
\[
\int_{X}d^{2}(x,\Gamma_{cont}^{K})d\Lambda_{N}(x)+\beta\text{Length}(\gamma^{K})\leq\int_{X}d^{2}(x,\Gamma^{K})d\Lambda_{N}(x)+\beta\sum_{k=1}^{K-1}d(\gamma_{k}^{K},\gamma_{k+1}^{K}).
\]

Suppose that $\gamma^{K}$ converges uniformly in $t$ to some limiting
AC curve $\gamma$. Applying Lemma \ref{lem:continuity-unpenalized-objective},
we see that
\[
\lim_{K\rightarrow\infty,N\rightarrow\infty}\int_{X}d^{2}(x,\Gamma_{cont}^{K})d\Lambda_{N}(x)=\int_{X}d^{2}(x,\Gamma)d\Lambda(x).
\]
Combining this with the facts that
$\liminf_{K\rightarrow\infty}\text{Length}(\gamma^{K})\geq\text{Length}(\gamma)$
(from Lemma \ref{lem:length-is-lsc}(ii)) and
$\text{Length}(\gamma^{K})=\sum_{k=1}^{K-1}d(\gamma_{k}^{K},\gamma_{k+1}^{K})$,
we see that
\[
\liminf_{K,N\rightarrow\infty}\left(\int_{X}d^{2}(x,\Gamma^{K})d\Lambda_{N}(x)+\beta\sum_{k=1}^{K-1}d(\gamma_{k}^{K},\gamma_{k+1}^{K})\right)\geq\int_{X}d^{2}(x,\Gamma)d\Lambda(x)+\beta\text{Length}(\gamma).
\]

\emph{Step 2.} Let $\gamma\in AC([0,1];X)$ be constant speed. For
each $K$, we define $\{\gamma_{k}^{K}\}_{k=1}^{K}$ by
\[
\gamma_{k}^{K}=\gamma_{k/K}.
\]
 Note that for each $K$ we have
\[
\text{Length}(\gamma)\geq\sum_{k=0}^{K-1}d(\gamma_{k}^{K},\gamma_{k+1}^{K}).
\]
We can extend $\{\gamma_{k}^{K}\}_{k=1}^{K}$ to a piecewise-constant
function from $[0,1]$ to $X$ having the same range, by setting $\gamma_{t}^{K}:=\gamma_{\lfloor tK\rfloor}^{K}$.
Note that $\gamma_{t}^{K}$ converges uniformly in $t$ to $\gamma$
as $K\rightarrow\infty$, since
\[
\sup_{t\in[0,1]}d(\gamma_{t},\gamma_{t}^{K})\leq\frac{1}{K-1}\text{Length}(\gamma).
\]
Hence we can apply Lemma \ref{lem:continuity-unpenalized-objective}
to deduce that
\[
\lim_{K,N\rightarrow\infty}\int_{X}d^{2}(x,\Gamma^{K})d\Lambda_{N}(x)=\int_{X}d^{2}(x,\Gamma)d\Lambda(x).
\]
It follows that
\[
\limsup_{K,N\rightarrow\infty}\left(\int_{X}d^{2}(x,\Gamma^{K})d\Lambda_{N}(x)+\beta\sum_{k=1}^{K-1}d(\gamma_{k}^{K},\gamma_{k+1}^{K})\right)\leq\int_{X}d^{2}(x,\Gamma)d\Lambda(x)+\beta\text{Length}(\gamma).
\]

\emph{Step 3}. Finally, given any constant speed $\gamma\in
AC([0,1];X)$, let $\gamma_{k}^{K}$ be defined as in Step 2; and, let
$\{\gamma_{k}^{*K}\}_{k=1}^{K}$ be a minimizer of
$\text{PPC}^{K}(\Lambda_{N})$, and let $\gamma^{*K}$ be a
constant-speed piecewise geodesic interpolation of
$\{\gamma_{k}^{*K}\}_{k=1}^{K}$.

First we claim that the $\gamma^{*K}$'s have uniformly bounded length. Note that for any $K\in\mathbb{N}$, we can always consider a competitor $\{\bar{\gamma}_k^K\}_{k=1}^K$ where
all $K$ knots are located at the same (arbitrary) point
in $X$; in this case,
$\sum_{k=0}^{K-1}d(\bar{\gamma}_{k}^{K},\bar{\gamma}_{k+1}^{K})=0$ and
$\int_{X}d^{2}(x,\bar\Gamma^{K})d\Lambda_{N}(x)\leq\text{Diam}(X)^{2}$. It
therefore follows, from the fact that
$\{\overset{*}{\gamma}_{k}^{K}\}_{k=1}^{K}$ is a minimizer of
$\text{PPC}_{K}(\Lambda_{N})$, that
\[
\text{Diam}(X)^{2}\geq\int_{X}d^{2}(x,\Gamma^{*K})d\Lambda_{N}(x)+\beta\sum_{k=0}^{K-1}d\left(\gamma_{k}^{*K},\gamma_{k+1}^{*K}\right)\geq\beta\sum_{k=0}^{K-1}d\left(\gamma_{k}^{*K},\gamma_{k+1}^{*K}\right).
\]
Consequently, $\text{Length}(\gamma^{*K})=\sum_{k=0}^{K-1}d\left(\gamma_{k}^{*K},\gamma_{k+1}^{*K}\right)\leq\beta^{-1}\text{Diam}(X)^{2}$
uniformly in $K$. So by Lemma \ref{lem:length-is-lsc} we can pass
to a subsequence $K_{j}$ and extract a limiting AC curve $\gamma^{*}$
such that $\gamma_{t}^{*K_{j}}\rightarrow\gamma_{t}^{*}$ uniformly
in $t$; and moreover,
\[
\liminf_{K_{j}\rightarrow\infty}\text{Length}(\overset{*}{\gamma}^{K_{j}})\geq\text{Length}(\gamma^{*}).
\]

More generally,
let $\{\gamma_{k}^{*K_{j}}\}_{k=1}^{K_{j}}$
be any subsequence of minimizers of $\text{PPC}^{K_{j}}(\Lambda_{N})$,
with geodesic interpolations $\gamma^{*K_{j}}$, such that
$\gamma_{t}^{*K_{j}}$ converges uniformly in $t$ to some
$\gamma^{*}\in AC([0,1];X)$. We deduce that
\begin{multline*}
\int_{X}d^{2}(x,\Gamma^{*})d\Lambda(x)+\beta\text{Length}(\gamma^{*}) \\
\begin{aligned}
& \leq\liminf_{N,K_{j}\rightarrow\infty}\left(\int_{X}d^{2}(x,\Gamma^{*K_{j}})d\Lambda_{N}(x)+\beta\sum_{k=1}^{K_{j}-1}d(\gamma_{k}^{*K_{j}},\gamma_{k+1}^{*K_{j}})\right)\\
 & \leq\limsup_{N,K_{j}\rightarrow\infty}\left(\int_{X}d^{2}(x,\Gamma^{K_{j}})d\Lambda_{N}(x)+\beta\sum_{k=1}^{K_{j}-1}d(\gamma_{k}^{K_{j}},\gamma_{k+1}^{K_{j}})\right)\\
 & \leq\int_{X}d^{2}(x,\Gamma)d\Lambda(x)+\beta\text{Length}(\gamma)
\end{aligned}
\end{multline*}
where, in the computation above, we have used Step 1 in the first
inequality, the fact that $\{\gamma_{k}^{*K}\}_{k=1}^{K}$ is a minimizer
in the second inequality, and Step 2 in the third inequality. As $\gamma$
was chosen arbitrarily (up to time reparametrization), this shows
that $\gamma^{*}$ is a minimizer of $\text{PPC}(\Lambda)$ as desired.
\end{proof}

\subsection{\texorpdfstring{Proofs for Section \ref{sec:Wasserstein}}{Proofs for Section 3}}
\label{subsec:proofs-for-Wasserstein}
\begin{lem}
\label{lem:doubly empirical random measure}Let $V$ be a complete
metric space. Let $d_{\mathcal{P}(V)}$ be a complete separable metric
on $\mathcal{P}(V)$. Suppose that for each $m=1,\ldots,M$, we have
$\mu_{m},\mu_{m}^{\prime}\in\mathcal{P}(V)$, and $d_{\mathcal{P}(V)}(\mu_{m},\mu_{m}^{\prime})<\varepsilon$.
Then,
\[
\mathbb{W}_{1}\left(\frac{1}{M}\sum_{m=1}^{M}\delta_{\mu_{m}},\frac{1}{M}\sum_{m=1}^{M}\delta_{\mu_{m}^{\prime}}\right)<\varepsilon.
\]
\end{lem}

\begin{proof}
Consider the map $T$ which sends $\delta_{\mu_{m}}$ to $\delta_{\mu_{m}^{\prime}}$;
with respect to this map, we have that
\[
\int_{\mathcal{P}(V)}d_{\mathcal{P}(V)}(\nu,T(\nu))d\left(\frac{1}{M}\sum_{m=1}^{M}\delta_{\mu_{m}}\right)(\nu)=\frac{1}{M}\sum_{m=1}^{M}d_{\mathcal{P}(V)}(\mu_{m},\mu_{m}^{\prime})<\varepsilon.
\]
This directly implies that $\mathbb{W}_{1}\left(\frac{1}{M}\sum_{m=1}^{M}\delta_{\mu_{m}},\frac{1}{M}\sum_{m=1}^{M}\delta_{\mu_{m}^{\prime}}\right)<\varepsilon$
as desired.
\end{proof}

\begin{proof}[Proof of Theorem \ref{thm:iterated-glivenko-cantelli}]
Assume that $d_{\mathcal{P}(V)}$ metrizes weak{*} convergence. In
what follows, we condition on the event that $\mathbb{W}_{1}(\Lambda_{N},\Lambda)\rightarrow0$,
which is guaranteed to have probability 1 by the Glivenko-Cantelli theorem applied to $\mathcal{P}(\mathcal{P}(V))$.
From the triangle inequality, it then suffices to show that $\mathbb{W}_{1}(\Lambda_{N},\hat{\Lambda}_{N,M})\rightarrow0$, either in probability or a.s. respectively. We take $M$ to depend on $N$, and assume that
$M\rightarrow\infty$ as $N\rightarrow\infty$.

We take $d_{\mathcal{P}(V)}$ to be a maximum mean discrepancy (MMD)
distance $\text{MMD}_{\mathcal{H}}$ coming from a reproducing kernel
Hilbert space $\mathcal{H}$ of functions from $V$ to $\mathbb{R}$,
for which $\text{MMD}_{\mathcal{H}}$ metrizes weak{*} convergence
on $\mathcal{P}(V)$; the existence of such an RKHS is guaranteed
by the results described in Appendix \ref{sec:rkhs}.

From Lemma \ref{lem:doubly empirical random measure}, we have that
\[
\mathbb{E}\left[\mathbb{W}_{1}(\Lambda_{N},\hat{\Lambda}_{N,M})\right]\leq\frac{1}{N}\sum_{n=1}^{N}\mathbb{E}\left[\text{MMD}_{\mathcal{H}}(\mu_{n},\hat{\mu}_{n})\right].
\]
From Appendix \ref{sec:rkhs} we have that $\mathbb{E}\left[\text{MMD}_{\mathcal{H}}(\mu_{n},\hat{\mu}_{n})\right]=O(M^{-1/2})$,
and so $\mathbb{E}\left[\mathbb{W}_{1}(\Lambda_{N},\hat{\Lambda}_{N,M})\right]\rightarrow0$
as $N,M\rightarrow\infty$. This implies that there exists a subsequence
of $N,M$ (which we do not relabel) along which $\mathbb{W}_{1}(\Lambda_{N},\hat{\Lambda}_{N,M})\rightarrow0$
almost surely; it follows that along this subsequence, $\mathbb{W}_{1}(\hat{\Lambda}_{N,M},\Lambda)\rightarrow0$
almost surely as desired.

For the second claim, we study the concentration of $\text{MMD}_{\mathcal{H}}(\mu_{n},\hat{\mu}_{n})$
around its expectation. By Appendix \ref{sec:rkhs}, we have for all
$t>0$,
\[
\mathbb{P}\left[\text{MMD}_{\mathcal{H}}(\mu_{n},\hat{\mu}_{n})\geq t+\mathbb{E}[\text{MMD}_{\mathcal{H}}(\mu_{n},\hat{\mu}_{n})]\right]\leq\exp\left(-\frac{Mt^{2}}{D_{\mathcal{H}}}\right)
\]
for some constant $D_{\mathcal{H}}>0$. Take $t=M^{-\gamma}$ for
some $\gamma\in(0,1/2)$; with probability at least $1-N\exp\left(-\frac{M^{1-2\gamma}}{D_{\mathcal{H}}}\right)$,
it then holds that
\[
\mathbb{W}_{1}(\Lambda_{N},\hat{\Lambda}_{N,M})\leq M^{-\gamma}+\frac{1}{N}\sum_{n=1}^{N}\mathbb{E}[\text{MMD}_{\mathcal{H}}(\mu_{n},\hat{\mu}_{n})].
\]
By the Borel-Cantelli lemma, this inequality holds almost surely for
all but finitely many $N$ (which implies that $\mathbb{W}_{1}(\Lambda_{N},\hat{\Lambda}_{N,M})\rightarrow0$
almost surely) provided that
\[
\sum_{N=1}^{\infty}N\exp\left(-\frac{M^{1-2\gamma}}{D_{\mathcal{H}}}\right)<\infty.
\]
In turn, it suffices that for some $\gamma\in(0,1/2)$ and $\beta>1$,
we have $N\exp\left(-\frac{M^{1-2\gamma}}{D_{\mathcal{H}}}\right)\leq N^{-\beta}$
for $N$ and $M$ sufficiently large, in other words,
\[
M\geq\left(D_{\mathcal{H}}(\beta+1)\log N\right)^{1/(1-2\gamma)}.
\]
This holds $N$ and $M$ sufficiently large as long as there exists
constants $C>0$ and $q>1$ such that $M\geq C(\log N)^{q}$.

(Lastly, in the case where the number of samples $M_n$ in $\mu_n$ depends on the index $n$, the argument above still applies if we take $M:=\min_{1\leq n \leq N}$: this is because $M^{-1/2}\geq M_n^{-1/2}$ and similarly for $\exp\left(-\frac{Mt^{2}}{D_{\mathcal{H}}}\right)$, so all the same estimates hold.)
\end{proof}

\begin{proof}[Proof of Proposition \ref{prop:finite-reads}]
As before, we condition on the event that $\mathbb{W}_{1}(\Lambda_{N},\Lambda)\rightarrow0$,
which is guaranteed to have probability 1 by Glivenko-Cantelli. It
then suffices to show that $\mathbb{W}_{1}(\Lambda_{N},\hat{\Lambda}_{N,M,R})\rightarrow0$
in probability.

From Lemma \ref{lem:doubly empirical random measure}, we have that
\[
\mathbb{W}_{1}(\Lambda_{N},\hat{\Lambda}_{N,M})\leq\frac{1}{N}\sum_{n=1}^{N}W_{1}(\mu_{n},\hat{\mu}_{n}^{M})
\]
and from Proposition 1.1 of \cite{boissard2014mean} we have that
$\mathbb{E}[W_{1}(\mu_{n},\hat{\mu}_{n}^{M})]\rightarrow0$ as $M\rightarrow\infty$,
at a uniform rate which depends only on $d$.

Now, let $\tilde{\mu}_n^{M,R}$ denote the finite-read approximation of $\hat{\mu}_n^M$ with $R$ many total reads (in the sense of equation (\ref{eq:noisy_empirical})). Applying Lemma \ref{lem:doubly empirical random measure}
again, we have that
\[
\mathbb{W}_{1}(\hat{\Lambda}_{N,M},\hat{\Lambda}_{N,M,R})\leq\frac{1}{N}\sum_{n=1}^{N}W_{1}(\hat{\mu}_{n}^{M},\tilde{\mu}_{n}^{M,R})
\]
and by \cite[Theorem 2.4]{kim2024optimal}, we have that $\mathbb{E}[W_{1}(\hat{\mu}_{n}^{M},\tilde{\mu}_{n}^{M,R})]\leq C\sqrt{M/R}+o_M(1)$
as $M\rightarrow\infty$ (in particular the estimate is uniform), as long
Assumption 2.3 from that article holds. It follows that, assuming $M/R\rightarrow 0$, we have
\begin{align*}
\mathbb{E}\left[\mathbb{W}_{1}(\Lambda_{N},\hat{\Lambda}_{N,M,R})\right] & \leq\mathbb{E}\left[\mathbb{W}_{1}(\Lambda_{N},\hat{\Lambda}_{N,M})\right]+\mathbb{E}\left[\mathbb{W}_{1}(\hat{\Lambda}_{N,M},\hat{\Lambda}_{N,M,R})\right]\\
 & \leq\frac{1}{N}\sum_{n=1}^{N}\mathbb{E}[W_{1}(\mu_{n},\hat{\mu}_{n}^{M})]+\mathbb{E}[W_{1}(\hat{\mu}_{n}^{M},\tilde{\mu}_{n}^{M,R})]\\
 & \rightarrow0.
\end{align*}
This implies that $\mathbb{W}_{1}(\Lambda_{N},\hat{\Lambda}_{N,M,R})\rightarrow0$
in probability, as desired.

To establish the second claim, we reason just as in the second half of the proof of Theorem \ref{thm:iterated-glivenko-cantelli} above. Namely, we need to establish that $\mathbb{W}_1(\Lambda_N,\hat{\Lambda}_{N,M,R})$ satisfies a Hoeffding-type concentration inequality, and then reason by applying the Borel-Cantelli lemma. Because the proof is so similar to that of the preceding theorem, we provide a sketch only.

It suffices to separately establish a concentration upper bound for $W_{1}(\mu_{n},\hat{\mu}_{n}^{M})$ and $W_{1}(\hat{\mu}_{n}^{M},\tilde{\mu}_{n}^{M,R})$. The first case is already established by \cite[Proposition 20]{weed2019sharp}, which shows that
\[
\mathbb{P}\left[W_{1}(\mu_{n},\hat{\mu}_{n}^{M})\geq \mathbb{E}[W_{1}(\mu_{n},\hat{\mu}_{n}^{M})]+t\right]\leq \exp \left(-\frac{2Mt^2}{\textrm{Diam}(V)^2}\right).
\]
(The proposition is stated for the case where the diameter of the underlying space is 1, but the adaptation is routine.)
It remains, therefore, only to provide a similar bound for $W_{1}(\hat{\mu}_{n}^{M},\tilde{\mu}_{n}^{M,R})$.

We reason by following carefully the proof of \cite[Theorem 2.4]{kim2024optimal}. In this proof, it is established that $\mathbb{E}[W_{1}(\hat{\mu}_{n}^{M},\tilde{\mu}_{n}^{M,R})]\leq \mathbb{E}\left[\frac{1}{M}\sum_{i=1}^MZ_i\right]$, where the random variables $Z_i$ are independent with range $[0,2]$, and are defined explicitly in the statement of \cite[Lemma 4.4]{kim2024optimal}. The bound in the statement of \cite[Theorem 2.4]{kim2024optimal} is then actually derived by upper bounding $\mathbb{E}\left[\frac{1}{M}\sum_{i=1}^MZ_i\right]$, in other words we also have that $\mathbb{E}\left[\frac{1}{M}\sum_{i=1}^MZ_i\right]\leq C \sqrt{(M/R)}+o_M(1)$. Furthermore, it holds that for any constant $c>0$,
\[
\mathbb{P}\left[W_{1}(\hat{\mu}_{n}^{M},\tilde{\mu}_{n}^{M,R})\geq c\right]\leq \mathbb{P}\left[\frac{1}{M}\sum_{i=1}^MZ_i \geq c\right]
\]
as can be seen by combining \cite[equation (4.1)]{kim2024optimal} with the final set of display equations of the proof of \cite[Theorem 2.4]{kim2024optimal}. This means that
\begin{align*}
\mathbb{P}\left[W_{1}(\hat{\mu}_{n}^{M},\tilde{\mu}_{n}^{M,R})\geq\mathbb{E}\left[\frac{1}{M}\sum_{i=1}^{M}Z_{i}\right]+t\right] & \leq\mathbb{P}\left[\frac{1}{M}\sum_{i=1}^{M}Z_{i}\geq\mathbb{E}\left[\frac{1}{M}\sum_{i=1}^{M}Z_{i}\right]+t\right]\\
\text{(Hoeffding)} & \leq\exp\left(-Mt^{2}\right).
\end{align*}

Altogether, we have $\mathbb{E}\left[\frac{1}{M}\sum_{i=1}^{M}Z_{i}\right]\rightarrow 0$ provided $M/R \rightarrow 0$, and that $W_{1}(\hat{\mu}_{n}^{M},\tilde{\mu}_{n}^{M,R})$ satisfies a Hoeffding-type tail bound w.r.t. $\mathbb{E}\left[\frac{1}{M}\sum_{i=1}^{M}Z_{i}\right]$ (depending only on $M$). Combining our two concentration bounds, by Borel-Cantelli we have that $\mathbb{W}_1(\Lambda_N,\hat{\Lambda}_{N,M,R})\rightarrow0$ almost surely, so long as $M\geq C(\log N)^q$, and also $M/R \rightarrow 0$.

\end{proof}

\subsection{\texorpdfstring{Proofs for Section \ref{sec:seriation}}{Proofs for Section 4}}
\label{subsec:proofs-for-seriation}
\begin{proof}[Proof of Proposition \ref{prop:lengthbound-1d}]
(i) Let $P$ denote the range of $\rho_{t}$ in $X$. Since $\Lambda$
is supported on $P$, we observe that $\int_{X}d^{2}(x,P)d\Lambda(x)=0$.
At the same time, since $\gamma^{*\beta}$ is a minimizer for $\text{PPC}(\Lambda;\beta)$,
we have that
\[
\int_{X}d^{2}(x,\Gamma^{*\beta})d\Lambda(x)+\beta\text{Length}(\gamma^{*\beta})\leq\int_{X}d^{2}(x,P)d\Lambda(x)+\beta\text{Length}(\rho)
\]
where $\Gamma^{*\beta}$ is the graph of $\gamma^*{\beta}$. It follows
that
\[
0\leq\int_{X}d^{2}(x,\Gamma^{*\beta})d\Lambda(x)\leq\beta\left(\text{Length}(\rho)-\text{Length}(\gamma^{*\beta})\right)
\]
which establishes the first claim.

(ii) Note that by change of variables for measures,
\[
\int_{X}d^{2}(x,\Gamma^{*\beta})d\Lambda(x)=\int_{0}^{1}d^{2}(\rho_{t},\Gamma^{*\beta})dt.
\]
Hence, our estimate from the proof of (i) shows that
\[
\int_{0}^{1}d^{2}(\rho_{t},\Gamma^{*\beta})dt\leq\beta\left(\text{Length}(\rho)-\text{Length}(\gamma^{*\beta})\right)\leq\beta\text{Length}(\rho).
\]
The claim now follows from Markov's inequality, which here indicates
that
\[
\text{Leb}_{[0,1]}\left\{ t:d^{2}(\rho_{t},\Gamma^{*\beta})\geq\alpha\right\} \leq\frac{1}{\alpha}\int_{0}^{1}d^{2}(\rho_{t},\Gamma^{*\beta})dt.
\]
\end{proof}

\begin{proof}[Proof of Theorem \ref{thm:seriation-consistency}]
In  Proposition~\ref{prop:lengthbound-1d}, we saw that
$\text{Length}(\gamma^{*\beta})\leq\text{Length}(\rho)$,
so by
 Arzela-Ascoli's theorem we can extract   a subsequence which converges, uniformly in $t$, to
a curve $\gamma^*\in AC([0,1];X)$.
By the lower semicontinuity of $\text{Length}(\cdot)$, we also have
that $\text{Length}(\gamma^*)\leq\text{Length}(\rho)$. At the same
time, we have the estimate
\[
\int_{X}d^{2}(x,\Gamma^{*\beta})d\Lambda(x)\leq\beta\text{Length}(\rho).
\]
Sending $\beta\rightarrow0$, we see from Lemma \ref{lem:continuity-unpenalized-objective}
that $\int_{X}d^{2}(x,\Gamma^*)d\Lambda(x)=0$.
In particular, $\gamma^*$ is a minimizer of $\text{PPC}(\Lambda;0)$.
More generally the same applies for any convergent subsequence
$(\gamma^{*\beta_{n}})_{n\in\mathbb{N}}$.

 We first check that $\Gamma^*\supseteq P$. Since $0=\int_{0}^{1}d^{2}(\rho_{t},\Gamma^*)dt\geq\int_{\{t:\rho_{t}\notin\Gamma^*\}}d^{2}(\rho_{t},\Gamma^*)dt,$ the set  $\{t:\rho_{t}\notin\Gamma^*\}$ has zero Lebesgue measure. On the other hand, it is  an open set in $[0,1]$ because $\gamma^*$ and $\rho$ are continuous from $[0,1]$. Therefore $\{t:\rho_{t}\notin\Gamma^*\}$  is an empty set, verifying   $\Gamma^*\supseteq \text{supp}(\Lambda)=P$.

Now, suppose for the sake of contradiction that $\Gamma^*\supsetneq P$.
In order to establish contradiction, it suffices to show that $\Gamma^*\supsetneq P$
implies $\text{Length}(\gamma^*)>\text{Length}(\rho)$. (This is more or less obvious for curves in Euclidean space, but since we work in a general compact metric space we give a verbose justification.) In this argument
we will use the assumption that $\rho$ is injective.

Observe that since $\Gamma$ is compact and $P$ is compact, $\{t:\gamma^*_{t}\notin P\}$
is open inside $[0,1]$. If $\Gamma^* \supsetneq P$ then $\{t:\gamma^*_{t}\notin P\}$ is non-empty, and so contains some
non-empty compact sub-interval $K\subset\{t:\gamma^*_{t}\notin P\}$.
Accordingly let $\gamma^*_{k_{0}} \in \{t:\gamma^*_{t}\notin P\}$; note this point is disjoint from $P$, since
$P$ is compact.

Define the sets
\[
T_{0}:=\{t\in[0,1]:t<k_{0},\gamma^*_{t}\in P\};\qquad
T_{1}:=\{t\in[0,1]:t>k_{0},\gamma^*_{t}\in P\}.
\]
Note that at least one of $T_{0}$ and $T_{1}$ must be non-empty
since $P\cap\Gamma\neq\emptyset$. Without loss of generality let
$T_{0}\neq\emptyset$. Let $t_{0}$ be the last time before $k_{0}$
such that $\gamma^*_{t_{0}}\in P$ (such a last time exists since $P\cap\Gamma^*$
is closed). Since $\gamma^*_{(\cdot)}$ is absolutely continuous, we
have
\[
\text{Length}(\gamma^*\llcorner[t_{0},k_{0}])\geq d(\gamma^*_{t_{0}},\gamma^*_{k_{0}})>0.
\]
Moreover, we have that
\[
\text{Length}(\gamma^*)\geq\text{Length}(\gamma^*\llcorner[0,t_{0}])+\text{Length}(\gamma^*\llcorner[t_{0},k_{0}]).
\]
We now consider two cases. If $T_{1}$ is empty, then by construction
the graph of $\gamma^*\llcorner[0,t_{0}]$ contains $P$. But since
$\text{Length}(\rho)=\mathcal{H}^{1}(P)$ and $\text{Length}(\gamma^*\llcorner[0,t_{0}])\geq\mathcal{H}^{1}(\Gamma^*\llcorner[0,t_{0}])$ by Lemma \ref{lem:hausdorff equals length},
this establishes that $\text{Length}(\gamma^*\llcorner[0,t_{0}])\geq\text{Length}(\rho)$.
Therefore $\text{Length}(\gamma^*)>\text{Length}(\rho)$ which is a
contradiction. (The case where $T_{0}$ is empty and $T_{1}$ is non-empty
is symmetric.)

Now we consider the case where $T_{0}$ and $T_{1}$ are both non-empty,
which is very similar. Let $t_1$ denote the first time after $k_0$ that $\gamma^*_{t_1}\in T_1$. It holds that
\[
\text{Length}(\gamma^*\llcorner[t_{0},k_{0}])\geq d(\gamma^*_{t_{0}},\gamma^*_{k_{0}})>0,
\]
\[
\text{Length}(\gamma^*\llcorner[k_{0},t_{1}])\geq d(\gamma^*_{k_{0}},\gamma^*_{t_{1}})>0,
\]
and
\begin{align*}
\text{Length}(\gamma^*)
&= \text{Length}(\gamma^*\llcorner[0,t_{0}])+\text{Length}(\gamma^*\llcorner[t_{0},k_{0}])\\
& \quad +\text{Length}(\gamma^*\llcorner[k_{0},t_{1}])+\text{Length}(\gamma^*\llcorner[t_{1},1])
\end{align*}
while by construction the (disjoint) union $\Gamma^*\llcorner[0,t_{0}]\cup\Gamma^*\llcorner[t_{1},1]$
contains $P$. Hence by Lemma \ref{lem:hausdorff equals length},
\begin{align*}
\text{Length}(\gamma^*\llcorner[0,t_{0}])+\text{Length}(\gamma^*\llcorner[t_{1},1]) & \geq\mathcal{H}^{1}(\Gamma^*\llcorner[0,t_{0}]\cup\Gamma^*\llcorner[t_{1},1])\\
 & \geq\mathcal{H}^{1}(P)\\
 & =\text{Length}(\rho).
\end{align*}
This likewise shows that $\text{Length}(\gamma^*)>\text{Length}(\rho)$,
as desired.

It remains only to argue that $\gamma^*$ coincides with a reparametrization of $\rho$ which is ether order-preserving or order-reversing. Let $\hat{\gamma}^*$ and $\hat{\rho}$ denote the constant speed reparametrizations of $\gamma^*$ and $\rho$ respectively. Of course, $\mathrm{Length}(\hat{\rho})=\mathrm{Length}(\hat{\gamma}^*)$ (by Lemma \ref{lem:hausdorff equals length}). It follows from Lemma \ref{lem:arclen-most-efficient} that $\hat{\gamma}^*=\hat{\rho}$ up to time-reversal. Therefore, since $\hat{\gamma}^*$ is an order-preserving/reversing reparametrization of $\gamma^*$ (and vice versa), it follows that $\gamma^*$ is either an order-preserving or order-reversing reparametrization of $\hat{\rho}$, and hence of $\rho$.
\end{proof}

\begin{proof}[Proof of Proposition
  \ref{prop:pseudotimes-consistency}]

    Let $\tilde{\rho}_t$ be the reparametrization of $\rho_t$
    guaranteed by \cref{thm:seriation-consistency}, and without the
    loss of generality, assume $\tilde{\rho}_t$ is order-preserving.
    Since the inverse map $\tilde \rho_t \mapsto t$ is continuous
    (\cref{prop:inv-cont}) and $\Gamma$ is compact, $\tilde \rho_t
    \mapsto t$ is in fact \emph{uniformly} continuous; that is, for
    all $\epsilon > 0$, there exists $\delta(\epsilon) > 0$ such
    that
    \[
      \tag{$\forall s,t \in [0,1]$} d(\rho_t, \rho_s) <
      \delta(\epsilon) \implies |t-s| < \epsilon.
    \]
    Fix an $\epsilon_0 > 0$ with $\min_{i,j} d(\rho_i, \rho_j) >
    4\epsilon_0$ and let $\epsilon < \min \{\epsilon_0, \frac{1}{2}
    \delta(\epsilon_0/\mathrm{Length}(\rho))\}$. Again by
    \cref{thm:seriation-consistency}, taking $\beta > 0$ sufficiently
    small yields $\sup_{t\in[0,1]} d(\gamma_t^{*\beta},
    \tilde{\rho}_t) < \epsilon$.

    For all $i$, let $t_i \in [0,1]$ such that $\tilde{\rho}_{t_i} =
    \rho_i$, and also let $\hat{\tau}_i \in \argmin_{t \in [0,1]}
    d(\gamma_t^{*\beta}, \tilde{\rho}_{t_i})$. It suffices to show
    that for all $i,j$, we have $t_i < t_j$ iff $\hat{\tau}_i <
    \hat{\tau}_j$. To whit, applying the reverse triangle inequality at
    ($a$), ($b$) below yields
    \begin{align*}
      d(\gamma_{t_i}^{*\beta}, \gamma_{t_j}^{*\beta}) \overset{(a)}{\geq}
      |d(\gamma_{t_i}^{*\beta}, \tilde{\rho}_j) - d(\tilde{\rho}_j,
      \gamma_{t_j}^{*\beta})|
      \geq d(\gamma_{t_i}^{*\beta}, \tilde{\rho}_j)
      \overset{(b)}{\geq} |d(\gamma_{t_i}^{*\beta}, \tilde{\rho}_i) -
        d(\tilde{\rho}_i, \tilde{\rho}_j)| \overset{(c)}{\geq}
      3\epsilon_0,
    \end{align*}
    with ($c$) following from
    $d(\gamma_{t_i}^{*\beta}, \tilde{\rho}_{t_i}) < \epsilon <
    \epsilon_0$ and $d(\tilde{\rho}_i, \tilde{\rho}_j) >
    4\epsilon_0$. Since $\gamma_t^{*\beta}$ is parametrized to have constant speed,
    the overall inequality implies $|t_i - t_j| \geq
    3\epsilon_0/\mathrm{Length}(\rho)$.

    On the other hand,
    \[
      d(\tilde{\rho}_{\hat{\tau}_i}, \tilde{\rho}_{t_i}) \leq
      d(\tilde{\rho}_{\hat{\tau}_i}, \gamma_{\hat{\tau}_i}^{*\beta}) +
      d(\gamma_{\hat{\tau}_i}^{*\beta}, \tilde{\rho}_{t_i})
      \overset{(d)}{\leq}
      d(\tilde{\rho}_{\hat{\tau}_i}, \gamma_{\hat{\tau}_i}^{*\beta}) +
      d(\gamma_{t_i}^{*\beta}, \tilde{\rho}_{t_i}) \overset{(e)}{<}
      \delta(\epsilon_0/\mathrm{Length}(\rho)),
    \]
    where ($d$) follows from construction of $\hat{\tau}_i$ and ($e$)
    follows from $\sup_t d(\gamma_t^{*\beta}, \tilde{\rho}_t) <
    \epsilon$. By definition of $\delta$, this implies $|\hat{\tau}_i -
    t_i| \leq \epsilon_0/\mathrm{Length}(\rho)$. The symmetric
    argument gives $|t_j - \hat{\tau}_j| <
    \epsilon_0/\mathrm{Length}(\rho)$ as well. Together with the
    earlier bound $|t_i - t_j| \geq 3\epsilon_0/\mathrm{Length}(\rho)$
    we see $t_i < t_j$ iff $\hat{\tau}_i < \hat{\tau}_j$, as desired.
\end{proof}

\section{\texorpdfstring{Extensions for Section \ref{sec:basics}}{Extensions for Section 2}}
\label{appn:practical_model}
\label{sec:extensions}

In this appendix, we discuss some extra modifications one can make to the principal curve variational problem and its discretization, beyond what has been explained in the main text.

\subsection{Nonlocal discretization}

Here we consider a nonlocal discretization scheme for the $\textrm{PPC}$ functional,
similar to what was originally proposed in practice in \cite{hastie1989} (and likewise is the scheme used in the R principal
curves package). In experiments, we have found such schemes to be useful at the finite-data level since they allow for stable performance with a greater number of knots given a fixed amount of data: in other words, they allow us to increase the resolution of the discretization of the principal curve. Of separate interest, it has been suggested (see discussion in \cite{hastie1989}) that this type of nonlocal kernel smoothing scheme might induce a form of \emph{implicit regularization}, even in the absence of an explicit length penalty. We have not yet been able to verify this claimed regularization effect mathematically, but leave this interesting direction as one for future work.

We consider the following rather general \emph{nonlocal, nonuniform}
smoothing kernel. Let $w:\mathbb{R}_{+}\rightarrow\mathbb{R}_{+}$
be a Borel function which is compactly supported on $[0,1]$. We write
\[
w_{h}(t):=\frac{1}{h}w\left(\frac{t}{h}\right)
\]
for $h>0$. Note $w_{h}$ is compactly supported on $[0,h]$. In practice,
we suggest a $w$ of the form   $(1-|t|^{p})_+^{q}$   (a.k.a. an
``Epanechnikov kernel''), which is a standard
choice for kernel smoothers used in nonparametric regression.  For simplicity, in what follows we only consider the case
where $\Lambda_{N}:=\frac{1}{N}\sum_{n=1}^{N}\delta_{x_{n}}$ for
points $x_{n}\in X$, and where $\Lambda_{N}$ converges to some limiting
measure $\Lambda$ as $N\rightarrow\infty$.

Now we consider the nonlocal, discrete objective
\begin{align*}
&\text{PPC}_{w}^{K}(\Lambda_{N}) (\gamma_1, ..., \gamma_K)\\
&:=\frac{1}{N}\sum_{k=1}^{K}\sum_{j=1}^{K}\sum_{x_{n}\in I_{j}}d^{2}(x_{n},\gamma_{k})\frac{1}{C_{j}}w_{h_{j}}\left(\frac{\widehat{\text{Dist}}(\gamma_{j},\gamma_{k})}{\sum_{i=1}^{K-1}d(\gamma_{i},\gamma_{i+1})}\right)+\beta\sum_{i=1}^{K-1}d(\gamma_{i},\gamma_{i+1})
\end{align*}
where $\widehat{\text{Dist}}(\gamma_{j},\gamma_{k})$ denotes the
arcwise distance from $\gamma_{j}$ to $\gamma_{k}$ along the discrete
curve $\{\gamma_{k}\}_{k=1}^{K}$; in other words,
\[
\widehat{\text{Dist}}(\gamma_{j},\gamma_{k}):=\begin{cases}
\sum_{i=j}^{k-1}d(\gamma_{i},\gamma_{i+1}) & j<k\\
\sum_{i=k}^{j-1}d(\gamma_{i},\gamma_{i+1}) & k<j\\
0 & j=k.
\end{cases}
\]
Here $C_{j}$ is a normalization constant that is chosen so that $\frac{1}{C_{j}}\sum_{k=1}^{K}w_{h_{j}}\left(\frac{\widehat{\text{Dist}}(\gamma_{j},\gamma_{k})}{\sum_{i=1}^{K-1}d(\gamma_{i},\gamma_{i+1})}\right) \leq 1$.
Note that for each $j$ we allow the choice of a different $h_{j}$;
for example, each $h_{j}$ can be chosen adaptively given the data.
In what follows, we write $\bar{h}:=\max_{j\in[1,K]}h_{j}$.

\begin{prop}
\label{prop:discrete-nonlocal-to-continuum}
Let $X$ be a compact geodesic metric space. Minimizers of the functional
$\text{PPC}_{w}^{K}(\Lambda_{N})$ converge to minimizers of $\text{PPC}(\Lambda)$
as $K,N\rightarrow\infty$ and $\bar{h}\rightarrow0$, in the same
sense as for Theorem \ref{thm:discrete-to-continuum}.
\end{prop}

We give a proof of this proposition momentarily.

The nonlocal discrete objective $\text{PPC}_w^K(\Lambda_N)$ can be used in place of $\text{PPC}^K(\Lambda_N)$ in our coupled Lloyd's algorithm (Algorithm \ref{alg:coupled-lloyd}); the interpretation is that when updating the knots $\gamma_k$, we do not just take into account the Voronoi cell $I_k$ but also give some weight to data located in Voronoi cells $I_j$ for nearby knots $\gamma_j$. For completeness, we state explicitly what the algorithm looks like with this alternate objective.

\begin{algorithm}[H]
  \SetKwComment{Comment}{/* }{ */}
  \caption{Nonlocal Coupled Lloyd's Algorithm for Principal Curves}\label{alg:coupled-lloyd-nonlocal}

  \SetKwInOut{Input}{Input}
  \Input{$\textrm{data }\{x_n\}_{n=1}^N, \textrm{parameters } \beta>0, \{h_k\}_{k=1}^K, \varepsilon>0$}

  \SetKwFunction{Initialize}{initialize\_knots}

  $\{\gamma_k\}_{k=1}^K \gets \Initialize{}$\;

  \Repeat{$\varepsilon\textrm{-convergence}$}{
   $\{\gamma_k\}_{k=1}^K \gets
    \texttt{TSP\_ordering}(\{\gamma_k\}_{k=1}^K)$
    \Comment*[r]{min-length ordering}
   $\{I_k\}_{k=1}^K \gets
    \texttt{compute\_Voronoi\_cells}(\{\gamma_k\}_{k=1}^K)$\;
       $\{\gamma_k\}_{k=1}^K \gets
    \argmin\limits_{\{\gamma'_k\}_{k=1}^K }
    \text{PPC}_{w}^{K}(\Lambda_{N}) (\gamma'_1, ..., \gamma'_K)$ \texttt{using $I_k$'s from 4}\;
  }

  \KwResult{$\{\gamma_k\}_{k=1}^K$ \Comment*[r]{The updated output knots}}

\end{algorithm}

\begin{proof}[Proof of Proposition \ref{prop:discrete-nonlocal-to-continuum}]
Our proof strategy is to reduce the convergence of minimizers of the nonlocal
discrete functional to the same result for the discrete functional, which was already established in the proof of Theorem \ref{thm:discrete-to-continuum}.

Let  $\{\gamma_1, ..., \gamma_K\}\subset X$.
 For simplicity,  denote for the first term in the functional $\text{PPC}_{w}^{K}(\Lambda_{N})$,
\begin{align*}
  S_{N,K} : =  \frac{1}{N}\sum_{k=1}^{K}\sum_{j=1}^{K}\sum_{x_{n}\in I_{j}}d^{2}(x_{n},\gamma_{k})\frac{1}{C_{j}}w_{h_{j}}\left(\frac{\widehat{\text{Dist}}(\gamma_{j},\gamma_{k})}{\sum_{i=1}^{K-1}d(\gamma_{i},\gamma_{i+1})}\right).
\end{align*}

\emph{Step 1}. From the fact that $I_{j}$
is the Voronoi cell for $\gamma_{j}$, it holds that
\[
\sum_{x_{n}\in I_{j}}d^{2}(x_{n},\gamma_{k})\geq\sum_{x_{n}\in I_{j}}d^{2}(x_{n},\gamma_{j})
\]
and so
\begin{align*}
 S_{N,K} & \geq\frac{1}{N}\sum_{k=1}^{K}\sum_{j=1}^{K}\sum_{x_{n}\in I_{j}}d^{2}(x_{n},\gamma_{j})\frac{1}{C_{j}}w_{h_{j}}\left(\frac{\widehat{\text{Dist}}(\gamma_{j},\gamma_{k})}{\sum_{i=1}^{K-1}d(\gamma_{i},\gamma_{i+1})}\right)\\
 & =\frac{1}{N}\sum_{j=1}^{K}\sum_{x_{n}\in I_{j}}d^{2}(x_{n},\gamma_{j})\frac{1}{C_{j}}\sum_{k=1}^{K}w_{h_{j}}\left(\frac{\widehat{\text{Dist}}(\gamma_{j},\gamma_{k})}{\sum_{i=1}^{K-1}d(\gamma_{i},\gamma_{i+1})}\right).
\end{align*}
Since we chose $C_{j}$ so that
\[
\frac{1}{C_{j}}\sum_{k=1}^{K}w_{h_{j}}\left(\frac{\widehat{\text{Dist}}(\gamma_{j},\gamma_{k})}{\sum_{i=1}^{K-1}d(\gamma_{i},\gamma_{i+1})}\right) \leq 1,
\]
this ensures that for any $\{\gamma_{k}\}_{k=1}^{K}$ and $\{x_{n}\}_{n=1}^{N}$
atoms in $\Lambda_{N}$, and any $\bar{h}$,

\begin{align*}
 S_{N,K}
+\beta\sum_{i=1}^{K-1}d(\gamma_{i},\gamma_{i+1})
\geq\frac{1}{N}\sum_{k=1}^{K}\sum_{x_{n}\in I_{k}}d^{2}(x_{n},\gamma_{k})+\beta\sum_{i=1}^{K-1}d(\gamma_{i},\gamma_{i+1}).
\end{align*}
Accordingly, we can apply Step 1 from the proof of Theorem \ref{thm:discrete-to-continuum},
already established, to deduce that,
 if $\gamma^{K}$, the piecewise geodesic curve from $\{\gamma_1, ..., \gamma_K\}$,  converges uniformly in $t$ to some limiting
AC curve $\gamma$, then
\begin{multline*}
\liminf_{N,K\rightarrow\infty;\bar{h}\rightarrow0}\left(
S_{N,K} +\beta\sum_{i=1}^{K-1}d(\gamma_{i},\gamma_{i+1})\right)\\
\begin{aligned}
& \geq\liminf_{N,K\rightarrow\infty}\left(\frac{1}{N}\sum_{k=1}^{K}\sum_{x_{n}\in I_{k}}d^{2}(x_{n},\gamma_{k})+\beta\sum_{i=1}^{K-1}d(\gamma_{i},\gamma_{i+1})\right)\\
 & \geq\int_{X}d^{2}(x,\Gamma)d\Lambda(x)+ \beta \text{Length}(\gamma).
\end{aligned}
\end{multline*}
(In fact the same inequality holds even if $\bar{h}$ is not sent
to zero; but we do not use this.)

\emph{Step 2}.
We observe that
\[
d^{2}(x_{n},\gamma_{k})\leq d^{2}(x_{n},\gamma_{j})+2d(x_{n},\gamma_{j})d(\gamma_{j},\gamma_{k})+d^{2}(\gamma_{j},\gamma_{k}).
\]
By assumption on $X$, we have that $d(x_{n},\gamma_{j})\leq\text{Diam}(X)$;
and, since we have assumed that  $w_{h_j}$ is compactly supported on $[0,h_j]$,
either
\[
\frac{\widehat{\text{Dist}}(\gamma_{j},\gamma_{k})}{\sum_{i=1}^{K-1}d(\gamma_{i},\gamma_{i+1})}<h_{j},\text{ or }w_{h_{j}}\left(\frac{\widehat{\text{Dist}}(\gamma_{j},\gamma_{k})}{\sum_{i=1}^{K-1}d(\gamma_{i},\gamma_{i+1})}\right)=0.
\]
Additionally, we have that $d(\gamma_{j},\gamma_{k})\leq\widehat{\text{Dist}}(\gamma_{j},\gamma_{k})$
simply from the triangle inequality. Thus, since $h_{j}\leq\bar{h}$,
we have that, for all $j$ and $k$ such that $w_{h_{j}}\left(\frac{\widehat{\text{Dist}}(\gamma_{j},\gamma_{k})}{\sum_{i=1}^{K-1}d(\gamma_{i},\gamma_{i+1})}\right)>0$,
\begin{align*}
d^{2}(x_{n},\gamma_{k}) & \leq d^{2}(x_{n},\gamma_{j})+2d(x_{n},\gamma_{j})\widehat{\text{Dist}}(\gamma_{j},\gamma_{k})+\left(\widehat{\text{Dist}}(\gamma_{j},\gamma_{k})\right)^{2}\\
 & \leq d^{2}(x_{n},\gamma_{j})+2\text{Diam}(X)\bar{h}\sum_{i=1}^{K-1}d(\gamma_{i},\gamma_{i+1})+\left(\bar{h}\sum_{i=1}^{K-1}d(\gamma_{i},\gamma_{i+1})\right)^{2}.
\end{align*}
Now, let us assume that $\sum_{i=1}^{K-1}d(\gamma_{i},\gamma_{i+1})\leq\text{Diam}(X)^{2}/\beta$;
for establishing convergence of minimizers, this shall result in no
loss of generality, as we have already seen in previous arguments.
Under this assumption,
\[
d^{2}(x_{n},\gamma_{k})\leq d^{2}(x_{n},\gamma_{j})+2\text{Diam}(X)^{3}\bar{h}/\beta+\bar{h}^{2}\text{Diam}(X)^{4}/\beta^2.
\]
It follows that, for any $\{\gamma_{k}\}_{k=1}^{K}$, $\Lambda_{N}$,
and $\bar{h}$,

\begin{align*}
 S_{N,K} & \leq\frac{1}{N}\sum_{k=1}^{K}\sum_{j=1}^{K}\sum_{x_{n}\in I_{j}}\left(d^{2}(x_{n},\gamma_{j})+ O(\bar h)
\right)\frac{1}{C_{j}}w_{h_{j}}\left(\frac{\widehat{\text{Dist}}(\gamma_{j},\gamma_{k})}{\sum_{i=1}^{K-1}d(\gamma_{i},\gamma_{i+1})}\right)\\
 & =\frac{1}{N}\sum_{j=1}^{K}\sum_{x_{n}\in I_{j}}\left(d^{2}(x_{n},\gamma_{j})+
 O(\bar h)
 \right)\frac{1}{C_{j}}\sum_{k=1}^{K}w_{h_j}\left(\frac{\widehat{\text{Dist}}(\gamma_{j},\gamma_{k})}{\sum_{i=1}^{K-1}d(\gamma_{i},\gamma_{i+1})}\right)\\
 & \leq \frac{1}{N}\sum_{j=1}^{K}\sum_{x_{n}\in I_{j}}\left(d^{2}(x_{n},\gamma_{j})
 + O(\bar h)
 \right)\\
 & =\left(\frac{1}{N}\sum_{j=1}^{K}\sum_{x_{n}\in I_{j}}d^{2}(x_{n},\gamma_{j})
 \right)+ O(\bar h).
\end{align*}
For the last line, notice that the double summation has only $N$ effective terms.

Consequently,

\begin{multline*}
\limsup_{N,K\rightarrow\infty;\bar{h}\rightarrow0}\left(
 S_{N,K}+\beta\sum_{i=1}^{K-1}d(\gamma_{i},\gamma_{i+1})\right)\\
\begin{aligned}
&
\leq\limsup_{N,K\rightarrow\infty;\bar{h}\rightarrow0}\left(\frac{1}{N}\sum_{j=1}^{K}\sum_{x_{n}\in I_{j}}\left(d^{2}(x_{n},\gamma_{j})+
O(\bar h)
\right)
+\beta\sum_{i=1}^{K-1}d(\gamma_{i},\gamma_{i+1})\right)\\
 & \leq\limsup_{N,K\rightarrow\infty}\left(\frac{1}{N}\sum_{j=1}^{K}\sum_{x_{n}\in I_{j}}d^{2}(x_{n},\gamma_{j})+\beta\sum_{i=1}^{K-1}d(\gamma_{i},\gamma_{i+1})\right).
\end{aligned}
\end{multline*}
Finally, applying Step 2 from the proof of Theorem \ref{thm:discrete-to-continuum}, for the specific choice of $\{\gamma_k\}_{k=1}^K$ from that Step 2, it holds that
\[
\limsup_{N,K\rightarrow \infty}\left(\frac{1}{N}\sum_{j=1}^{K}\sum_{x_{n}\in I_{j}}d^{2}(x_{n},\gamma_{j})+\beta\sum_{i=1}^{K-1}d(\gamma_{i},\gamma_{i+1})\right)\leq\int_{X}d^{2}(x,\Gamma)d\Lambda(x)+\beta\text{Length}(\gamma).
\]

\medskip

 \emph{Step 3.}
Lastly, the conclusion of the proof is now identical to Step 3 of the proof of Theorem \ref{thm:discrete-to-continuum}.
\end{proof}

\newpage

\subsection{Fixed endpoints and semi-supervision}

\subsubsection*{Fixed endpoints} Here we want to produce a
curve which best fits the data (subject to length penalty), but with
endpoints $\bar{\gamma}_{0}$ and $\bar{\gamma}_{1}$ which are specified
in advance. In other words, we solve the modified optimization problem
\[
\min_{\gamma\in AC([0,1];X)}\left\{ \int_{X}d^{2}(x,\Gamma)d\Lambda(x)+\beta\text{Length}(\gamma):\gamma_{0}=\bar{\gamma}_{0},\gamma_{1}=\bar{\gamma}_{1}\right\} .
\]
We note that the set $\{\gamma\in AC([0,1];X):\gamma_{0}=\bar{\gamma}_{0},\gamma_{1}=\bar{\gamma}_{1}\}$
is closed inside $AC([0,1];X)$, so
one can prove existence of minimizers by an identical argument to
the one we provide in the proof of Proposition \ref{prop:mins-exist}.
Likewise, the corresponding optimization problem for discretized $\gamma$
takes the form
\[
\min_{\gamma_{1},\ldots,\gamma_{K}\in X}\left\{ \int_{X}d^{2}(x,\Gamma^{K})d\Lambda(x)+\beta\sum_{k=1}^{K-1}d(\gamma_{k},\gamma_{k+1}):\gamma_{1}=\bar{\gamma}_{0},\gamma_{K}=\bar{\gamma}_{K}\right\} .
\]
In particular, the Lloyd-type algorithm we provide for minimizing
$\text{PPC}^{K}(\Lambda)$ can be easily modified to consider this
optimization problem; one simply initializes the knots $\gamma_{0}$
and $\gamma_{K}$ appropriately and leaves them fixed at each iteration. Furthermore this discretization can be shown to converge to the continuum problem in the sense of Theorem \ref{thm:discrete-to-continuum} by an identical argument, also.

Likewise, one can also take $\gamma_0$ and $\gamma_K$ as fixed in the argument of $\text{PPC}_w^K(\Lambda_N)(\cdot)$ and run Algorithm \ref{alg:coupled-lloyd-nonlocal} with $\gamma_0$ and $\gamma_K$ fixed throughout. This is a minute modification of Algorithm \ref{alg:coupled-lloyd-nonlocal}, but since it is this modification which we use in our experiments we present it explicitly below. Note that the only changes are that: in line 3, the TSP subroutine should be understood as keeping $\gamma_0$ and $\gamma_K$ fixed while being allowed to permute the other indices of the knots; and, in line 5, one only updates the locations of the knots $\{\gamma_k\}_{k=2}^{K-1}$.

\begin{algorithm}[H]
  \SetKwComment{Comment}{/* }{ */}
  \caption{Nonlocal, Fixed-Endpoint Coupled Lloyd's Algorithm}\label{alg:coupled-lloyd-nonlocal-fixed}

  \SetKwInOut{Input}{Input}
  \Input{$\textrm{data }\{x_n\}_{n=1}^N, \textrm{parameters } \beta>0, \{h_k\}_{k=1}^K, \varepsilon>0$}

  \SetKwFunction{Initialize}{initialize\_knots}

  $\{\gamma_k\}_{k=1}^K \gets \Initialize{}$\;

  \Repeat{$\varepsilon\textrm{-convergence}$}{
   $\{\gamma_k\}_{k=1}^K \gets
    \texttt{TSP\_ordering}(\{\gamma_k\}_{k=1}^K)$
    \Comment*[r]{min-length ordering}
   $\{I_k\}_{k=1}^K \gets
    \texttt{compute\_Voronoi\_cells}(\{\gamma_k\}_{k=1}^K)$\;
       $\{\gamma_k\}_{k=1}^K \gets
    \argmin\limits_{\{\gamma'_k\}_{k=2}^{K-1} }
    \text{PPC}_{w}^{K}(\Lambda_{N}) (\gamma'_1, ..., \gamma'_K)$ \texttt{using $I_k$'s from 4}\;
  }

  \KwResult{$\{\gamma_k\}_{k=1}^K$ \Comment*[r]{The updated output knots}}

\end{algorithm}

Statistically, one can think of this modified optimization problem
as follows. Let the distribution $\Lambda$ describe a noisy observation
of a ground truth curve $\rho_{t}$ which we are trying to infer.
If an oracle has told us the locations of $\rho_{0}$ and $\rho_{1}$,
but the other temporal values of $\rho_{(\cdot)}$ are unlabelled,
then we can estimate $\rho_{t}$ up to time-reparametrization by solving
the optimization problem above with $\rho_{0}=\bar{\gamma}_{0}$ and
$\rho_{1}=\bar{\gamma}_{1}$. This \emph{semi-supervised} estimation
of $\rho_{t}$ is especially important in the application to the seriation
problem we consider in Section~\ref{sec:seriation}, as the initial and
terminal temporal labels are sufficient to provide identifiability
of the temporal ordering in the sense of Proposition \ref{prop:pseudotimes-consistency}.

\subsubsection*{Semi-supervision} More generally, subsets of $AC([0,1];X)$ of the form
\[
\{\gamma\in AC([0,1];X):\forall j=1,\ldots,J,\gamma_{t_{j}}=\bar{\gamma}_{t_{j}}\}
\]
for finitely many fixed points $\bar{\gamma}_{t_{j}}$, are closed
inside $AC([0,1];X)$. In particular, if we consider the semi-supervised
problem where an oracle specifies the ground truth $\rho_{t}$ for
finitely many time points, then the corresponding PPC-type optimization
problem is given by
\[
\min_{\gamma\in AC([0,1];X)}\left\{ \int_{X}d^{2}(x,\Gamma)d\Lambda(x)+\beta\text{Length}(\gamma):\forall j=1,\ldots J,\gamma_{t_{j}}=\bar{\gamma}_{t_{j}}\right\} .
\]
This type of semi-supervised inference problem makes sense in the
setting of the seriation problem, if temporal labels are accessible
in principle but substantially more costly than unlabelled data. Indeed many existing trajectory inference methods for single-cell omics data allow for semi-supervised specification of some of the data \cite{saelens2019comparison}.

We can also impose this semi-supervision at the discrete level: suppose for simplicity that the times $t_j$ are rational and that $K$ is sufficiently large that for each $j$, $t_j=k_j/K$ for some $1\leq k_j \leq K$. Then,
\[
\min_{\gamma_{1},\ldots,\gamma_{K}\in X}\left\{ \int_{X}d^{2}(x,\Gamma^{K})d\Lambda(x)+\beta\sum_{k=1}^{K-1}d(\gamma_{k},\gamma_{k+1}):\forall j=1,\ldots,J,\gamma_{k_j}=\bar{\gamma}_{k_j},\right\}
\]
is a discrete objective with $J$ many fixed midpoints, and its consistency with the continuum objective with fixed midpoints can be proved along identical lines to the proof of Theorem \ref{thm:discrete-to-continuum}. Likewise one can employ the nonlocal discrete objective $\text{PPC}_w^K$ with fixed midpoints in an identical fashion, and modify Algorithm \ref{alg:coupled-lloyd-nonlocal-fixed} so that the TSP subroutine in Step 2 is instead constrained to hit the intermediate fixed knots $\bar{\gamma}_{k_j}$ at the specified intermediate indices.

\section{\texorpdfstring{Additional details for Section \ref{sec:Wasserstein}}{Additional details for Section 3}}
\label{appn:extra-Wasserstein}

In this Appendix, we present Algorithm \ref{alg:coupled-lloyd} explicitly in the Wasserstein principal curves case (where $X=\mathcal{P}(V)$ and $d=W_{2}$,
and the data is delivered as empirical measure ``batches'' $\hat{\mu}_{n}$). We then elaborate on how the main optimization step in each iteration of the algorithm can be effectively performed.

In prose, Algorithm \ref{alg:coupled-lloyd} reads as follows in this case:

\begin{enumerate}
    \item Initialize the ``knots'' $\{\gamma_{k}\}_{k=1}^{K}$. (Line 1 in Algorithm \ref{alg:coupled-lloyd})
    \item Given some predetermined threshold $\varepsilon>0$, repeat steps 3-5 below in a loop, until the value of the discrete objective $\sum_{k=1}^{K}\sum_{n\in
      I_{k}}\frac{W_2^{2}(\hat{\mu}_{n},\gamma_{k})}{N}+\beta\sum_{k=1}^{K-1}W_2(\gamma_{k},\gamma_{k+1})$ drops by less than
     $\varepsilon>0$. (Lines 2 and 6)
    \item Permute the indices for the knots $\{\gamma_{k}\}_{k=1}^{K}$ to
    minimize the total length $\sum_{k=1}^{K-1}W_2(\gamma_{k},\gamma_{k+1})$. This amounts to solving the traveling salesman problem with respect to the matrix of pairwise distances between the $\gamma_k$'s. (Line 3)
    \item Given $\{\gamma_{k}\}_{k=1}^{K}$, compute the sets $I_{k}:=\{n:\gamma_{k}\text{ is the closest knot to }x_{n}\}$.
    Tiebreak lexicographically if necessary. (Line 4)
    \item Given the $I_{k}$'s from step 4, compute a new set of $\gamma_{k}$'s
    by minimizing
    $\sum_{k=1}^{K}\sum_{n\in I_{k}}\frac{W_2^{2}(\hat{\mu}_{n},\gamma_{k})}{N}+\beta\sum_{k=1}^{K-1}W_2(\gamma_{k},\gamma_{k+1}).$ (line 5)
\end{enumerate}
This is the form we will work with for the remainder of this section.

\subsection*{Computing the minimizers from Step 5 in practice}
Below, we explain how to recast the minimization problem in Step 5 above, so that it can be passed to an ordinary Wasserstein barycenter solver.

First, we rewrite the objective, $J$,
\begin{multline*}
\begin{aligned}
J &=  \sum_{k=1}^{K}\sum_{n\in I_{k}}\frac{W_{2}^{2}(\hat{\mu}_{n},\gamma_{k})}{N}+\beta\sum_{k=1}^{K-1}W_{2}(\gamma_{k+1},\gamma_{k})\\
 &= \sum_{k=2}^{K-1}\left(\sum_{n\in I_{k}}\frac{W_{2}^{2}(\hat{\mu}_{n},\gamma_{k})}{N}+\beta\frac{W_{2}(\gamma_{k+1},\gamma_{k})+W_{2}(\gamma_{k},\gamma_{k-1})}{2}\right)\\
 & \quad +\left(\sum_{n\in I_{1}}\frac{W_{2}^{2}(\hat{\mu}_{n},\gamma_{1})}{N}+\beta\frac{W_{2}(\gamma_{1},\gamma_{2})}{2}\right)+\left(\sum_{n\in I_{K}}\frac{W_{2}^{2}(\hat{\mu}_{n},\gamma_{K})}{N}+\beta\frac{W_{2}(\gamma_{K},\gamma_{K-1})}{2}\right)
\end{aligned}
\end{multline*}
This is just a re-summation of the objective into separate terms
for each $k$. Note that the initial and final terms for $k=0,K$ only have one adjacent knot. Then, observe that
\[
W_{2}(\gamma_{k},\gamma_{k-1})=\frac{W_{2}^{2}(\gamma_{k},\gamma_{k-1})}{W_{2}(\gamma_{k},\gamma_{k-1})}
\]
Now, denote $\Delta(k,k-1):=W_{2}(\gamma_{k},\gamma_{k-1})$, and
rewrite the objective as
\begin{multline*}
\begin{aligned}
J = &\sum_{k=2}^{K-1}\left(\sum_{n\in I_{k}}\frac{W_{2}^{2}(\hat{\mu}_{n},\gamma_{k})}{N}+\beta\frac{W_{2}^{2}(\gamma_{k+1},\gamma_{k})}{2\Delta(k+1,k)}+\beta\frac{W_{2}^{2}(\gamma_{k},\gamma_{k-1})}{2\Delta(k,k-1)}\right)\\
& +\left(\sum_{n\in I_{1}}\frac{W_{2}^{2}(\hat{\mu}_{n},\gamma_{1})}{N}+\beta\frac{W_{2}^{2}(\gamma_{1},\gamma_{2})}{2\Delta(1,2)}\right)+\left(\sum_{n\in I_{K}}\frac{W_{2}^{2}(\hat{\mu}_{n},\gamma_{K})}{N}+\beta\frac{W_{2}^{2}(\gamma_{K},\gamma_{K-1})}{2\Delta(K,K-1)}\right)
\end{aligned}
\end{multline*}
Then: at each iteration, we can optimize each $k$-block of the objective
separately in $\gamma_{k}$, using the values of $\Delta(k+1,k)$,
$\Delta(k,k-1)$, $\gamma_{k+1}$, and $\gamma_{k-1}$ computed at
the last iteration. Explicitly, this looks like repeatedly computing
\[
\gamma_k = \argmin\limits_{\tilde{\gamma_{k}}}\left(\sum_{n\in I_{k}}\frac{W_{2}^{2}(\hat{\mu}_{n},\tilde{\gamma_{k}})}{N}+\beta\frac{W_{2}^{2}(\gamma_{k+1},\tilde{\gamma_{k}})}{2\Delta(k+1,k)}+\beta\frac{W_{2}^{2}(\tilde{\gamma_{k}},\gamma_{k-1})}{2\Delta(k,k-1)}\right)
\]
for $1<k<K$, and similarly for $k=1,K$. The tuple of $\tilde{\gamma}_{k}$'s
returned is then stored and used as the knots for the next iteration.
In particular, the minimization problem above is \emph{precisely}
a Wasserstein barycenter problem where the weights are either $\frac{1}{N}$,
$\frac{\beta}{2\Delta(k+1,k)}$, or $\frac{\beta}{2\Delta(k,k-1)}$; this
problem can be passed directly to any number of the Wasserstein barycenter
solvers in the python OT package, for example $\mathtt{ot.lp.free\_support\_barycenter}$.

\begin{rem}
In practice, we initialize the $\gamma_k$'s as discrete measures in Step 1 above. In view of the fact that the  Wasserstein barycenter of a finitary family of discrete measures is again discrete \cite{anderes2016discrete}, it then holds automatically that the $\gamma_k$'s remain discrete measures throughout every further iteration of the algorithm.
\end{rem}

\begin{rem}
An alternative approach to recasting the objective from Step 5 above, suggested in \cite{kirov2017multiple} in the Euclidean
case, is to recast the objective as a constrained optimization problem.
In particular, one can also rewrite the objective as
\[
\sum_{k=1}^{K}\sum_{n\in I_{k}}\frac{W_{2}^{2}(\hat{\mu}_{n},\gamma_{k})}{N}+\beta\sum_{k=1}^{K-1}z_{k}\text{ with the constraints }z_{k}=W_{2}(\gamma_{k},\gamma_{k+1})\;(k=1,\ldots,K-1)
\]
and then jointly optimize over $\{\gamma_{k}\}_{k=1}^{K}$ and $\{z_{k}\}_{k=1}^{K-1}$
under this constraint using ADMM or similar. We do not pursue this
approach in our numerical experiments, however we simply indicate
to the reader that this is the main alternate approach to this type
of objective that we are aware of.

\end{rem}

\section{How to project onto a curve in Wasserstein space}
\label{appn:linearized-ot-projection-scheme}
This section concerns numerical methods to project a given probability
measure $\rho\in\mathcal{P}(V)$ onto a curve $\gamma_{t}\in AC([0,1];\mathcal{P}(V))$,
or equivalently to approximately compute
\[
\hat{\tau}(\rho)\in\argmin\limits_{t\in[0,1]}W_{2}(\rho,\gamma_{t}).
\]
To this end, we will motivate and develop an efficient projection scheme based on ideas from linearized optimal transport, which we present as Algorithm \ref{alg:curve_projection}.

In our numerics, we shall focus on the case where $\rho$ is discrete and $\gamma_{t}$ is
represented by knots $\{\gamma_{k}\}_{k=1}^{K}$, each of which is
also discrete. As a first step, one can simply compute all of the
values of $W_{2}(\rho,\gamma_{k})$ for $k=1,\ldots,K$, and pick
the lowest one; however, it is preferable to approximate $\hat{\tau}(\rho)$
to a finer resolution than this.

\subsection*{Mathematical setup from linearized OT}
Recall that in the main text, we sometimes work with the geodesic
interpolation of $\{\gamma_{k}\}_{k=1}^{K}$. Generalizing, let $T^{opt}$
be an optimal $W_{2}$ map between some distribution $\mu_{0}$ to a distribution $\mu_{1}$ and assume
$V\subseteq\mathbb{R}^{d}$; the geodesic interpolation, $\mu_t$, is then given
by
\[
\mu_{t}:=(tT^{opt}+(1-t)\text{Id})_{\#}\mu_{0}.
\]
In the case of $K$ knots $\{\gamma_{k}\}_{k=1}^{K}$, the geodesic
interpolation from $\gamma_{k}$ to $\gamma_{k+1}$ is defined using
the same construction, but with the time variable rescaled to run
from $\frac{k-1}{K-1}$ to $\frac{k}{K-1}$. Consequently, our problem
of computing $\hat{\tau}(\rho)$ can be reduced to projecting onto
a curve of the form $\mu_{t}=(tT^{opt}+(1-t)\text{Id})_{\#}\mu_{0}$,
namely evaluating $\argmin_{t\in[0,1]}W_{2}(\rho,\mu_{t})$.

Projection onto the curve $\rho$ can be approached by ``brute force'', by
first computing the OT map $T$ and then evaluating $\mu_{t}$ (and
thence $W_{2}(\rho,\mu_{t})$) for many values of $t\in[0,1]$; in
fact, this brute force approach is employed in \cite{seguy2015principal}
in the context of Wasserstein-valued PCA. In practice, this approach is computationally costly, requiring on the order of $10 K$ to $100 K$ evaluations of $W_2$ distance for sufficient accuracy, where $K$ is the number of knots. In addition, noise from entropic regularization of these distances can affect the quality of ordering. Unfortunately, minimizing
$W_{2}(\rho,\mu_{t})$ is difficult, because it is known \cite[Example 9.1.5]{ambrosio2008gradient}  that $t\mapsto W_{2}(\rho,\mu_{t})$ is nonconvex in general.

Our strategy will be to exploit a \emph{convex surrogate}
for the problem of minimizing $W_{2}(\rho,\mu_{t})$, for which the
problem of computing the minimal $t$ will have a unique, closed-form
solution. Fortunately, such a convex surrogate problem is available,
owing to the formalism of ``generalized Wasserstein geodesics''
and ``linearized optimal transport''  \cite{ambrosio2008gradient, bai2023linear, moosmuller2023linear, wang2013linear}. Let us explain
this approach.

Suppose, momentarily, that there exists an OT map $T^{0}$ from $\rho$
to $\mu_{0}$, as well as an OT map $T^{1}$ from $\rho$ to $\mu_{1}$.
The \emph{generalized $W_{2}$ geodesic from $\mu_{0}$ to $\mu_{1}$
with base point $\rho$} is then given by
\[
\tilde{\mu}_{t}:=(tT^{1}+(1-t)T^{0})_{\#}\rho.
\]
This curve also satisfies $\tilde{\mu}_{0}=\mu_{0}$ and $\tilde{\mu}_{1}=\mu_{1}$.
However, unlike with the true geodesic interpolation, it holds that
$t\mapsto W_{2}^{2}(\rho,\tilde{\mu}_{t})$ is 1-strongly convex \cite[Lemma 9.2.1]{ambrosio2008gradient}  and therefore has a unique minimizer \footnote{In fact, this
object was first introduced in \cite{ambrosio2008gradient} \emph{precisely because}
it satisfies this convexity property.}. We remark that in the case
where $W_{2}(\rho,\mu_{0})$ is small, then $\tilde{\mu_{t}}$ is
a \emph{bona fide} approximation of $\mu_{t}$ --- indeed, as $W_{2}(\rho,\mu_{0})\rightarrow0$,
then $T^{0}$ converges to the identity map, and $T^1$ converges to $T^{opt}$, and so $(tT^{1}+(1-t)T^{0})_{\#}\rho$
converges to $(tT^{opt}+(1-t)\text{Id})_{\#}\mu_0$ for each
$t\in[0,1]$.

In our context, $\rho$ is a ``data batch'' stored as a discrete
measure $\hat{\rho}_{n}$, and the distributions $\mu_{0}$, $\mu_{1}$ are the knots
$\gamma_{k}$, $\gamma_{k+1}$, respectively. We note that these knots also take the form of discrete measures. In the case where $\hat{\rho}_{n}$, $\gamma_{k}$, and $\gamma_{k+1}$
have differing numbers of atoms, OT maps from $\hat{\rho}_{n}$to
$\gamma_{k}$ or $\gamma_{k+1}$ may not exist. The workaround we
employ is that we replace $T^{0}$ and $T^{1}$ with \emph{barycentric
projection maps}, specifically \emph{entropic Brenier maps} (cf. \cite[Definition 5.4.2]{ambrosio2008gradient}  and \cite{divol2025tight, pooladian2021entropic, rigollet2025sample}). Let $\pi^{0}$
and $\pi^{1}$ be couplings (not necessarily optimal) between $\rho$
and $\mu_{0}$, and $\rho$ and $\mu_{1}$ respectively. The \emph{barycentric
projection map} $\bar{T}^{0}$ is given by
\[
\bar{T}^{0}(x):=\int yd\pi_{x}^{0}(y)
\]
where $\pi_{x}^{0}$ is the regular conditional probability measure
of $\pi^{0}$ w.r.t. $\rho$ (so that $d\pi(x,y)=d\pi_{x}(y)d\rho(x)$).
$\bar{T}^{1}(x)$ is defined identically. In practice, we will take
$\pi^{0}$ and $\pi^{1}$ to be entropy-regularized approximations
of the true OT plans from $\rho$ to $\mu_{0}$ and $\mu_{1}$. By \cite[Remark 2]{pooladian2021entropic}, it holds that in this case,
$\bar{T}^{0}$ and $\bar{T}^{1}$ are themselves gradients of convex
functions. Thus, by the converse to Brenier's theorem \cite[Theorem 1.48]{santambrogio2015optimal}, it holds that $\bar{T}^{0}$ is itself an OT map from $\rho$
to $\bar{T}_{\#}^{0}\rho$, and similarly for $\bar{T}^{1}$. Accordingly,
the curve
\[
\bar{\mu}_{t}:=(t\bar{T}^{1}+(1-t)\bar{T}^{0})_{\#}\rho
\]
is a generalized geodesic from $\bar{T}_{\#}^{0}\rho$ to $\bar{T}_{\#}^{1}\rho$,
and the function $t\mapsto W_{2}^{2}(\rho,\bar{\mu}_{t})$
is again 1-strongly convex (so with a unique minimizer).

\begin{rem}
    We emphasize that $\tilde{\mu}_t$ and $\bar{\mu}_t$ depend on the choice of $\rho$. In particular, generalized geodesics will \emph{not} produce a single, ``regularized'' interpolation of a set of knots $\{\gamma_k\}$, rather the interpolation is distinct for each $\rho$ which is being projected. This is not problematic for our application to the seriation problem, where we just need to effectively compute the projection pseudotimes.
\end{rem}

\subsection*{Approximate projection algorithm}
With this mathematical background in hand, we now give the practical
details for approximate projection onto Wasserstein principal curves.

Let
\begin{equation*}
    \hat{\rho}_{n}=\frac{1}{M}\sum_{n=1}^{N}\delta_{x_{n}^{m}}
\end{equation*}
be a data batch (stored as an empirical measure), and let\footnote{As noted in Appendix \ref{appn:extra-Wasserstein}, in general our knots $\{\gamma_{k}\}_{k=1}^{K}$
will always be discrete measures. However in principal the atoms may
not have equal weights. We neglect this in the  text above to
reduce notational burden.}
\begin{equation*}
    \gamma_{k}=\frac{1}{J}\sum_{j=1}^{J}\delta_{z_{k}^{j}} \qquad \gamma_{k+1}=\frac{1}{L}\sum_{\ell=1}^{L}\delta_{z_{k+1}^{\ell}}
\end{equation*}
be two adjacent knots in our (discrete) Wasserstein principal curve $\{\gamma_{k}\}_{k=1}^{K}$.
Let $\pi^{k}$ and $\pi^{k+1}$ be approximate optimal couplings between
$\hat{\rho}_{n}$ and $\gamma_{k}$, and $\hat{\rho}_{n}$ and $\gamma_{k+1}$,
respectively, computed using entropic regularization. Here $\pi^{k}$
can be viewed as an $M\times J$ matrix $[\pi_{mj}^{k}]$ and similarly
$\pi^{k+1}$ can be viewed as an $M\times L$ matrix $[\pi_{ml}^{k}]$.
The entropic Brenier maps $T^{k}$ and $T^{k+1}$ induced by $\pi^{k}$
and $\pi^{k+1}$ are then given by
\[
T^{k}(x_{n}^{m}):=\frac{\sum_{j=1}^{J}\pi_{m\ell}^{k+1}z_{k}^{j}}{\sum_{j=1}^{J}\pi_{m\ell}^{k+1}};\qquad T^{k+1}(x_{n}^{m}):=\frac{\sum_{j=1}^{L}\pi_{m\ell}^{k+1}z_{k+1}^{\ell}}{\sum_{\ell=1}^{L}\pi_{m\ell}^{k+1}}
\]
and so the generalized geodesic between $\bar{\gamma}_{k}:=T_{\#}^{k}\hat{\rho}_{n}$
and $\bar{\gamma}_{k+1}:=T_{\#}^{k+1}\hat{\rho}_{n}$ is given by
\[
\bar{\gamma}_{t_{k}}:=\left(t_{k}T^{k}+(1-t_{k})T^{k+1}\right)_{\#}\hat{\rho}_{n}\quad t_{k}\in[0,1].
\]
(The entire piecewise-generalized geodesic connecting the $\bar{\gamma}_{k}$'s
is then defined by concatenating the $\bar{\gamma}_{t_{k}}$'s and
rescaling time to be constant-speed from 0 to 1 along the whole curve.)

We already mentioned above that $t_{k}\mapsto W_{2}^{2}(\hat{\rho}_{n},\bar{\gamma}_{t_{k}})$
is 1-strongly convex; here, we derive minimal time $t_{k}^{*}$, in
closed form in terms of $\hat{\rho}_{n}$, $T^{k}$, and $T^{k+1}$.
\begin{prop}
\label{prop:approx_proj}
The minimizer $t_{k}^{*}$ of $t_{k}\mapsto W_{2}^{2}(\hat{\rho}_{n},\bar{\gamma}_{t_{k}})$
is either $0$, $1$, or
\[
\frac{\sum_{m=1}^{M}\left(x_{n}^{m}-T^{k}(x_{n}^{m})\right)\cdot\left(T^{k+1}(x_{n}^{m})-T^{k}(x_{n}^{m})\right)}{\sum_{m=1}^{M}\Vert T^{k+1}(x_{n}^{m})-T^{k}(x_{n}^{m})\Vert^{2}}.
\]
\end{prop}

\begin{proof}
Clearly $t_{k}^{*}$ is either at the endpoints of the interval $[0,1]$,
or in the interior; so we focus on this last case. By \cite[Lemma 9.2.1]{ambrosio2008gradient}, it holds that the map $t_{k}T^{k}+(1-t_{k})T^{k+1}$ is
actually optimal from $\hat{\rho}_{n}$ to its pushforward, and so
for all $t_{k}\in[0,1]$, we have that
\[
W_{2}^{2}(\hat{\rho}_{n},\bar{\gamma}_{t_{k}})=\frac{1}{M}\sum_{m=1}^{M}\left\Vert \left(t_{k}T^{k}+(1-t_{k})T^{k+1}\right)(x_{n}^{m})-x_{n}^{m}\right\Vert ^{2}.
\]
Differentiating the right hand side with respect to $t_{k}$ and setting
the derivative equal to zero establishes the claim. (This critical point is automatically unique, and a minimizer, in view of the fact, also justified by \cite[Lemma 9.2.1]{ambrosio2008gradient}, that $t\mapsto W_2^2(\hat{\rho}_n,\bar{\gamma}_{t_k})$ is strongly convex.)
\end{proof}

Using the results of Proposition \ref{prop:approx_proj} we can now create our algorithm by simply finding the best projection onto each segment and then choosing the segment with the overall lowest projection distance.
In practice, we find this algorithm to be orders of magnitude faster than the naive approach, as well as more stable numerically.

\begin{algorithm}[H]

    \SetKwComment{Comment}{/* }{*/}
    \caption{Efficient Projection of a Data Batch to the Curve}\label{alg:curve_projection}

    \SetKwInOut{Input}{Input}
    \Input{$\textrm{data batch } \hat{\rho}_n$, discrete curve $\{\gamma_k\}_{k=1}^K$}

    \SetKwFunction{tmap}{entropic\_transport\_plan}
    \SetKwFunction{bary}{entropic\_Brenier\_map}
    \SetKwFunction{curveDist}{arclength\_along\_curve}

    \For{$k \in 1 \ldots K$}{
        $\pi^k \gets \tmap(\hat{\rho}_n, \gamma_k)$\;
    }
    \BlankLine

    $\textrm{best\_dist} \gets \infty$\;
    $\lambda_\textrm{best} \gets 0$\;
    \BlankLine

    \For{$k \in 1 \ldots K-1$}{
        $T^k \gets \bary(\hat{\rho}_n, \pi^k)$\;
        $T^{k+1} \gets \bary(\hat{\rho}_n, \pi^{k+1})$\;
        $t \gets \textrm{clip}\left(\frac{\sum_{m=1}^{M}\left(x_{n}^{m}-T^{k}(x_{n}^{m})\right)\cdot\left(T^{k+1}(x_{n}^{m})-T^{k}(x_{n}^{m})\right)}{\sum_{m=1}^{M}\Vert T^{k+1}(x_{n}^{m})-T^{k}(x_{n}^{m})\Vert^{2}},0,1\right)$\;
        $\textrm{dist}^2 \gets \frac{1}{M}\sum_{m=1}^{M}\left\Vert \left(tT^{k}+(1-t)T^{k+1}\right)(x_{n}^{m})-x_{n}^{m}\right\Vert ^{2}$\;
        $\tau \gets \curveDist(k, t)$\;
        \BlankLine

        \If{$\textrm{dist} < \textrm{best\_dist}$}{
            $\textrm{best\_dist} \gets \textrm{dist}$\;
            $\tau_\textrm{best} \gets \tau$\;
        }
    }

    \KwResult{$\tau_\textrm{best}$ \Comment*[r]{The approx. best projection pseudotime}}
\end{algorithm}

\section{Additional details for experiments}
\label{appn:extra-experiments}
This Appendix provides additional details for the experiments we describe in Section \ref{sec:experiments} of the main text. We also give further algorithmic details for aspects such as  initialization of a principal curve, and heuristic parameter selection.

\subsection*{Description of prior seriation methods}

As mentioned in the main text, we compare our approach to seriation based on principal curves with two existing seriation methods: seriation based on the \emph{Traveling Salesman Problem} (TSP) \cite{LAPORTE1978259_tsp_seriation}, and spectral seriation \cite{atkins_spectral_1998, spectral_seriation_fogel_alexandre, natik2021consistencyspectralseriation}. Both approaches shall take as input the matrix $W=[W_2(\hat\rho_{t_i}, \hat\rho_{t_j})]_{i,j}$ of pairwise $W_2$ distances between empirical measures $\hat\rho_{t_i}$.

For TSP, we regard $W$ as a matrix of edge weights in a complete graph of time points. The TSP approach to seriation is to visit all nodes exactly once and return to the starting node
on a distance-minimizing path. The seriation ordering is given by this minimizing path. Using a dummy node, we can enforce the fixed start and endpoint regime by making the lack of an edge between the known start and end prohibitively expensive. The Traveling Salesman Problem can be solved exactly for up to thousands of nodes using solvers such as Concorde \cite{applegate2006traveling_concorde}.

For spectral seriation, we create a similarity matrix $A$ using a Gaussian kernel with bandwidth $\sigma$:
\begin{equation*}
    A = \exp(-W^2/\sigma^2).
\end{equation*}
We then form the normalized Laplacian $L=D^{\frac12}AD^{-\frac12}$, where
$$
  D =
  \begin{bmatrix}
    \sum_{j=1}^N A_{1, j} & & \\
    & \ddots & \\
    & & \sum_{j=1}^N A_{N, j}
  \end{bmatrix}.
$$
The Fiedler eigenvector of $L$ then imposes a seriation \cite{atkins_spectral_1998}. Solvers exist which allow us to solve eigenvector problems for large matrices \cite{bientinesi_parallel_2005, marek_elpa_2014}, allowing us also to apply spectral seriation alongside principal curves.

\subsection*{Initialization for Principal Curves}
We use a warm start approach to initialize the knots of the principal curves solver. We first initialize the knots randomly from the data points such that each $\gamma'_i = x_j$ for some $j \in 1 \ldots N$ uniformly random for each $i$. Then, we solve $\textrm{PPC}^K_w$ using Algorithm \ref{alg:coupled-lloyd-nonlocal-fixed} to get the curve $\gamma^{\text{init}}$. Using $\gamma^{\text{init}}$, we can re-initialize our knots to be equally spaced points along the curve such that $\gamma_i = \gamma^{\text{init}}(\frac{i-1}{K-1})$. In practice, these can be the closest data points to each evaluation. This initialization is what we refer to as a warm start. Our experience suggests that a warm start provides a moderate performance advantage over a purely random initialization.

Additionally, we find that the mean squared distance of data to the curve offers a reasonable proxy for seriation performance. By re-initializing multiple times and selecting the best curve by this metric, we can further improve performance. Generally, for performance comparisons, we use $25$ re-initializations, selecting the best one by squared distance to the curve. For parameter fitting, we average over all re-initializations as we value smoothness more in these applications (see below).

\subsection*{Mouse Dataset Pre-processing}
\label{appn:mouse_preprocessing}
To effectively test on Qiu et. al.'s mouse dataset \cite{mouse_time_series} we must downsample it and apply dimensionality reduction. The original dataset has approximately $12.4$ million cells, making it prohibitively time consuming to test all methods at the full scale. Additionally, we expect to apply principal curves in a high-timepoint and moderate cell counts context, which is better represented with a downsampled dataset. To do this, we first downsample the full data to $100{,}000$ cells (barcodes are available in our public code repository for reproducibility). Using scanpy \cite{scanpy}, we then run PCA with $250$ components, and compute $50$ nearest neighbours. Finally, we apply diffusion mapping with $25$ components as our dimensionality reduction method of choice. To create the datasets used in Figure \ref{fig:mouse_comparison_sweep}, we make $20$ random downsamplings at each shown cell level, noting that the number of cells per time point is kept uniform.

\subsection{Parameter selection for principal curves}
\label{appn:param_selection}
In this section, we present a heuristic parameter selection method for the principle curves parameters $K$ (number of knots), $h$ (kernel bandwidth) and $\beta$ (length penalty), which we summarize as Algorithm \ref{alg:heuristic_params}. This algorithm assumes that we use the version of the principal curves objective described in Appendix \ref{appn:practical_model} with a nonlocal kernel and fixed endpoints (used for the experiments described in Section \ref{sec:experiments}). The objective $\textrm{PPC}^K_w$ depends on the aformentioned parameters $K$, $h$, $\beta$, which must be specified by the user.  For the number of knots, we suggest setting $K = \sqrt{N}$. For $h$ and $\beta$, we suggest a heuristic method based on optimizing mean squared distances to the principal curve using a knee-finding algorithm. This heuristic method was used for parameter selection for the test datasets demonstrated in the text.

We first need to select a reasonable value for the number of knots, $K$. Here we have prescribed taking $K \sim \sqrt{N}$, which we conjecture to be  the optimal rate given sufficient samples. This parameter scaling is justified by the following informal reasoning.

Ultimately, we need our discrete principal curves to converge to
continuum principal curves. On the one hand, this requires that
$K\rightarrow \infty$, so that we can actually resolve a generic
principal curve in the limit. On the other hand, we would like each
knot $\gamma_k^N$ in a discrete principal curve to converge to the
``population knot'' $\gamma_k$ associated to the Voronoi region $I_k$, in the limit where $N\rightarrow \infty$. In order for this to hold for a typical knot, this means that, \emph{on average}, the number of data batches $\hat{\rho}_n$ inside a typical $I_k$ must also go to infinity; for this to hold, it is also necessary that $N/K \rightarrow \infty$.

Balancing these two desiderata suggests that we should choose $K$ as a function of $N$ in such a way that $\max\{{K,N/K}\}$  goes to infinity at the fastest rate possible, hence the choice that $K \sim \sqrt{N}$.

To do heuristic parameter selection of $h$ and $\beta$, we make a key empirical observation that when the optimal solution is non-degenerate, the mean squared Wasserstein distance of data points to the curve, that is $\sum_{k=1}^{K}\sum_{n\in I_{k}}W_2^2(\hat{\rho}_{n},\gamma)$, is well-correlated to seriation performance. In particular, if we plot these squared distances with only one of $h$ or $\beta$ varying on a log-scale, we tend to see a convex, increasing curve with an exponential form (see Figure \ref{fig:fitting_iter}). This curve suggests that as parameter values decrease we eventually transition into a zone of overfitting where performance returns are limited. We wish to realize the majority of possible performance gains without overfitting to an arbitrary level. The zone where we start to transition between performance and overfitting can be approximated as the knee point of these respective curves. The Kneedle algorithm \cite{kneedle} provides a reproducible approach to finding these knees. To capture the majority of performance gains  that still occur after the knee, we suggest choosing the parameter value where the slope is $5\%$ of the slope at the knee. Combining this with our estimate of $K$, we give Algorithm \ref{alg:heuristic_params} as a reproducible approach to choose parameters.

\begin{algorithm}[H]
    \SetKwComment{Comment}{/* }{*/}
    \caption{Heuristic Parameter Selection}\label{alg:heuristic_params}

    \SetKwInOut{Input}{Input}
    \Input{$\textrm{data }\{\hat{\rho}_n\}_{n=1}^N$}

    \SetKwFunction{guessh}{guess\_h}
    \SetKwFunction{squareDists}{square\_dists}
    \SetKwFunction{knee}{kneedle}
    \SetKwFunction{param}{param\_where}
    \SetKwFunction{fitExp}{fit\_exp}
    \SetKwFunction{fitParam}{fit\_parameter}

    \SetKwProg{Fn}{Function}{:}{}
    \Fn{\fitParam{$p$, $p_{\texttt{search}}$, $K$, $h$, $\beta$}}{
        $\texttt{exp\_curve} \gets \fitExp(\squareDists(K, h, \beta, p_{\texttt{search}}))$\;
        \Comment{Interpolating curve fitting performed using log-space params}
        \BlankLine

        $\texttt{knee} \gets \knee(\texttt{exp\_curve})$\;
        \Return $\param\!\left(
            \frac{d\,\texttt{exp\_curve}}{d \log p}
            =
            0.05
            \frac{d\,\texttt{exp\_curve}(\texttt{knee})}{d \log p}
        \right)$\;
    }

    $K \gets \lceil \sqrt{N} \rceil$\;
    $h \gets \guessh()$ \Comment*[r]{Reasonable initial value, e.g. $h=0.05$}
    \BlankLine

    \Repeat{Until parameter convergence}{
        $\beta \gets \fitParam(\beta, \beta\texttt{ search space}, K, h, \beta)$\;
        $h \gets \fitParam(h, h\texttt{ search space}, K, h, \beta)$\;
    }

    \KwResult{$K, h, \beta$ \Comment*[r]{Inferred parameter values}}
\end{algorithm}

\begin{figure}[H]
    \centering
    \includegraphics[width=.95\linewidth]{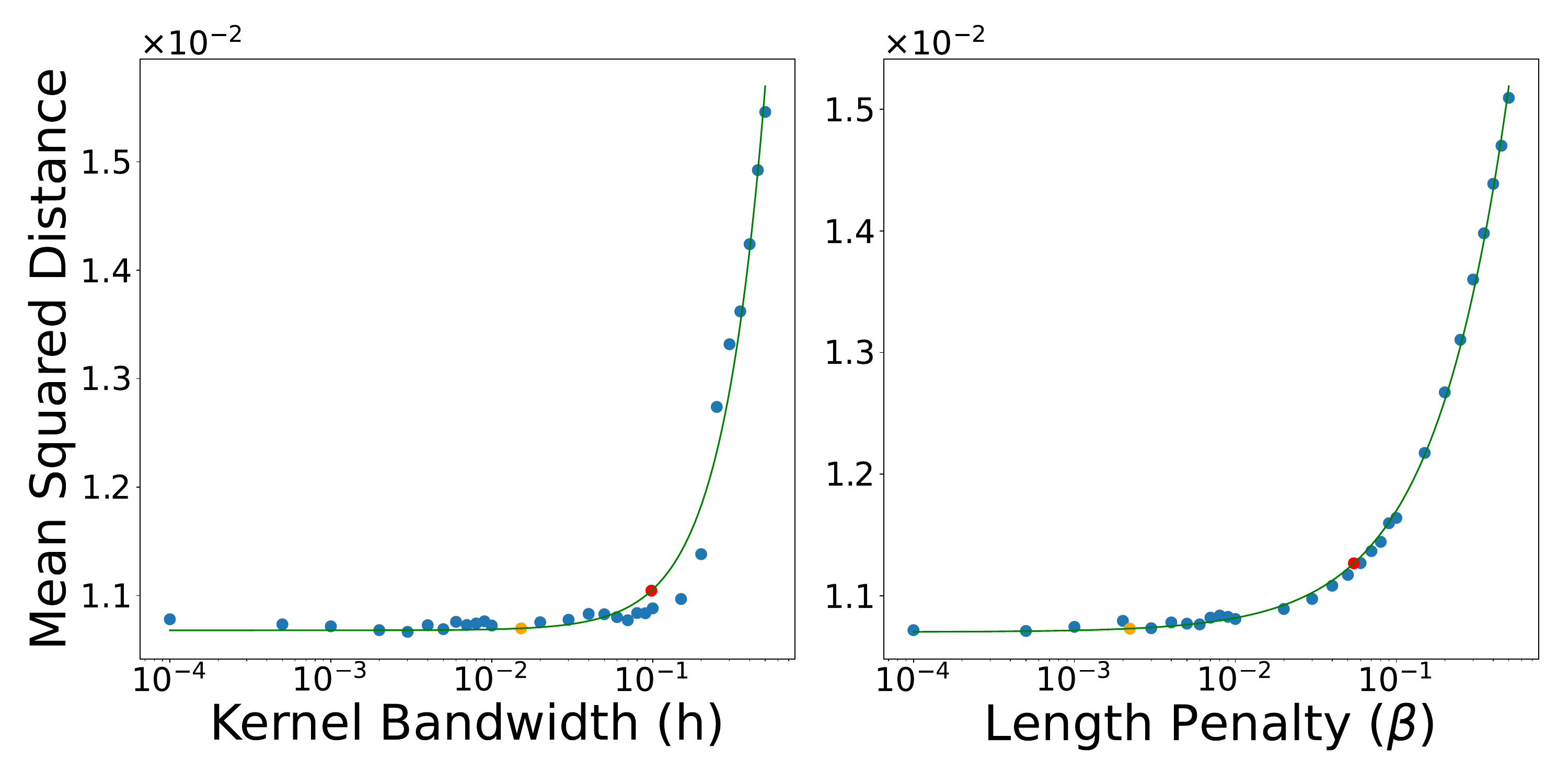}
    \caption{A plot showing squared distances of dataset points to principal curve knots for the mouse data. Each plot shows one parameter varying while the other is fixed during the parameter inference procedure. $h$ varies in the left plot, and $\beta$ varies in the right plot. Blue points are an average of squared distances over $7$ data resamplings and $25$ re-initializations. The green curve is a curve of best fit of the form $ae^{bx} +c$ where $x$ is $\log$ parameter values. The red point represents the knee point of the curve, and the yellow dot represents the predicted best parameter value according to Algorithm \ref{alg:heuristic_params}.}
    \label{fig:fitting_iter}
\end{figure}

For the mouse dataset, Algorithm \ref{alg:heuristic_params} gives parameter values of $K = 7$, $h = 0.015223$, and $\beta = 0.0022096$ after $4$ iterations of the algorithm when fit at the $10{,}000$ cells level. We use these values for all principle curves tests on the mouse dataset.

\subsection*{Parameter selection for spectral seriation}
Spectral seriation depends on a kernel bandwidth parameter, $\sigma$, which must be specified by the user. For the mouse dataset, a sweep was performed using the ground truth seriation to determine the optimal $\sigma$. This sweep was repeated for each cell count level. At each level we tested on $100$ equally spaced values of $\sigma \in [0, 3]$, repeating on each of the $20$ downsamplings. We found two performant values for $\sigma$ depending on the total number of cells, with $\sigma = 0.15152$ being used for $5000$ cells or more and $\sigma = 1$ otherwise. These parameter sweeps can be seen in Figure \ref{fig:mouse_sigma_sweeps}.
\begin{figure}[H]
    \centering
    \includegraphics[width=\linewidth]{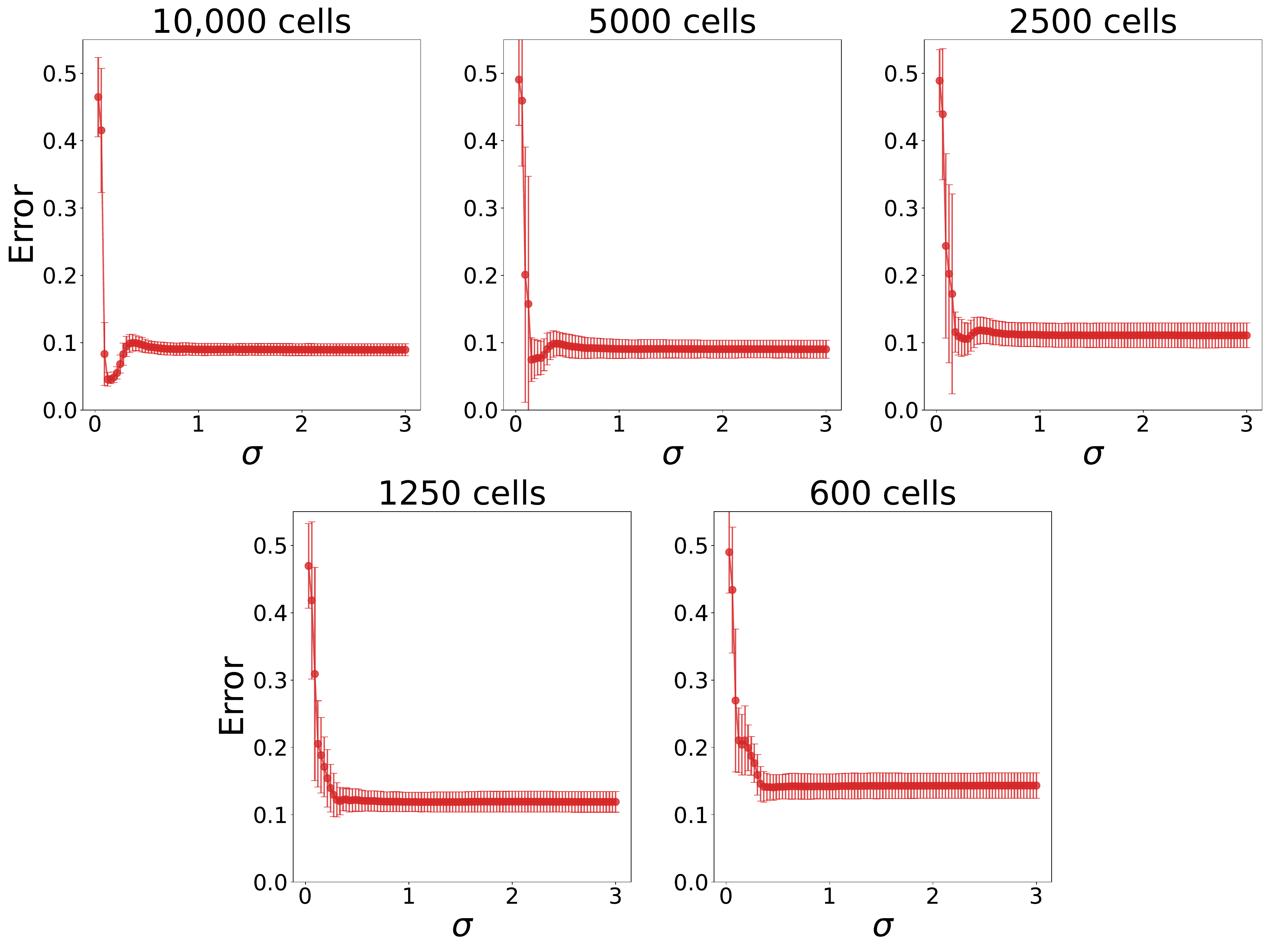}
    \caption{A set of parameter sweeps of spectral seriation performance on the mouse dataset as $\sigma$ varies. The different plots show performance on datasets as the total number of cells changes. Points are taken as mean error over $20$ resamplings. Error bars indicate standard deviation.}
    \label{fig:mouse_sigma_sweeps}
\end{figure}

\subsection*{Entropy regularization parameters for Wasserstein distances}
All tested methods implicitly rely on entropy-regularized optimal transport. Thus, we also need to choose an entropy regularization parameter. We use the entropy regularization parameter $\varepsilon^{\text{reg}} = 0.05$ in all cases except for the final projection step for principal curves. There, we use $\varepsilon^{\text{reg}} = 0.02$. We note that in our testing, the competing methods did not see significant benefits from using the lower $\varepsilon^{\text{reg}}$.

\subsection*{Empirical validations of data consistency}
 In the main text, we developed consistency results for the principle curves method which would be valuable to verify empirically. (In the main text, our experiments only evaluate principal curves as a tool for seriation.) Recall that our data takes the form of the doubly empirical measure $\hat{\Lambda}_{N, M} = \frac{1}{N} \sum_{n=1}^N \delta_{\hat{\rho}_{n}}$, where $N$ is the number of timepoints and $M$ is the number of observations per timepoint. To inform data collection, we want to see that as expected, if $N, M \to \infty$ that solutions $\hat{\gamma}$ to $\text{PPC}^K_w(\hat{\Lambda}_{N, M})$ converge towards the ground truth curve $\rho$. We can check consistency via the mean squared distance of data representing the true curve to a fitted curve. For a representative dataset of held-out data, this error should look like $\mathbb{E}_{x\sim\textrm{data}} [W_2^2(x, \hat{\gamma})]$. As we approach the true curve, this value should approach some minimum.

We created two sets of downsampled mouse datasets to test data consistency: one downsampling the number of times $N$, and the other downsampling the cells per timepoint, $M$. For the time-downsampled datasets, we randomly downsample the number of timepoints using the mouse datasets with $10{,}000$ total cells. Only the first and last timepoint are kept fixed. The number of cells per timepoint is fixed at $232$ (the same level as $10{,}000$ total cells in Figure \ref{fig:mouse_comparison_sweep}). This was done for $N \in \{7, 10, 15, 20, 25, 30, 35, 40, 43\}$ with $25$ resampling repeats. For the cells per timepoint downsampling, we kept the number of timepoints fixed at $43$. We downsampled the total cells to $\{600, 1250, 2500, 5000, 7500, 10{,}000\}$ starting with the mouse datasets with $10{,}000$ total cells. This range corresponds to $M \in \{13, 29, 58, 116, 174, 232\}$. Similar to before, we used $25$ data resampling repeats.

To evaluate performance, $\text{PPC}^K_w$ was solved on the downsampled datasets with $K=7$ knots and $h, \beta$ as in Appendix \ref{appn:param_selection}. We then projected non-downsampled held-out datasets with $20{,}000$ total cells onto the fitted curves. The held-out datasets do not share any cells with their paired downsampled datasets. In doing this, we assume the non-downsampled datasets approximate the true curve well. We expect a lower projection distance as $N, M$ increase and we are able to fit a better curve. The results of the procedure can be seen in Figure \ref{fig:N_consistency}. The experiment clearly shows error decreasing with increasing $N$, and the same with increasing $M$. This result provides an important piece of empirical evidence for data consistency, showing that a more detailed dataset will give better performance for principal curves. Both plots appear to asymptote at a small but finite mean-squared error; note that according to our consistency theory, we do not expect the mean squared error to decay all the way down to zero unless we send $\beta\rightarrow 0$ as well, which we have not done here.

\begin{figure}[H]
    \centering
    \includegraphics[width=\linewidth]{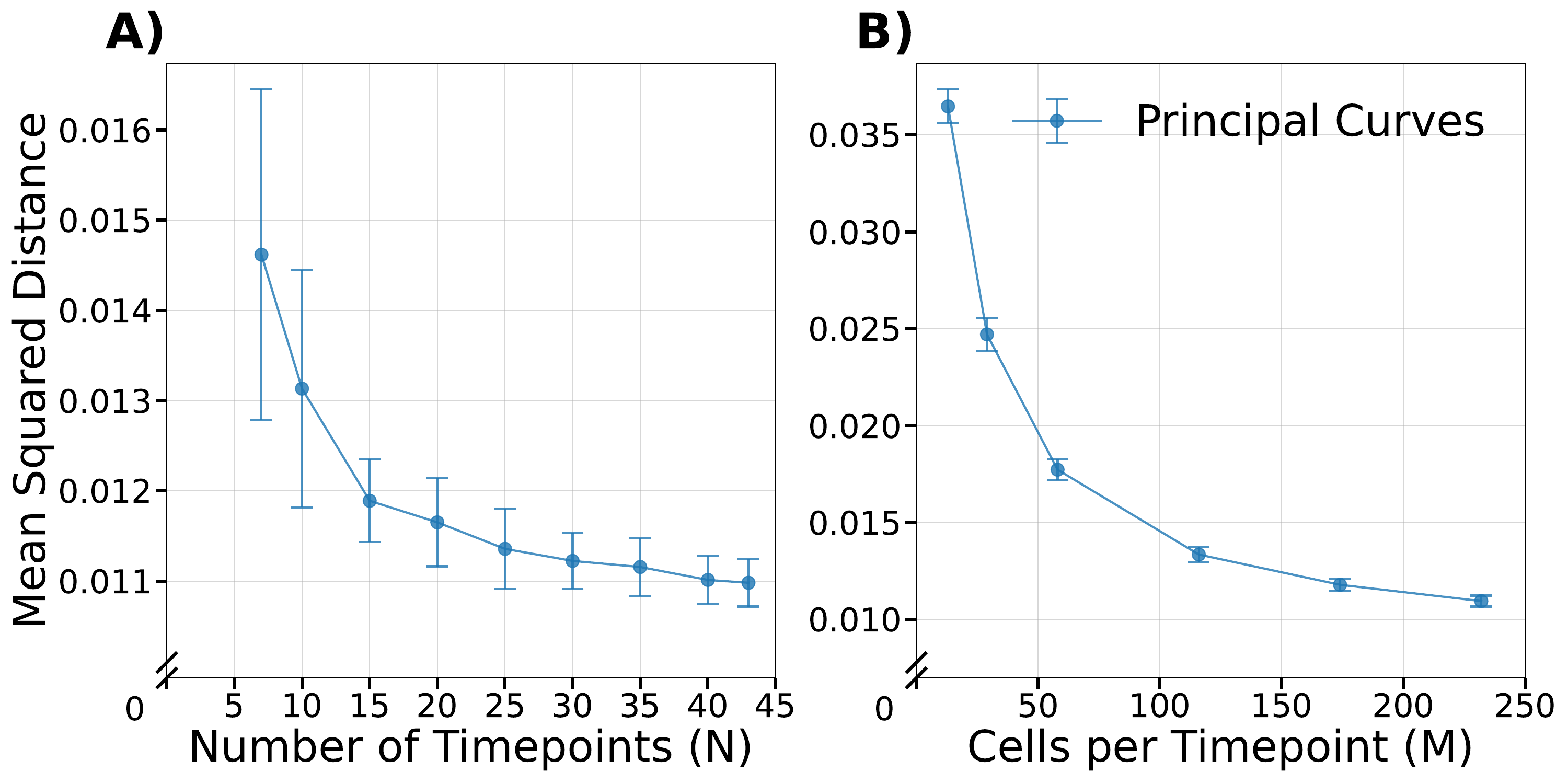}
    \caption{A figure showing mean projection distances of held out mouse datasets onto a curve fit on downsampled datasets. In panel A we see the number of timepoints ($N$) vary, while cells per timepoint ($M$) remains constant. In panel B we see the cells per timepoint vary while $N$ remains constant at $43$. Error is taken by projecting held-out mouse datasets with $20{,}000$ total cells and $43$ timepoints onto a curve fit on downsampled data. Points are taken as a mean over $25$ data resamplings. Error bars indicate standard deviation (bootstrapped confidence intervals).}
    \label{fig:N_consistency}
\end{figure}

\section{Experiments on simulated datasets}
\label{appn:simulated_experiments}
In this section we give the details of a set of experiments on simulated data. These experiments aim to highlight the data modes where the principal curves method is advantageous to use. We first present the design of two datasets: one with a simple branching structure and one which adds a section of rapid development and complexity to the branching. After discussing parameter selection, we evaluate all methods on these two datasets, comparing and discussing the differing performance results.

\subsection*{Test Dataset 1: Simple Branching}
Our first test dataset mimics a cellular differentiation event with a branching Gaussian mixture, obtained from a simple 2-D branching curve convolved with a Gaussian. It intends to act as an example of the simplest possible curve modeling a biological event. The curve $\rho_t$ is obtained by convolving a curve of discrete measures $\mu_t$ with a Gaussian
$$\rho_t = \mathcal{N}(0,\sigma^2) * \mu_t,$$
where $\sigma=0.1$ and $\mu_t$ is defined as follows:
\begin{equation*}
    \mu_t = \begin{cases}
         \delta_{(0, 1 -t)} &  0 \leq t \leq 1\\
        \frac{1}{2} \delta_{\frac{t-1}{\sqrt{2}}(-1, -1)}+\frac 1 2 \delta_{\frac{t-1}{\sqrt{2}}(1, -1)}
        & 1 \le t \leq 1 + \sqrt{2}
    \end{cases}
\end{equation*}

To produce the dataset, we uniformly partition $[0, 1 + \sqrt{2}]$ into times $t_i$; then for each $t_i$, we then draw i.i.d. samples from $\rho_{t_i}$. The resulting dataset can be seen in Figure \ref{fig:simulated_curve_fit} panel A.

\subsection*{Test Dataset 2: Branching with Rapid Developmental Change}
For our second test dataset, we present a more complex branching dataset, with an area of rapid direction change that is poorly sampled. This curve $\rho_t$ is again obtained by convolving a curve of discrete measures $\mu_t$ with a Gaussian of variance $\sigma^2 = 0.01$
$$\rho_t = \mathcal{N}(0,\sigma^2) * \mu_t,$$
where $\sigma=0.1$, but here $\mu_t$ is defined as follows:
\begin{equation*}
    \mu_t = \begin{cases}
        \delta_{(0, 1-t)} & 0 \leq t \leq 1 \\
        \delta_{\frac{t-1}{0.1}(1.5, 0)} & 1 \le t \leq 1.1 \\
        \frac12 \delta_{(t-1.1) \big(-1, 1 \big) + (1.5, 0)} + \frac12 \delta_{(t-1.1) \big (1, 1 \big ) + (1.5, 0)} & 1.1 \le t \leq 2.1. \\
    \end{cases}
\end{equation*}
To produce the dataset, again uniformly partition $[0, 2.1]$ into times $t_i$; then for each $t_i$, we then draw i.i.d. samples from $\rho_{t_i}$. The resulting dataset can be seen in Figure \ref{fig:simulated_curve_fit} panel D.

In the biological context, this dataset represents a period of rapid developmental change followed by a branching event. In time series experiments, time points may not always be captured at a fine enough time resolution to characterize such changes. This leads to segments of the developmental curve having low numbers of observed time points and thus breaks in continuity. Additionally, the change in geometry in the dataset adds a level of complexity that brings it closer to real data.

\subsection*{Parameter Selection for Simulated Data}
To select parameters for our simulated datasets, we generally follow the heuristic approach given in Appendix \ref{appn:extra-experiments} with some notable differences to account for the nature of the simulated data. These differences stem from the relatively simple geometry of these datasets breaking assumptions that would work on real data.

The first difference stems from the assumption that the data is sufficiently complex such that each datapoint contributes significantly to the geometry of the curve. This assumption is reasonable for most real datasets where timepoints are still relatively sparsely spaced in developmental time. However, in our simulated datasets, the total number of timepoints, $N$, is allowed to be large compared to the curvature of the underlying dataset. At a certain scale, noise will dominate, with additional timepoints adding less information. To combat this, we offer a heuristic for an ``effective'' number of timepoints, $N_\text{eff}$. To start, we note that these datasets have some euclidean length $L$ to their underlying curves (following a single branch) and a sampling variance $\sigma^2$. We then propose using $N_\text{eff} = \frac{L}{\sigma}$, which aims to roughly divide the curve into segments on the same length scale as the noise. For Test Dataset 1 this gives $N_\text{eff} = 24.4$, resulting in $K = 5$. For Test Dataset 2, this give $N_\text{eff} = 39.4$ and consequently $K = 7$.

Based on the heuristic approach, we can infer the following parameters: for Simulated Dataset 1 we get $h = 0.046616$ and $\beta = 0.0052225$, and for Simulated Dataset 2 we get $h = 0.010468$ and $\beta = 0.0012075$. However, we note that on inspection using the ground truth, Simulated Dataset 1 trends towards a degenerate solution of only three knots and high $h$ and $\beta$. In this situation, our assumption about square distances to the knots does not hold, as the dataset does not have sufficient geometric complexity. That is, any solution which encodes a straight line in the $y$ axis will provide a strong seriation regardless of whether it is a good fit to the data.

For spectral seriation, we ran one sweep per dataset at $N=250$ with $10$ repeats re-sampling the dataset. The sweeps cover a range of $\sigma \in [0, 3]$ with $100$ equally spaced values. The optimal value $\sigma$ was found to be $1$ for Test Dataset 1 and $0.33333$ for Test Dataset 2. Summary plots of these sweeps can be seen in Figure \ref{fig:spectral_sweep_summary}.

\begin{figure}[t]
    \centering
    \includegraphics[width=.7\linewidth]{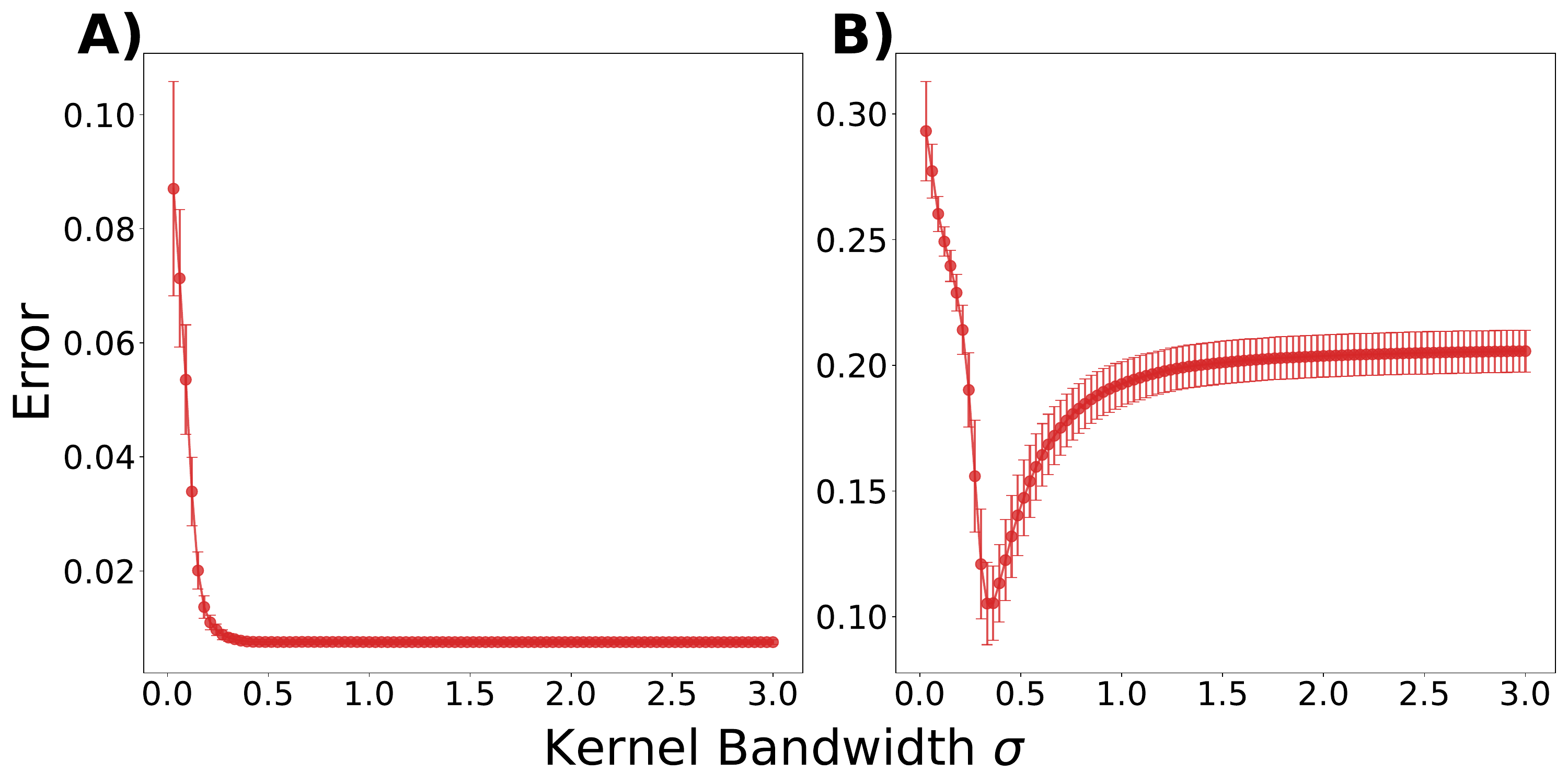}
    \caption{A sweep of kernel bandwidths for spectral seriation for Test Dataset 1 (A) and the Test Dataset 2 (B) at $N=250$. Values are taken as a mean of Normalized Kendall Tau Error over $10$ resamplings of the dataset. Error bars indicate standard deviation.}
    \label{fig:spectral_sweep_summary}
\end{figure}

\subsection*{Fixed Cell Budget Experiments on Simulated Data}
 In Section \ref{sec:experiments}, we explored the question of how the number of cells per timepoint affects seriation performance. There, we varied the number of cells per timepoint by downsampling cells. However, an alternate (indeed more realistic) sampling scheme involves keeping the total number of cells fixed and varying the number of timepoints. We call this the fixed budget sequencing problem. In this section, we will explore this fixed budget scenario on two simulated datasets: one being a simple branching dataset, and the other introducing data lossiness and complexity. We compare principal curves to the competing methods, showing that it is the most performant in the second scenario, but still competitive in the simpler dataset.
\begin{figure}[H]
    \centering
    \includegraphics[width=\linewidth]{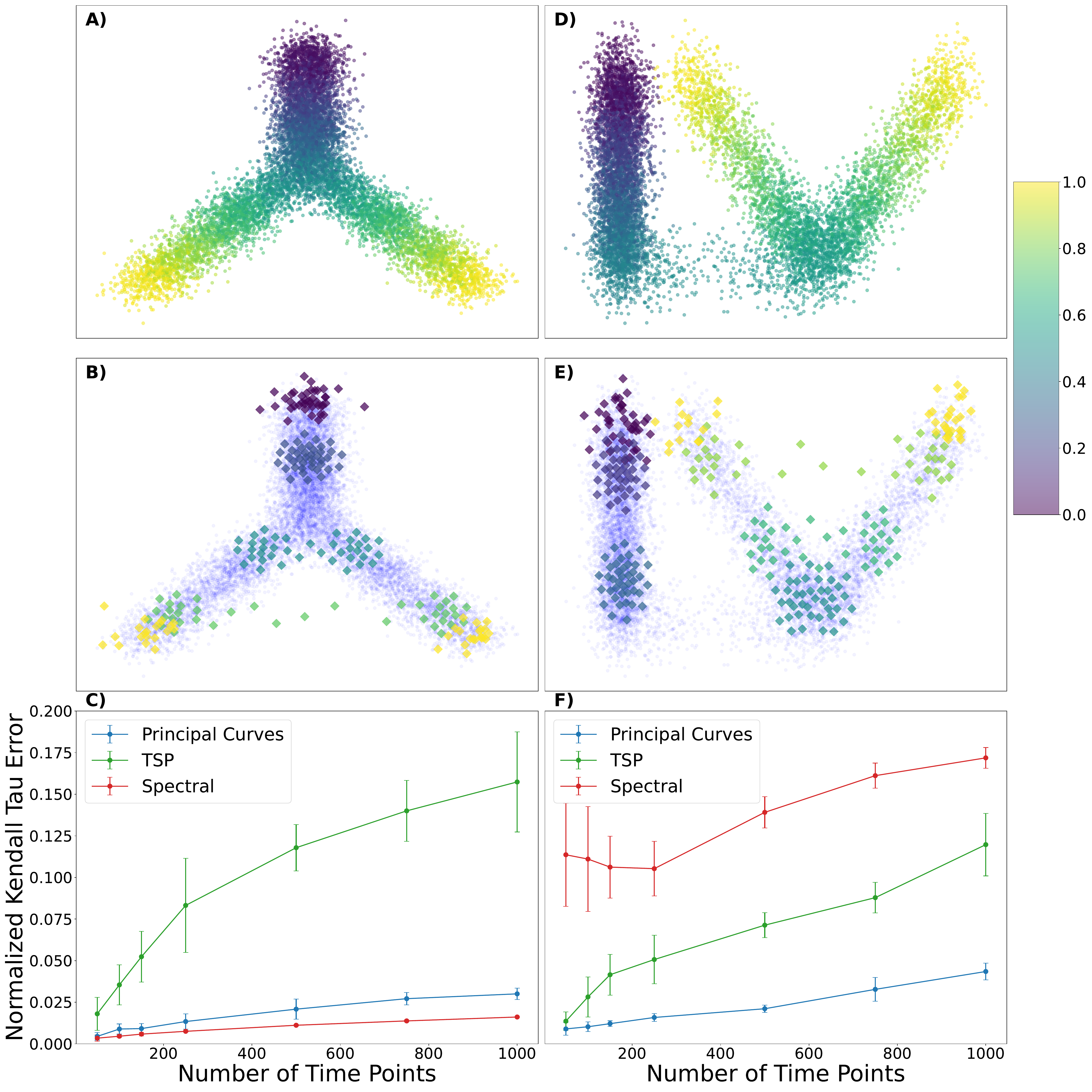}
    \caption{Plots of two simulated datasets alongside examples of a principal curve fitted onto them. A seriation performance comparison against competing seriation methods is also included. A simple dataset of probability measures with $250$ time points that undergoes a single branching event is show in panel (A). A more complex version of this dataset, where there is a rapid change in direction (mimicking a fast developmental event) is shown in panel (D). These are colored by the order of timepoints in $[0, 1]$ (see color bar). Panel (B) fits a principal curve onto the simple dataset using kernel bandwidth $h=0.046616$, length penalty $\beta=0.0052225$, and $K=5$ knots. The underlying dataset is colored in blue, with the principal curve knots overlaid (diamonds) are colored by their ordering. Panel (E) shows this for the more complex dataset, where the curve uses $h=0.010468$, $\beta=0.0012075$, and $K=7$. Panel (C) compares seriation performance on the simple dataset against other methods as the number of timepoints ($N$) varies in $[50, 1000]$ while the total number of cell is on a fixed budget of $10{,}000$. Error is measured by normalized Kendall Tau score. Points shown are the mean of $10$ data resamplings. Error bars indicate standard deviation. The parameters used for principal curves are the same as indicated previously, while $\sigma = 1$ is used for spectral seriation. Panel (F) shows the same comparison for the more complex dataset. All data parameters remain the same. Principal curves uses the parameters indicated for the more complex dataset. $\sigma = 0.33333$ is used for spectral seriation.}
    \label{fig:simulated_curve_fit}
\end{figure}

The first dataset we present is the simple branching curve (Test Dataset 1) described earlier in this appendix taking $M^{\text{total}} = 10{,}000$ atoms. We selected parameters for the principal curves method according to as described earlier in Appendix \ref{appn:simulated_experiments}. We then fit an example principal curve at $N = 250$, seeing a qualitatively good fit. See Figure \ref{fig:simulated_curve_fit} panel A for underlying data as compared to the fitted curve in panel B. Next, we generated the dataset with the number of timepoints $N$ ranging in $[50, 1000]$. Comparing principal curves to the competing methods, we see that while overall error is low across all regimes, we are marginally less performant than spectral seriation (see Figure \ref{fig:simulated_curve_fit} panel C).

Next, we look at a more challenging and realistic scenario: the branching curve with an area of rapid change. This area creates a directional change in the data, and also creates a lossy gap in data. With this dataset we aim to model the real-world effect of rapid developmental change in an embryonic time series. A test dataset drawn from this model can be seen in Figure \ref{fig:simulated_curve_fit} panel D, with a principal curve fitted to this dataset overlaid in panel E; as with Test Dataset 1, a parameter selection was first done using training data from the same model. Note that the estimated principal curve $\hat\gamma$ does not put any knots in the low-density patch; in this region $\hat\gamma$ is therefore a piecewise geodesic. In particular, the estimated principal curve is still able to capture the global structure of the dataset, despite the patch of missing/low resolution data. This ability is reflected in the performance against competing methods for this dataset (see Figure \ref{fig:simulated_curve_fit} panel F).

For this dataset, the performance of spectral seriation is particularly poor. Indeed we cannot expect spectral methods to be robust to a loss of data which results in changes to the large-scale geometry of the dataset, as the spectrum of the graph or manifold Laplacian is sensitive to changes in global geometric properties like the number of connected components \cite{chung1997spectral, ng2001spectral, schiebinger2015geometry}. On the other hand, as discussed in Subsection \ref{subsec:discrete-to-continuum}, our method includes a relaxation of the traveling salesman problem, and so with proper hyperparameter selection one should expect that the performance of our method should be similar to or better than that of TSP. Indeed this can be seen in both of our experiments. In real world applications, where lossy data and complex geometry is common, this result suggests principal curves may produce better seriation or be a viable option when TSP is not feasible at scale.

\begin{rem}
We note that existing pseudotime inference methods do exist which assign orderings to datasets with branching in the feature domain \cite{saelens2019comparison}, but the branching structure leads to separate, incomparable pseudotimes along each branch. If we instead view the data in the space of probability measures over features, as we do here, we can instead infer a single curve (and ordering) in the Wasserstein space for the whole dataset. For branching data, our approach therefore has the benefit of assigning an ordering which is comparable across branches.
\end{rem}

\end{appendix}
\end{document}